\documentclass[11pt,a4paper,final]{article}
\pdfoutput=1
\usepackage{graphicx}
\graphicspath{{../Figures/}}
\usepackage{subfigure}

\usepackage{amssymb,latexsym, braket}

\usepackage{booktabs}
\usepackage{multirow}
\usepackage{setspace}

\usepackage{amsmath}

\usepackage{mathrsfs,amsfonts}

\usepackage{amsthm,amsxtra}

\usepackage[final]{showkeys}

\usepackage{stmaryrd}

\usepackage{algorithm}
\usepackage{algorithmicx}
\usepackage{algpseudocode}
\usepackage{mathtools}
\newif\ifPDF
\ifx\pdfoutput\undefined
\PDFfalse
\else
\ifnum\pdfoutput > 0
\PDFtrue
\else
\PDFfalse
\fi
\fi

\ifPDF
\usepackage{pdftricks}
\begin{psinputs}
	\usepackage{pstricks}
	\usepackage{pstcol}
	\usepackage{pst-plot}
	\usepackage{pst-tree}
	\usepackage{pst-eps}
	\usepackage{multido}
	\usepackage{pst-node}
	\usepackage{pst-eps}
\end{psinputs}
\else
\usepackage{pstricks}
\fi

\ifPDF
\usepackage[debug,pdftex,colorlinks=true, 
linkcolor=blue, bookmarksopen=false,
plainpages=false,pdfpagelabels]{hyperref}
\else
\usepackage[dvips]{hyperref}
\fi

\usepackage[openbib]{currvita}

\usepackage{fancyhdr}

\usepackage[title]{appendix}

\newtheorem{theorem}{Theorem}[section]
\newtheorem{lemma}[theorem]{Lemma}
\newtheorem{definition}[theorem]{Definition}

\newtheorem{remark}[theorem]{Remark} 
\newtheorem{corollary}[theorem]{Corollary}

\newcommand{\sgn}{\operatorname{sgn}}

\newcommand{\supp}{\operatorname{supp}}

\newcommand{\dint}{\displaystyle\int}
\newcommand{\eps}{\varepsilon}

\DeclareMathOperator*{\argmin}{\arg\min}

 \newcommand{\bbE}{\mathbb E}
 \newcommand{\bbN}{\mathbb N}
\newcommand{\bbP}{\mathbb P} 
\newcommand{\bbR}{\mathbb R} 
\newcommand{\bbZ}{\mathbb Z} 

\newcommand{\bzero}{{\mathbf 0}}

\newcommand{\bxi}{\boldsymbol \xi}
\newcommand{\bzeta}{\boldsymbol \zeta}

\newcommand{\be}{\mathbf e}

 \newcommand{\bn}{\mathbf n}

\newcommand{\bu}{\mathbf u} \newcommand{\bv}{\mathbf v} 
\newcommand{\bw}{\mathbf w} \newcommand{\bx}{\mathbf x} 
\newcommand{\by}{\mathbf y} \newcommand{\bz}{\mathbf z}

\newcommand{\cA}{\mathcal A} 
  
\newcommand{\cE}{\mathcal E} 
 
\newcommand{\cI}{\mathcal I} \newcommand{\cJ}{\mathcal J}
 \newcommand{\cL}{\mathcal L}
 
\newcommand{\cO}{\mathcal O} \newcommand{\cP}{\mathcal P} 
\newcommand{\cQ}{\mathcal Q} \newcommand{\cR}{\mathcal R}
\newcommand{\cS}{\mathcal S} 
 \newcommand{\cV}{\mathcal V}
\newcommand{\cW}{\mathcal W}

\usepackage[normalem]{ulem} 

\newenvironment{keywords}
{\noindent{\bf Key words.}\small}{\par\vspace{1ex}}

\makeatletter
\newcommand{\chapterauthor}[1]{%
	{\parindent0pt\vspace*{-25pt}%
		\linespread{1.1}\large\scshape#1%
		\par\nobreak\vspace*{35pt}}
	\@afterheading%
}
\makeatother

\title{Error Analysis for the Implicit Boundary Integral Method}
\author{
Yimin Zhong\thanks{Department of Mathematics and Statistics, Auburn University, Auburn AL, USA (\href{yimin.zhong@auburn.edu}{yimin.zhong@auburn.edu})}
\and
Kui Ren\thanks{Department of Applied Physics and Applied Mathematics, Columbia University, New York NY, USA (\href{kr2002@columbia.edu}{kr2002@columbia.edu})}
\and 
Olof Runborg\thanks{Department of Mathematics, KTH Royal Institute of Technology, Stockholm, Sweden (\href{olofr@kth.se}{olofr@kth.se})}
\and  
Richard Tsai\thanks{Department of Mathematics and Oden Institute for Computational Engineering and Sciences, The University of Texas at Austin, Austin TX, USA (\href{ytsai@math.utexas.edu}{ytsai@math.utexas.edu})
}}
\date{}
\begin{document}
\maketitle

\begin{abstract}
The implicit boundary integral method (IBIM) provides a framework to construct quadrature rules on regular lattices for integrals over irregular domain boundaries. This work provides a systematic error analysis for IBIMs on uniform Cartesian grids for boundaries with 
different degree of regularities. We first show that the quadrature error gains an addition order of $\frac{d-1}{2}$ from the curvature for a strongly convex smooth boundary due to the ``randomness'' in the signed distances. This gain is discounted for degenerated convex surfaces. We then extend the error estimate to general boundaries under some special circumstances, including how quadrature error depends on the boundary's local geometry relative to the underlying grid.  Bounds on the variance of the quadrature error under random shifts and rotations of the lattices are also derived.
\end{abstract}


\section{Introduction}

Let $\Omega\subset \bbR^d$ be a general simply connected domain with $C^{\infty}$ boundary $\Gamma$. Without loss of generality, we assume that the origin $\bzero\in \Omega$ and define the $\eps$-tube $T_{\eps}$: 
\begin{equation*}
    T_{\eps} := \{x\in\Omega\mid 0 \le \mathrm{dist}(\bx, \Gamma) \le \eps\}, 
\end{equation*}
where $\mathrm{dist}(\bx, \Gamma) = \min_{\by\in\Gamma}|\by - \bx|$ is the \emph{unsigned} distance function to the boundary. We make further assumption that $\forall \bx\in T_{\eps}$, the projection $P_{\Gamma}(\bx): T_{\eps}\mapsto \Gamma$:
\begin{equation*}
    P_{\Gamma}(\bx) = \argmin_{\by\in \Gamma} |\by - \bx|
\end{equation*}
is well-defined and the \emph{signed} distance function $d_{\Gamma}$ is defined by 
\begin{equation*}
    d_{\Gamma} (\bx) := \bn(P_{\Gamma}(\bx)) \cdot (\bx - P_{\Gamma}(\bx)),
\end{equation*}
where $\bn$ denotes the outward unit normal vector on $\Gamma$. 
Let $\sigma$ be the boundary Lebesgue measure defined on $\Gamma$. We can then use the co-area formula to rewrite the following boundary integral for $f(\bx)\in C^{\infty}(\Gamma)$
\begin{equation*}
    \cI(f):= \int_{\Gamma} f(\bx) d\sigma(\bx)
\end{equation*}
into a volumetric integral in $T_{\eps}$:
\begin{equation}\label{EQ: TUBE}
    \cI(f) = \int_{T_{\eps}} f(P_{\Gamma}(\bx)) \theta_{\eps}(d_{\Gamma}(\bx)) J_{\eps}(\bx, d_{\Gamma} (\bx)) d\bx  .
\end{equation}
The weight function $\theta_{\eps}(s) := \eps^{-1}\theta(\eps^{-1} s)$ is a regularized 1D Dirac function such that 
\begin{equation*}
   \supp\theta  = [-1, 1]\; \text{ and }\; \int_{-1}^{1} \theta(s) ds = 1.
\end{equation*}
The function $J_{\eps}(\cdot, \eta)$ is the Jacobian for the projection $P_{\Gamma}$ on the level set surface $\{d_{\Gamma} = \eta\}$. For $d=2$, $J_{\eps}(\bx, \eta) = 1 - \eta \kappa(\bx)$ with $\kappa$ the signed curvature of the level curve $\Gamma_{\eta} := \{ \bx\mid d_{\Gamma}(\bx) = \eta \}$.  For $d=3$, $J_{\eps}(\bx, \eta) = 1 - 2 \eta H(\bx) + \eta^2 G(\bx)$ with $H$ and $G$ denoting the mean curvature and Gaussian curvature of the level surface $\Gamma_{\eta}$, respectively.

For the convenience of analysis, we introduce the function class $\cW_q$, $q\in\bbN$ that 
\begin{equation}\nonumber
    \cW_q := \Set{f\in C^{q-1}(\bbR) \ | \begin{array}{l}
              \supp f = [-\eps, \eps] \;\text{ and }\;\dint_{-\eps}^{\eps} f(s)ds = 1,\\\\
         f^{(q)} \text{ is a piecewise $C^1$ function on $\bbR$.}  
    \end{array}}, \; \text{ if }q\ge 1
\end{equation}
and 
\begin{equation}\nonumber
    \cW_0 := \Set{f\in L^{\infty}(\bbR) \ | \begin{array}{l}
              \supp f = [-\eps, \eps] \;\text{ and }\;\dint_{-\eps}^{\eps} f(s)ds = 1,\\\\
         f \text{ is a piecewise $C^1$ function on $\bbR$.}  
    \end{array}}.
\end{equation}
Here $f^{(q)}$ denotes the $q$-th derivative of $f$. The index $q$ denotes the order of regularity of the weight function.
Usual choices such as
\begin{equation}\label{EQ: WEIGHT}
    \theta_{\eps}^{\cos}(s) = \frac{1}{2\eps}\left(1 + \cos(\eps^{-1}\pi s)\right)\in \cW_{2},\quad \theta_{\eps}^{\Delta}(s) = \frac{1}{2\eps}(1 - \eps^{-1}|s|)\in \cW_1.
\end{equation}
\begin{figure}[!htb]
    \centering
    \includegraphics[scale=0.4]{ 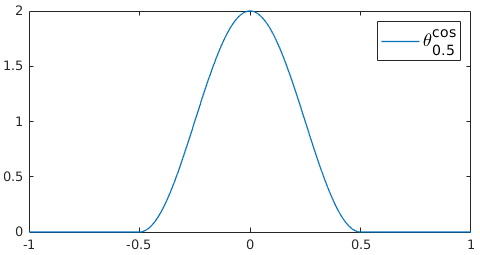}
    \includegraphics[scale=0.4]{ 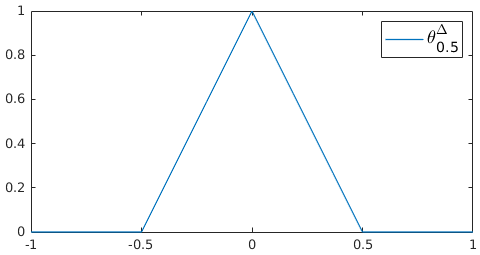}
    \caption{Shape of the the weight functions $\theta_{0.5}^{\cos}$ and $\theta_{0.5}^{\Delta}$ defined in~\eqref{EQ: WEIGHT}.}
    \label{fig: weight func}
\end{figure}
have been adopted in the literature; see, for instance, references~\cite{zhong2018implicit,kublik2013implicit,chen2017implicit,engquist2005discretization,izzo2022corrected,izzo2023high}. These functions are visualized in Figure~\ref{fig: weight func}. More choices of weight functions with higher regularities can be found in~\cite{tornberg2004numerical}.

The implicit boundary integral method (IBIM)~\cite{chen2017implicit,kublik2013implicit,kublik2016integration,kublik2018extrapolative} uses the summation over the regular lattice $(h\bbZ)^d$ to approximate the integral~\eqref{EQ: TUBE}:
\begin{equation}\label{EQ:IBIM}
\begin{aligned}
    \cI_h (f) &= h^d \sum_{\bn\in (h\bbZ)^d\cap T_{\eps}} f(P_{\Gamma}(\bn)) \theta_{\eps}(d_{\Gamma}(\bn)) J_{\eps}(\bn, d_{\Gamma} (\bn))  .
\end{aligned}
\end{equation}
We use $\Theta$ and $\cO$ to denote the standard Landau notations for big-Theta and big-O, respectively. When the width parameter $\eps = \Theta(h^{\alpha})$, $\alpha \in [0, 1)$,  the theoretical quadrature error previously derived is at the order of $\cO(h^{(q+1)(1 -\alpha)})$ for a general domain with smooth boundary~\cite{engquist2005discretization}. However, the total number of lattice points is at the order $\cO(\eps h^{-d})$, and the width parameter is often chosen as $\eps = \Theta(h)$. This turns the classical quadrature error into $\cO(1)$ in such a circumstance, a clearly undesired abrupt loss of accuracy.


The objective of this paper is to perform a more careful investigation of the above issue of vanishing accuracy. We perform a systematic error analysis using the Poisson Summation Formula for boundaries of different orders of regularity. For strongly convex smooth boundaries, we show that the quadrature error gains an addition order of $\frac{d-1}{2}$ from the curvature of the boundaries due to the ``randomness" in the signed distance functions for such boundaries. This additional gain avoids the accuracy catastrophe we see above.

The remainder of this paper is organized as follows. We first discuss in Section~\ref{SEC: 2} the error estimate of the implicit boundary integral method for strongly convex boundaries and the error statistics. We then investigate the issues and possible solutions for a general smooth convex boundary for 2D in Section~\ref{SEC: 3}. Generalization of the estimate of implicit boundary integral to general open curves in 2D is then presented in Section~\ref{SEC: 4}. Concluding remarks are offered in Section~\ref{SEC: 5}. 

\section{Error analysis for smooth convex boundaries in \texorpdfstring{$\mathbb{R}^d$}{}}
\label{SEC: 2}

We start with the simplest case when the boundary $\Gamma\in C^{\infty}$ is strongly convex; that is, the Gaussian curvature is bounded away from zero uniformly. To setup the notation, we denote by $$\cQ(\bx) :=  f(P_{\Gamma}(\bx)) \theta_{\eps}(d_{\Gamma}(\bx)) J_{\eps}(\bx, d_{\Gamma} (\bx))$$ the integrand in~\eqref{EQ: TUBE}. We also denote the mollifier $\psi_{\delta}(\bx) = \delta^{-d}\psi(\delta^{-1}\bx)$, where $\psi(\bu) = \psi(|\bu|)\in C_0^{\infty}([-1,1])$ is the bump function that 
\begin{equation*}
    \int_{-1}^1 \psi(s) ds = 1\,.
\end{equation*}

\subsection{Poisson Summation Formula}

We use the standard notation $\cS(\bbR^d)$ for the Schwartz function space on $\bbR^d$. For $f\in \cS(\bbR^d)$, the Fourier transform of $f$ is defined by 
\begin{equation*}
    \widehat{f}(\bzeta) = \int_{\bbR^d} f(\bx) e^{-2\pi i \bx\cdot \bzeta} d\bx.
\end{equation*}
The following Poisson Summation Formula is a standard tool whose proof can be found in~\cite{stein1993harmonic}.
\begin{lemma}[Poisson Summation Formula~\cite{stein1993harmonic}]
If $f$ satisfies that $|f(\bu)| + |\widehat{f}(\bu)| = \cO((1 + |\bu|)^{-d - \mu})$ for all $\bu\in\bbR^d$ and certain $\mu>0$, then 
\begin{equation}\label{EQ: PSF}
    \sum_{\bn\in \bbZ^d} f(\bn) =\sum_{\bn\in\bbZ^d} \widehat{f}(\bn).
\end{equation}
In particular, the above formula holds for $f\in \cS(\bbR^d)$.
\end{lemma}

Let $h\in\bbR_{+}$ be the lattice resolution in the summation~\eqref{EQ:IBIM}. If we apply the Poisson Summation Formula to $\cQ_{h, \delta}(\bx):= \cQ(h {\bx})\ast \psi_{\delta}\in \cS(\bbR^d)$, we have
\begin{equation}\label{EQ: POISSON}
    \sum_{\bn\in\bbZ^d} \cQ_{h, \delta}(\bn) = \sum_{\bn\in\bbZ^d} \widehat{\cQ_{h, \delta}}(\bn) = h^{-d} \sum_{\bn\in\bbZ^d} \widehat{\cQ}(h^{-1}{\bn})\widehat{\psi}(\delta\bn)\,.
\end{equation}
For the term $\bn = \bzero$ in the last summation, it equals to
\begin{equation}
    h^{-d}  \widehat{\cQ}(\bzero)\widehat{\psi}(\bzero) = h^{-d} \int_{\bbR^d} \cQ(\bx) d\bx \int_{\bbR^d} \psi(\bx) d\bx =  h^{-d} \int_{\bbR^d} \cQ(\bx) d\bx.
\end{equation}
which produces the desired integral in~\eqref{EQ: TUBE}. The objective of the rest of this section is to provide an estimate for the remaining terms that $\bn\neq \bzero$.

 \subsection{Main results on strongly convex boundaries}\label{SEC: MAIN}
 
 We will need the following two lemmas whose proofs are given in Appendices~\ref{PRF: STATION} and ~\ref{PRF: ESTIMATE}, respectively.
 

 
\begin{lemma}\label{LEM: STATIONARY PHASE}
Let $\Gamma\in C^{\infty}$ be closed and strongly convex. Then there exists a constant $C > 0$ such that when $\eps \le C$, $|\widehat{\cQ}(\bzeta)| = \cO(\eps^{-(q+1)}|\bzeta|^{-(d+1)/2 - q})$.
\end{lemma}
\begin{lemma}\label{LEM: ESTIMATE}
Let $\Gamma\in C^{\infty}$ be closed and strongly convex. Then the remaining terms in~\eqref{EQ: POISSON} are bounded by
    \begin{equation}
        \sum_{\bn\in\bbZ^d, \bn \neq \bzero} \widehat{\cQ}(h^{-1}{\bn})\widehat{\psi}(\delta\bn) = \begin{cases}
            \cO(\eps^{-(q+1)} h^{\frac{d+1}{2}+q} \delta^{-\frac{d-1}{2} +q}) \quad &\frac{d-1}{2} > q, \\ 
            \cO(\eps^{-(q+1)} h^{\frac{d+1}{2} + q} |\log\delta|)\quad &\frac{d-1}{2} = q, \\ 
             \cO(\eps^{-(q+1)} h^{\frac{d+1}{2} + q})\quad &\frac{d-1}{2} < q.
        \end{cases}
    \end{equation}
\end{lemma}
Lemma~\ref{LEM: ESTIMATE} divides the regularity order $q$ into three classes: sub-critical regularity ($q < \frac{d-1}{2}$), critical regularity ($q = \frac{d-1}{2}$), and super-critical regularity ($q > \frac{d-1}{2}$). Base d on this result, we prove, in the following subsections, the estimates of the error bound for implicit boundary integral based on different regularity regimes of the weight function $\theta_{\eps}$. We summarize the resulting estimates in Tab~\ref{tab: error estimates}. 
\begin{table}[!htb]
    \centering
    \begin{tabular}{@{}c l@{}}
    \toprule
        Regularity regime & Error bound of 
        $|\cI f - \cI_h f|$ \\
        \midrule
        $\min(2, \frac{d-1}{2}) > q \ge 1$  &  $\cO(h^{\frac{2d - \alpha(d+1)}{d+1-2q}})$ \\
        \midrule 
         $\frac{d-1}{2} > q \ge 2 $   & $\cO(h^{\frac{4d - 2\alpha(d+1)}{d+3-2q}})$\\
         \midrule  
         $\frac{d-1}{2} = q$ & $\cO(h^{\frac{d-1}{2}+(q+1)(1-\alpha)}|\log h|)$ \\
         \midrule 
        $\frac{d-1}{2} < q$ & $\cO(h^{\frac{d-1}{2}+(q+1)(1-\alpha)})$ \\
        \bottomrule
    \end{tabular}
    \caption{Error estimates for implicit boundary integral method with $\eps=\Theta(h^{\alpha})$, $\alpha\in[0,1]$ on strongly convex surfaces.}
    \label{tab: error estimates}
\end{table}

\subsubsection{{Sub-critical regularity} } 
We consider two separate cases.
\begin{theorem}\label{THM: LOW THM 1}
Under the assumption that $\Gamma\in C^{\infty}$ is closed and strongly convex, if $\eps = \Theta(h^{\alpha})$ and $\min(2,\frac{d-1}{2}) > q\ge 1$, then $|\cI(f) - \cI_h(f)| =   \cO(h^{\frac{2d-\alpha(d+1)}{d+1-2q}})$.
\end{theorem}
\begin{proof}
    First, by Lemma~\ref{LEM: ESTIMATE}, we have, 
    \begin{equation}\label{EQ: EST LOW REG}
        \left|\cI(f) - h^d \sum_{\bn\in\bbZ^d} \cQ_{h, \delta}(\bn) \right| =  \cO(\eps^{-q-1} h^{\frac{d+1}{2}+q} \delta^{-\frac{d-1}{2}+q})\,.
    \end{equation}
   We also estimate the bound for
    \begin{equation}\label{EQ: DIFF 2}
    \begin{aligned}
        \left|\sum_{\bn\in\bbZ^d} \cQ(h\bn) - \cQ_{h,\delta}(\bn)\right| &=  \left|  \sum_{\bn\in\bbZ^d} \int_{\bbR^d}\left( \cQ(h\bn) - \cQ(h(\bn - \by))\right) \psi_{\delta}(\by) d\by \right|\,.
    \end{aligned}
    \end{equation}
The summation has at most $\cO(\eps h^{-d})$ terms, and we also have the trivial estimate
    \begin{equation}\label{EQ: LIP}
        \cQ(h(\bn - \by)) - \cQ(h\bn) = \cO(L h\delta)\,,
    \end{equation}
        where $L = \cO(\eps^{-2})$ is the Lipschitz constant of $\cQ$. This leads to 
    \begin{equation}\label{EQ: EST 1st ORDER}
         \left|\sum_{\bn\in\bbZ^d} \cQ(h\bn) - \cQ_{h,\delta}(\bn)\right| = 
        \cO(\eps^{-1} h^{-d+1} \delta)\,.
    \end{equation}
    Combining the estimates~\eqref{EQ: EST LOW REG} and~\eqref{EQ: EST 1st ORDER} gives us
    \begin{equation}\label{EQ: TOTAL EST lOW REG 1st ORDER}
        \left|\cI(f) - \sum_{\bn\in\bbZ^d} \cQ(h\bn)\right| =
        \cO\left(\eps^{-q-1} h^{\frac{d+1}{2}+q} \delta^{-\frac{d-1}{2}+q} + \eps^{-1}h \delta\right)\,.
    \end{equation}
    Let $\eps = \Theta(h^{\alpha})$ and balance the orders between $h$ and $\delta$ in~\eqref{EQ: TOTAL EST lOW REG 1st ORDER}, we conclude that the balance is attained at $\delta = \cO(h^{\frac{d-1-2q(\alpha-1)}{d+1-2q}})$, and 
    \begin{equation}
         \left|\cI(g) - h^d \sum_{\bn\in\bbZ^d} \cQ(h\bn)\right|  =              \cO(h^{\frac{2d-\alpha(d+1)}{d+1-2q}})\,.
    \end{equation}
    This completes the proof.
\end{proof}

In the second case, we have the following result.
\begin{theorem}\label{THM: LOW THM 2}
Under the assumption that $\Gamma\in C^{\infty}$ is closed and strongly convex, if $\eps = \Theta(h^{\alpha})$ and $\frac{d-1}{2} > q\ge 2$, then $|\cI(f) - \cI_h(f)| = \cO(h^{\frac{4d- 2\alpha(d+1)}{d+3-2q}})$.
\end{theorem}
\begin{proof}
        Similar to the previous case, we use Lemma~\ref{LEM: ESTIMATE} to conclude that
        \begin{equation}
        \left|\cI(f) - h^d \sum_{\bn\in\bbZ^d} \cQ_{h, \delta}(\bn) \right| =   \cO(\eps^{-q-1} h^{\frac{d+1}{2}+q} \delta^{-\frac{d-1}{2}+q})\,.
    \end{equation}
    Since $q\ge 2$, $\cQ\in C^{1,1}(\bbR^d)$, we can take Taylor expansion
    \begin{equation}\label{EQ: TAYLOR EXP}
        \cQ(h(\bn - \by)) - \cQ(h\bn) = -h\by \cdot \nabla \cQ(h\bn) + \cO(K h^2 \delta^2  )
     \end{equation}
    where $K = \cO(\eps^{-3})$ is the Lipschitz constant of $\nabla \cQ$. Using the spherical symmetry of $\psi_{\delta}$, the first term on the right-hand side of~\eqref{EQ: TAYLOR EXP} will be canceled. We thus have
    \begin{equation}\label{EQ: EST 2nd ORDER}
         \left|\sum_{\bn\in\bbZ^d} \cQ(h\bn) - \cQ_{h,\delta}(\bn)\right| = 
       \cO(\eps^{-2} h^{-d+2} \delta^2)\,,
    \end{equation}
    which leads to
    \begin{equation}\label{EQ: TOTAL EST lOW REG 2nd ORDER}
        \left|\cI(f) - \sum_{\bn\in\bbZ^d} \cQ(h\bn)\right| =
        \cO\left(\eps^{-q-1} h^{\frac{d+1}{2}+q} \delta^{-\frac{d-1}{2}+q} + \eps^{-2} h^2 \delta^2\right)\,.
    \end{equation}
    The bounds attain balance when $\delta = \cO(h^{\frac{d-1 + (\alpha-1)(2-2q)}{d+3-2q}})$ at which point we have
     $$\left|\cI(f) - \sum_{\bn\in\bbZ^d} \cQ(h\bn)\right| = \cO(h^{\frac{4d- 2\alpha(d+1)}{d+3-2q}})\,,$$
     which is the desired result.
\end{proof}

\subsubsection{Critical regularity}
The same calculations can be performed in the case of critical regularity. We have the following result.
\begin{corollary}\label{COR: CRI REG}
    Under the assumption that $\Gamma\in C^{\infty}$ is closed and strongly convex, if $\eps = \Theta(h^{\alpha})$ and $\frac{d-1}{2} = q$, then $|\cI(f) - \cI_h(f)| = \cO(h^{\frac{d-1}{2}+(q+1)(1-\alpha)}|\log h|)$.
\end{corollary}
\begin{proof}
   Using the Lemma~\ref{LEM: ESTIMATE} for the critical value $q = \frac{d-1}{2}$, we get
    \begin{equation*}
        \left|\cI(f) - h^d \sum_{\bn\in\bbZ^d} \cQ_{h, \delta}(\bn) \right| =  \cO(\eps^{-q-1} h^{\frac{d+1}{2}+q} |\log \delta|)\,.
    \end{equation*}
    If $q < 2$, we reuse~\eqref{EQ: EST 1st ORDER}, and otherwise we reuse~\eqref{EQ: EST 2nd ORDER}, to get
    \begin{equation}\label{EQ: TOTAL EST CRI REG}
        \left|\cI(f) - h^d \sum_{\bn\in\bbZ^d} \cQ(h\bn)\right| = \begin{cases}
                \cO\left(\eps^{-q-1} h^{\frac{d+1}{2}+q} |\log\delta| + \eps^{-1} h \delta  \right)\quad &q < 2\,,\\
                \cO\left(\eps^{-q-1} h^{\frac{d+1}{2}+q} |\log\delta| + \eps^{-2} h^2 \delta^2\right) \quad &q\ge 2\,.
        \end{cases}
    \end{equation}
    Let $\eps =\Theta(h^{\alpha})$ and balance the orders between $h$ and $\delta$ in~\eqref{EQ: TOTAL EST CRI REG}, we find: 
    \begin{enumerate}
        \item [(1)] If $0\le q<2$, the balance is achieved when $h^{\frac{d-1}{2}+q-q\alpha}|\log\delta| \approx \delta$, which means the error is $\cO(h^{1-\alpha + \frac{d-1}{2}+q-q\alpha}|\log \delta|)$. Since $\frac{d-1}{2}+q-q\alpha > 0$, $\delta = \cO(h^{\frac{d-1}{2}+q-q\alpha-\tau})$ for any $\tau>0$. Hence $|\log\delta| = \cO(|\log h|)$.
        \item [(2)] If $q\ge 2$, the balance is achieved when $h^{(-q+1)\alpha} h^{\frac{d-3}{2}+q}|\log \delta| \approx \delta^2$ and the quadrature error is $\cO(h^{-(q+1)\alpha} h^{\frac{d+1}{2}+q}|\log\delta|)$. Because $\frac{d-3}{2}+q + (1-q)\alpha > 0$, we find that $\delta^2 = \cO(h^{\frac{d-3}{2}+q + (1-q)\alpha-\tau})$ for certain $\tau>0$ and it implies $|\log \delta| = \cO(|\log h|)$.
    \end{enumerate}
    The error bounds in both cases are the same; that is,
    \begin{equation}
         \left|\cI(g) - h^d \sum_{\bn\in\bbZ^d} \cQ(h\bn)\right|  =          \cO(h^{\frac{d-1}{2}+(q+1)(1-\alpha)}|\log h|)\,.
    \end{equation}
    The finishes the proof.
\end{proof}
\subsubsection{{Super-critical regularity }}
The bound for the super-critical regularity case follows in the same way.
\begin{corollary}
    Under the assumption that $\Gamma\in C^{\infty}$ is closed and strongly convex, if $\eps = \Theta(h^{\alpha})$ and $\frac{d-1}{2} < q$, then $|\cI(f) - \cI_h(f)| = \cO(h^{\frac{d-1}{2}+(q+1)(1-\alpha)})$.
\end{corollary}
\begin{proof}
    Using Lemma~\ref{LEM: ESTIMATE}, the estimate~\eqref{EQ: TOTAL EST CRI REG} now takes the form
       \begin{equation}\label{EQ: TOTAL EST HI REG}
        \left|\cI(f) - h^d \sum_{\bn\in\bbZ^d} \cQ(h\bn)\right| = \begin{cases}
                \cO\left(\eps^{-q-1} h^{\frac{d+1}{2}+q}  + \eps^{-1} h \delta  \right)\quad &q < 2\,,\\
                \cO\left(\eps^{-q-1} h^{\frac{d+1}{2}+q}  + \eps^{-2} h^2 \delta^2\right) \quad &q\ge 2\,.
        \end{cases}
    \end{equation}
    Therefore, the error estimate becomes 
    \begin{equation*}
         \left|\cI(g) - h^d \sum_{\bn\in\bbZ^d} \cQ(h\bn)\right|  =          \cO(h^{\frac{d-1}{2}+(q+1)(1-\alpha)})\,,
    \end{equation*}
    which is what we need to prove.
\end{proof}
The tube width usually takes $\alpha=1$ for practical applications of the implicit-boundary-integral method. If $\theta_{\eps}$ has sub-critical regularity $q < \frac{d-1}{2}$, we have the error exponents as $\frac{d-1}{d+1-2q}$ and $\frac{2(d-1)}{d+3-2q}$ in Table~\ref{tab: error estimates} for different regularity classes. These errors may still be far from optimal, and the sharper exponents need more sophisticated estimates for the remaining terms of the Poisson Summation Formula~\eqref{EQ: PSF}. The problem is closely related to the generalized cases of the Gauss Circle Problem (GCP)~\cite{berndt2018circle,heath1999lattice} that counts the number of lattice points inside the domain. This can be reviewed as a special case of implicit boundary integral by taking $\theta_{\eps} = \theta_{\eps}^C = \frac{1}{2\eps}\chi_{T_{\eps}}$ as the characteristic function on the tube and $f\equiv 1$ with an approximation of $J \approx 1$. In such case, a discontinuity of $\cQ$ exists across the boundary $\partial T_{\eps}$. It will involve additional boundary contributions in~\eqref{EQ: TOTAL EST lOW REG 1st ORDER} and~\eqref{EQ: TOTAL EST lOW REG 2nd ORDER}. We will provide an interesting example in Remark~\ref{REM: CONST} below. The sharp discrepancy estimate of GCP is still an open question. Recent developments in GCP and its variants can be found in~\cite{heath1999lattice,berndt2018circle,bourgain2017mean,huxley2003exponential,huxley2005exponential} and references therein.
\begin{remark}\label{REM: CONST}
One interesting case is when the weight function has a jump across the boundary. For instance, the case where $\theta^C_{\eps} = \frac{1}{2\eps}\chi_{T_{\eps}}\in \cW_0$ is the characteristic function on the tube. Due to the jump, we cannot directly reuse~\eqref{EQ: LIP} for~\eqref{EQ: DIFF 2}. We can estimate the summation in three categories
\begin{equation*}
      \left|\sum_{\bn\in\bbZ^d} \cQ(h\bn) - \cQ_{h,\delta}(\bn)\right| \le \cV_{in} + \cV_{bd} + \cV_{out}, 
\end{equation*}
where the summation $\cV_{in}$, $\cV_{bd}$, and $\cV_{out}$ are respectively
\begin{equation*}
\begin{aligned}
    \cV_{in} &=  \left|\sum_{\cS_{in}} \cQ(h\bn) - \cQ_{h,\delta}(\bn)\right| ,\quad \text{ where } \cS_{in} := \{\bn\mid  h \bn\in T_{\eps},\quad  B_{h\bn}(h\delta)\subset T_{\eps} \}\\
    \cV_{bd} &=  \left|\sum_{\cS_{bd}} \cQ(h\bn) - \cQ_{h,\delta}(\bn)\right| , \quad \text{ where } \cS_{bd} :=  \{\bn\mid  h \bn\in T_{\eps},\quad  B_{h\bn}(h\delta)\not\subset T_{\eps} \}\\
    \cV_{out} &= \left|\sum_{\cS_{out}} \cQ(h\bn) - \cQ_{h,\delta}(\bn)\right| ,\quad \text{ where } \cS_{out} := \{ \bn\mid h \bn\not\in T_{\eps}, \quad B_{h\bn}(h\delta)\cap T_{\eps} \neq\emptyset \}
\end{aligned}
\end{equation*}
We can estimate $\cV_{in}$ by $\cO(\eps h^{-d+2}\delta^2)$ by the second order Taylor expansion at $h\bn$. For both $\cV_{bd}$ and $\cV_{out}$, there are $\cO(\delta h^{-d+1})$ lattice points, and each summand is at order $\cO(\eps^{-1})$. Therefore, we obtain the estimate
\begin{equation}
   \left|\cI(f) - h^d\sum_{\bn\in\bbZ^d} \cQ(h\bn)\right| = \cO(\eps^{-1} h^{\frac{d+1}{2}} \delta^{-\frac{d-1}{2}} + h^{2}\delta^2 + \eps^{-1}\delta h)\,.
\end{equation}
The balance is attained at $\delta \sim \cO(h^{\frac{d-1}{d+1}})$ and error is $\cO(h^{\frac{2d-\alpha(d+1)}{d+1}})$ for $\eps =\Theta(h^{\alpha})$. In particular, when $\alpha=1$, we find that the error is bounded by $\cO(h^{\frac{d-1}{d+1}})$. As a comparison, in this case, a standard Monte Carlo integration for $\alpha=1$ gives a smaller error, bounded by $\cO(h^{\frac{d-1}{2}})$. 
\end{remark}

For fixed tube width of $\cO(1)$ ($\alpha = 0$), the error estimates are $\cO(h^{\frac{2d}{d-1}})$ for $\theta_{\eps}^{\Delta}$ ($q=1$ and $d>3$), which is almost second order convergence in high dimensions. It exceeds the theoretical estimate $\cO(h)$. For $\theta_{\eps}^{\cos}$ ($q=2$ and $d > 5$), the error becomes $\cO(h^{\frac{4d}{d-1}})$, which is almost 4th order when the dimension is high, exceeding the theoretical estimate $\cO(h^2)$. We gain these additional convergence rates from the strong convexity in high dimensions. However, for low dimensional surfaces of $d=2$ or $d=3$, once the regularity $q\ge 1 \ge \frac{d-1}{2}$, we have a generic error estimate $\cO(h^{\frac{d+1}{2}+q}|\log h|)$ or $\cO(h^{\frac{d+1}{2}+q})$.

\subsection{Variance of error under random translations}\label{SEC: STATS}

We now characterize the variance of the quadrature error caused by random shifts of the lattice, an important issue for practical applications~\cite{izzo2022corrected,izzo2023high, engquist2005discretization}. We denote by $\bzeta\in\bbR^d$ the random shift, and denote by $\cI_h(f;\bzeta)$ the implicit-boundary-integral quadrature with the random shift; that is,
$$\cI_h(f;\bxi) := h^d \sum_{\bn\in\bbZ^d} \cQ(h(\bn + \bxi))\,.$$
It is clear that the quadrature error is $h$-periodic, hence we only have to consider a random offset $\bxi\in[0, 1]^d$. First, we note that $\bbE_{\bxi} \cI_h(f;\bxi) = \cI(f)$, since  
\begin{equation*}
   \bbE_{\bxi} \cI_h(f;\bxi) =  h^d  
 \sum_{\bn\in\bbZ^d} \int_{[0, 1]^d} \cQ(h(\bn + \bxi)) d\bxi = \int_{\bbR^d} \cQ(\bx) d\bx = \cI(f).
\end{equation*}
We therefore define the variance of $\cI_h(f; \bzeta)$ as
\begin{equation}
    \textrm{Var}_{\bxi}[\cI_h(f;\bxi)]:= \int_{\bxi\in [0, 1]^d} \left|\cI(f) - h^d \sum_{\bn\in\bbZ^d} \cQ(h(\bn + \bxi))\right|^2 d\bxi\,.
\end{equation}
We are interested in providing an estimate for this variance.

The following theorem follows the idea of the classical result by Kendall~\cite{kendall1948number}. General $L^p$ estimates can be derived similarly; see, for instance, ~\cite{brandolini2015lp,huxley2014fourth}.
\begin{theorem}\label{THM: STAT 1}
Let $\Gamma\in C^{\infty}$ be closed and strongly convex. If $\eps=\Theta(h^{\alpha})$, then
    $ \mathrm{Var}_{\bxi}\left[\cI_h(f;\bxi)\right] = \cO(h^{d-1 + 2(q+1)(1-\alpha)})$.
\end{theorem}
\begin{proof}
Let $\cA(h, \bzeta):= h^d \sum_{\bn\in\bbZ^d} \cQ(h(\bn + \bzeta))$. It is then clear that $\cA(h,\bzeta)$ has a Fourier series expansion that converges in $L^2([0, 1]^d)$: 
\begin{equation*}
    \cA(h, \bzeta) \sim \sum_{\bxi\in \bbZ^d} a_{\bxi}(h) e^{2\pi i( \bzeta\cdot \bxi)}
\end{equation*}
where the Fourier coefficients are given as 
\begin{equation*}
\begin{aligned}
    a_{\bxi}(h) &= \int_{[0, 1]^d} \cA(h, \bzeta) e^{-2\pi i (\bzeta\cdot \bxi)} d\bzeta \\
    &=h^d \sum_{\bn\in\bbZ^d}  \int_{[0, 1]^d} \cQ(h(\bn + \bzeta)) e^{-2\pi i (\bxi\cdot \bzeta)} d\bzeta \\
    &=h^d  \sum_{\bn\in\bbZ^d}  \int_{[0, 1]^d} \cQ(h(\bn + \bzeta)) e^{-2\pi i (\bxi\cdot \bzeta)} d\bzeta \\
    &= \int_{\bbR^d} \cQ(\bx) e^{-2\pi i(h^{-1}\bxi \cdot \bx)} d\bx \\
    &= \widehat{\cQ}(h^{-1}\bxi)\,.
\end{aligned}
\end{equation*}
Notice that $a_{\bzero} = \cI(f)$. Therefore, we have
\begin{equation*}
    \mathrm{Var}_{\bxi}\left[\cI_h(f;\bxi)\right] = \sum_{\bxi\in\bbZ^d, |\bxi|\neq 0} |a_{\bxi}(h)|^2 = \sum_{\bxi\in\bbZ^d, |\bxi|\neq 0} |\widehat{\cQ}(h^{-1}\bxi)|^2\,. 
\end{equation*}
Using Lemma~\ref{LEM: STATIONARY PHASE}, with $q \ge 0$, we get 
\begin{equation*}
\begin{aligned}
    \mathrm{Var}_{\bxi}\left[\cI_h(f;\bxi)\right] &=\cO(\eps^{-2(q+1)} h^{d+1+2q}) \sum_{\bxi\in\bbZ^d, |\bxi|\neq 0} |\bxi|^{-(d+1) - 2q}\\ &=  \cO(h^{d-1+2(q+1)(1-\alpha)})\,.
\end{aligned}
\end{equation*}
The proof is complete.
\end{proof}

A probabilistic consequence of the above variance characterization is the following. If one randomly chooses a shift $\bxi\in\bbR^d$ in the implicit boundary integration, then by the Chebyshev inequality, there exists a constant $C$, depending on both $f$ and the surface $\Gamma$, that 
\begin{equation*}
    \bbP\Big(|\cI_{h}(f; \bxi) - \cI(f)| \ge k h^{\frac{d-1}{2}+(q+1)(1-\alpha)} \Big) < C k^{-2},\ \ k>0\,.
\end{equation*}
Hence, for weight functions with sub-critical regularity $q < \frac{d-1}{2}$, this implies that it is rare to achieve the possible worst bounds shown in Theorem~\ref{THM: LOW THM 1} and Theorem~\ref{THM: LOW THM 2}. 

\subsection{Numerical experiments}

In this section, we numerically verify (i) the quadrature error bounds shown in Tab~\ref{tab: error estimates} on circles in 2D and spheres in 3D; and (ii) the variance of quadrature error under random translations. The reference values of the boundary integrals are computed by \texttt{Mathematica} to machine precision. The \texttt{MATLAB} source code of all numerical experiments is hosted in the GitHub repository\footnote{\href{https://github.com/lowrank/ibim-error-analysis-experiments}{https://github.com/lowrank/ibim-error-analysis-experiments}}.

\subsubsection{2D circle}
The 2D experiments are performed on the circle $\Gamma$:
\begin{equation*}
    \frac{(x - x_0)^2}{r^2} + \frac{(y - y_0)^2}{r^2} = 1
\end{equation*}
with $r = \frac{3}{4}$ and $(x_0, y_0)$ a randomly sampled point. The test integrand function $f:\bbR^2 \mapsto \bbR$ is given as
\begin{equation*}
    f(x, y) = \cos(x^2 - y) \sin(y^2 - x^3)\,.
\end{equation*}
With the weight function $\theta_{\eps}^{\Delta}\in \cW_1$, the quadrature error has a upper bound estimate $\cO(h^{\frac{1}{2} + 2(1 - \alpha)})$. As we can see in Figure~\ref{fig: error decay}, for $\alpha=1$, $\alpha=\frac{1}{2}$ and $\alpha = 0$, the estimated error bounds $\cO(h^{\frac{1}{2}})$, $\cO(h^{\frac{3}{2}})$ and $\cO(h^{\frac{5}{2}})$ appear consistent with the numerical results. 
\begin{figure}[!htb]
    \centering
    \includegraphics[scale=0.34]{ 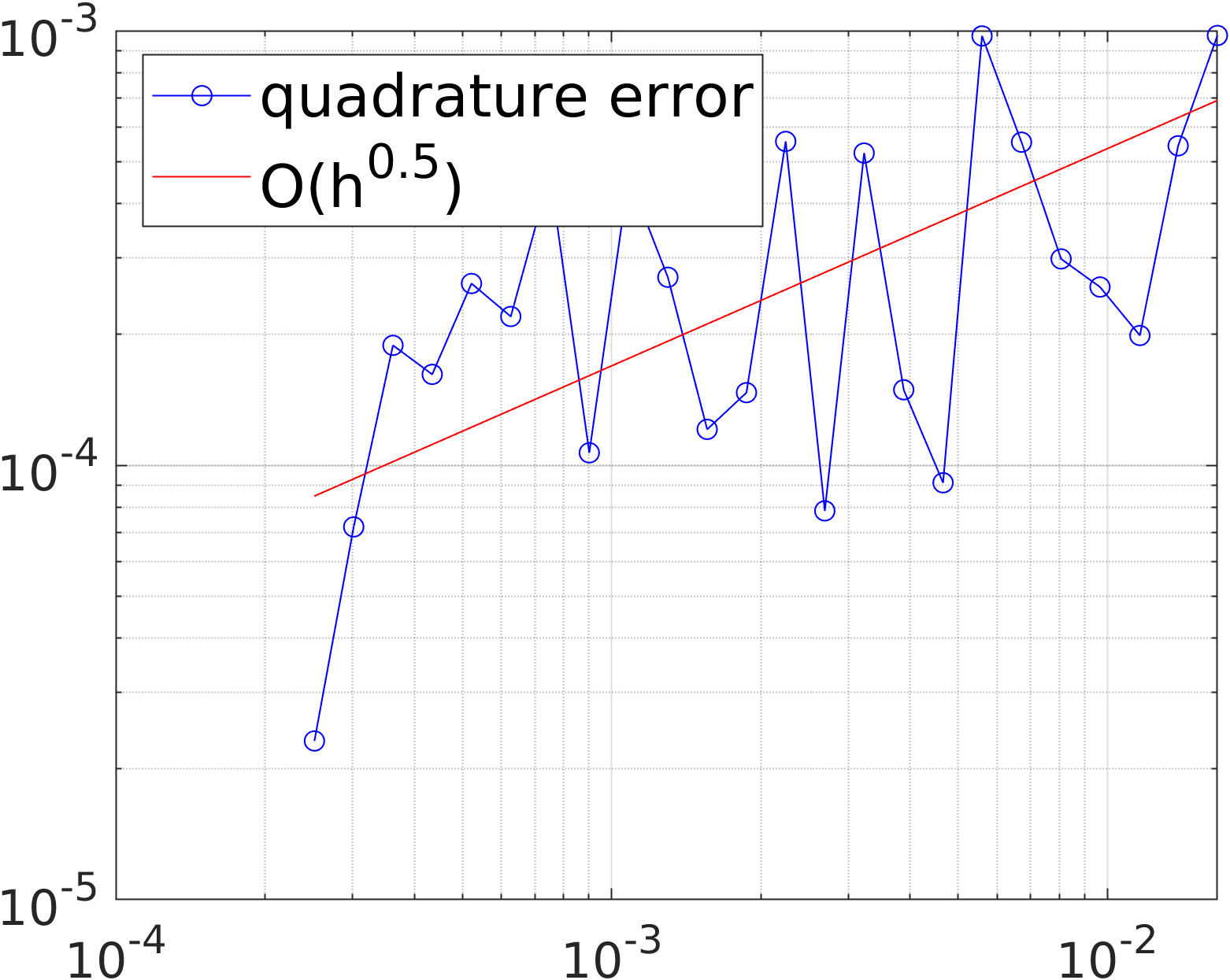}\qquad
    \includegraphics[scale=0.34]{ 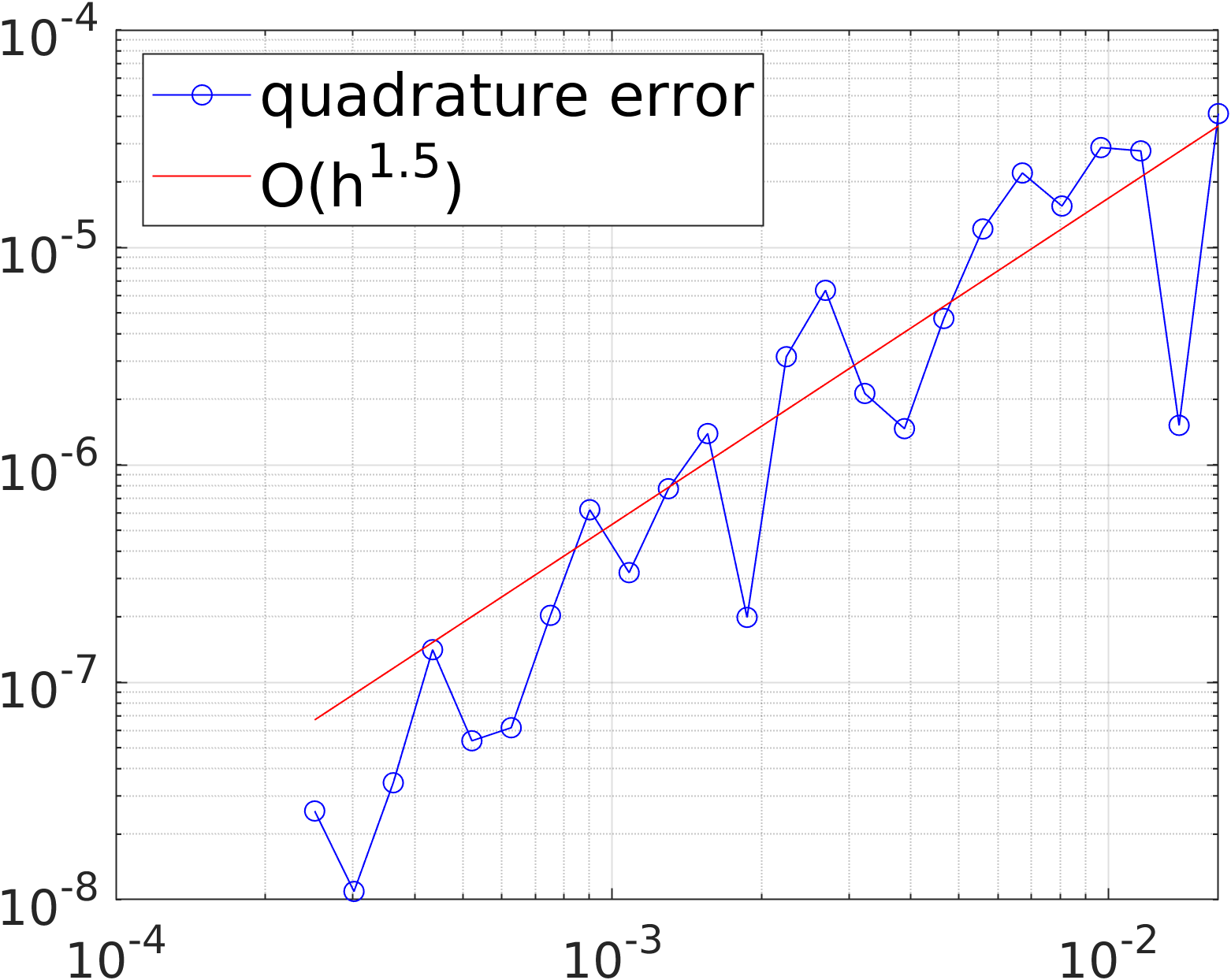}\qquad
    \includegraphics[scale=0.34]{ 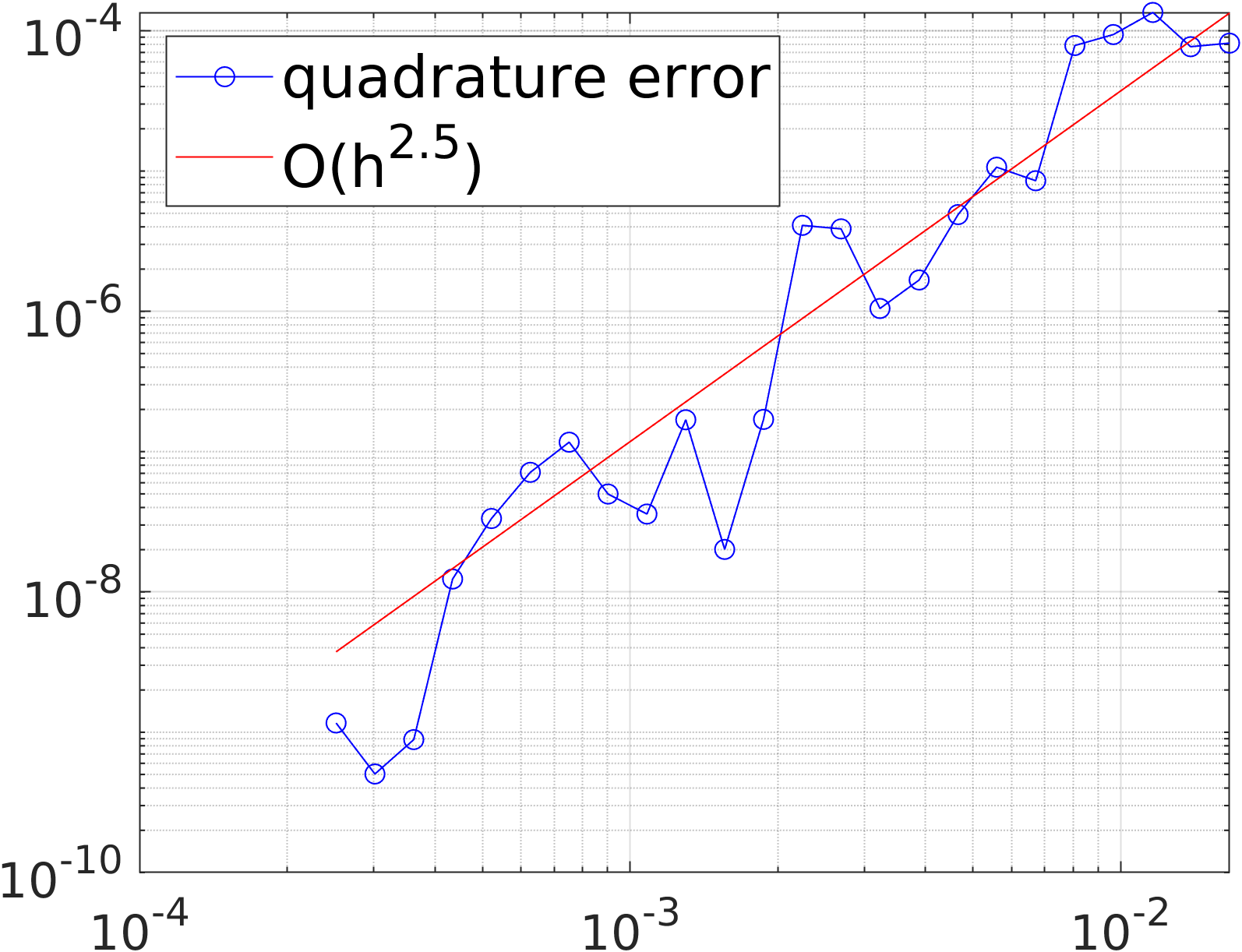}
    \caption{Quadrature error with weight function $\theta_{\eps}^{\Delta}$ with respect to grid size $h$. Left: tube width $\eps=2h$. Middle: tube width $\eps=2h^{\frac{1}{2}}$. Right: tube width $\eps=0.1$.}
    \label{fig: error decay}
\end{figure}
With the weight function $\theta_{\eps}^{\cos}\in \cW_2$, the quadrature error has an upper bound estimate $\cO(h^{\frac{1}{2} + 3(1 - \alpha)})$. In Figure~\ref{fig: error decay cos}, we find that for $\alpha=1$, $\alpha=\frac{1}{2}$ and $\alpha = 0$, the estimated error bounds $\cO(h^{\frac{1}{2}})$, $\cO(h^{2})$ and $\cO(h^{\frac{7}{2}})$ match the numerical results approximately.
\begin{figure}[!htb]
    \centering
    \includegraphics[scale=0.34]{ 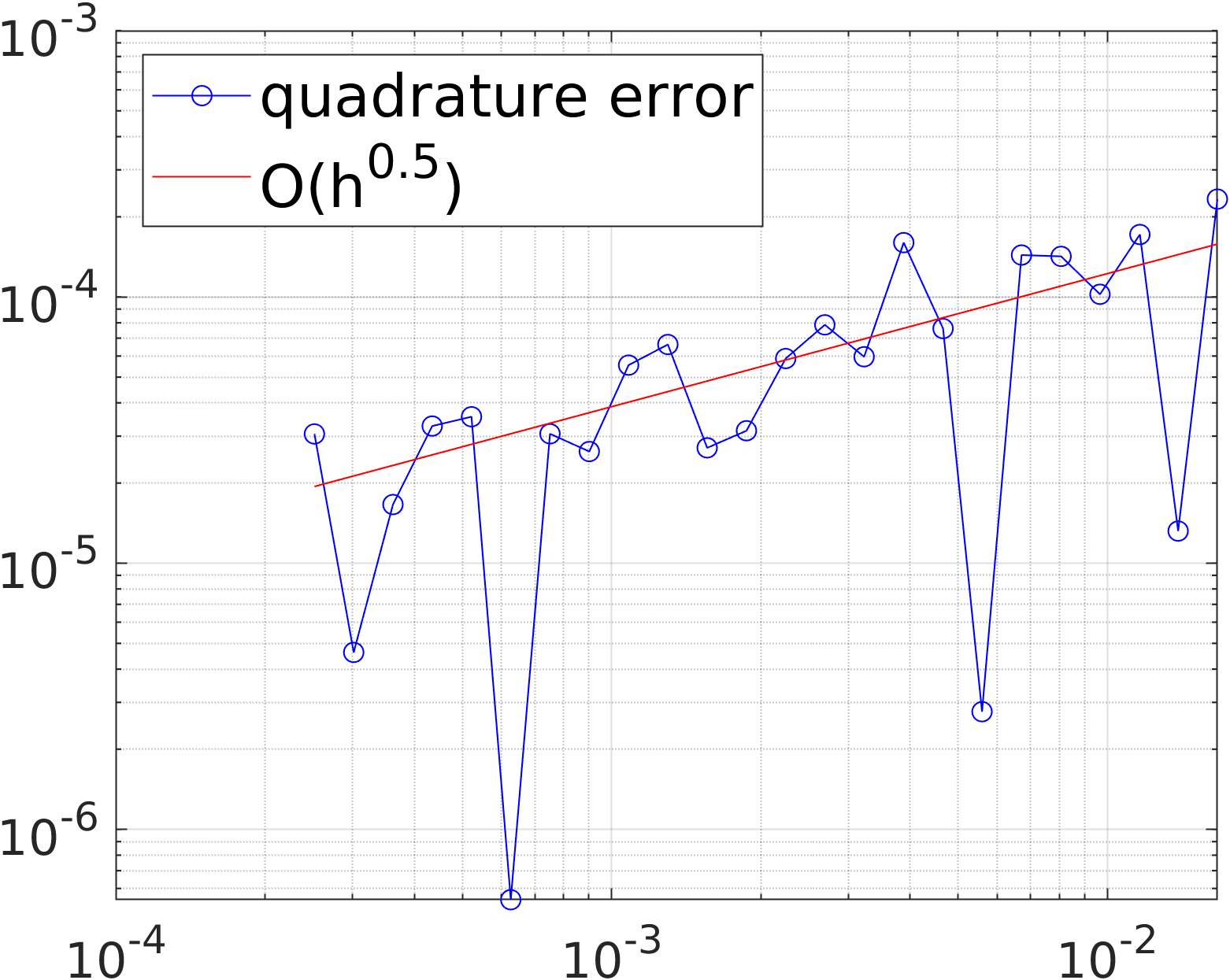}\qquad
    \includegraphics[scale=0.34]{ 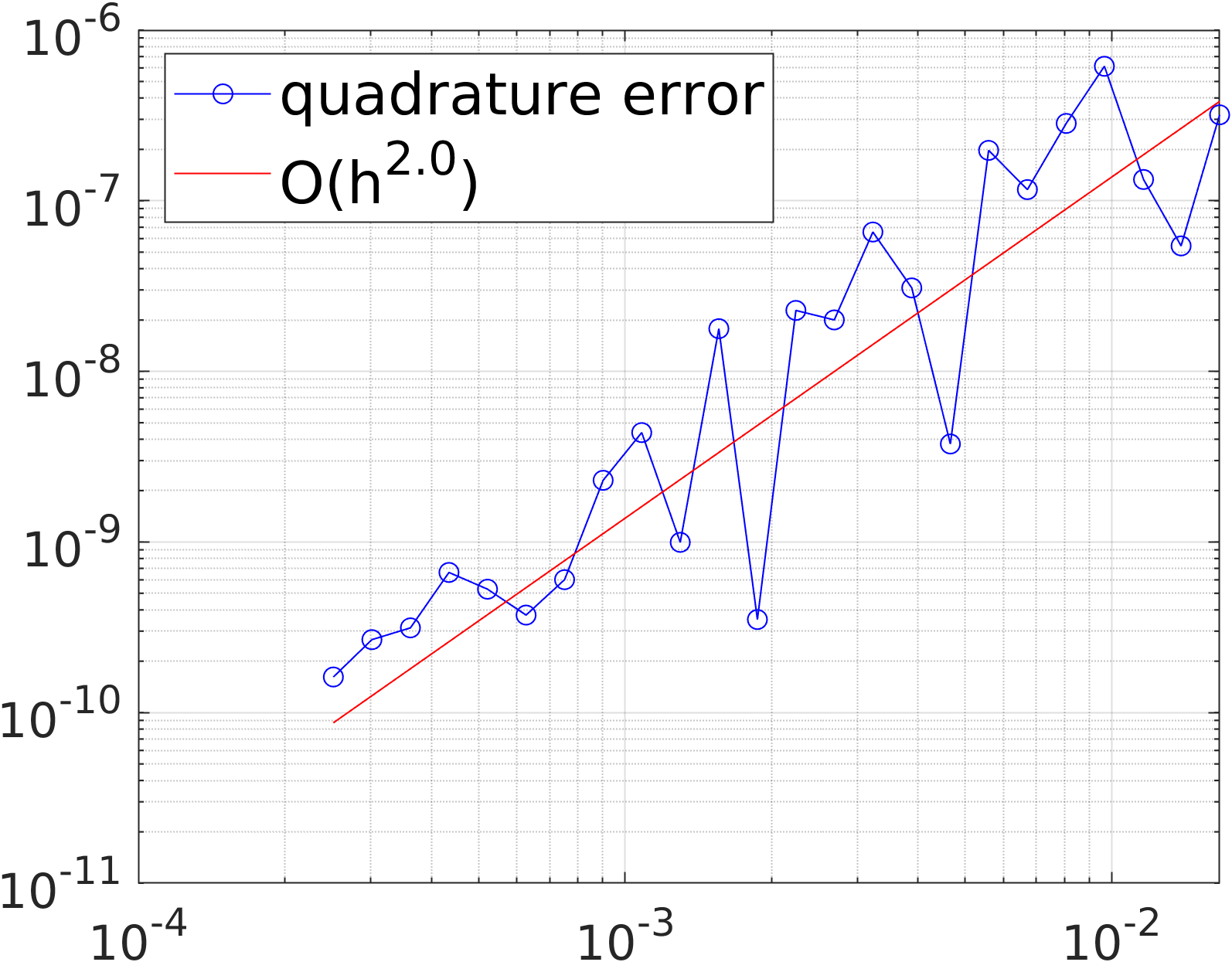}\qquad
    \includegraphics[scale=0.34]{ 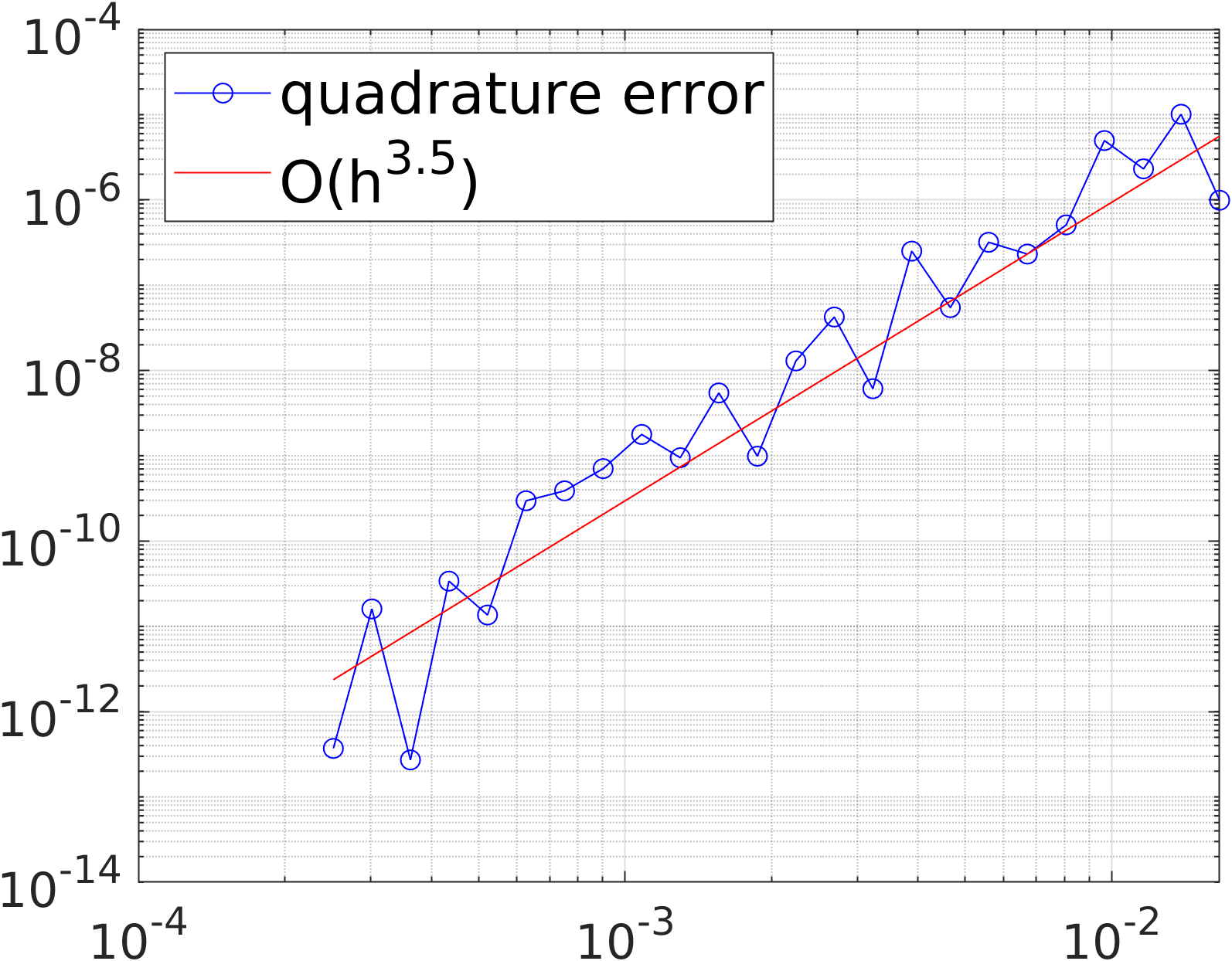}
    \caption{Quadrature error with weight function $\theta_{\eps}^{\cos}$ with respect to grid size $h$. Left: tube width $\eps=2h$. Middle: tube width $\eps=2h^{\frac{1}{2}}$. Right: tube width $\eps=0.1$.}
    \label{fig: error decay cos}
\end{figure}
The variances of error are plotted in Figure~\ref{fig: error var 1} and Figure~\ref{fig: error var 2} with 32 random shifts, consistent with the result of Theorem~\ref{THM: STAT 1}. 
\begin{figure}[!htb]
    \centering
     \includegraphics[scale=0.34]{ 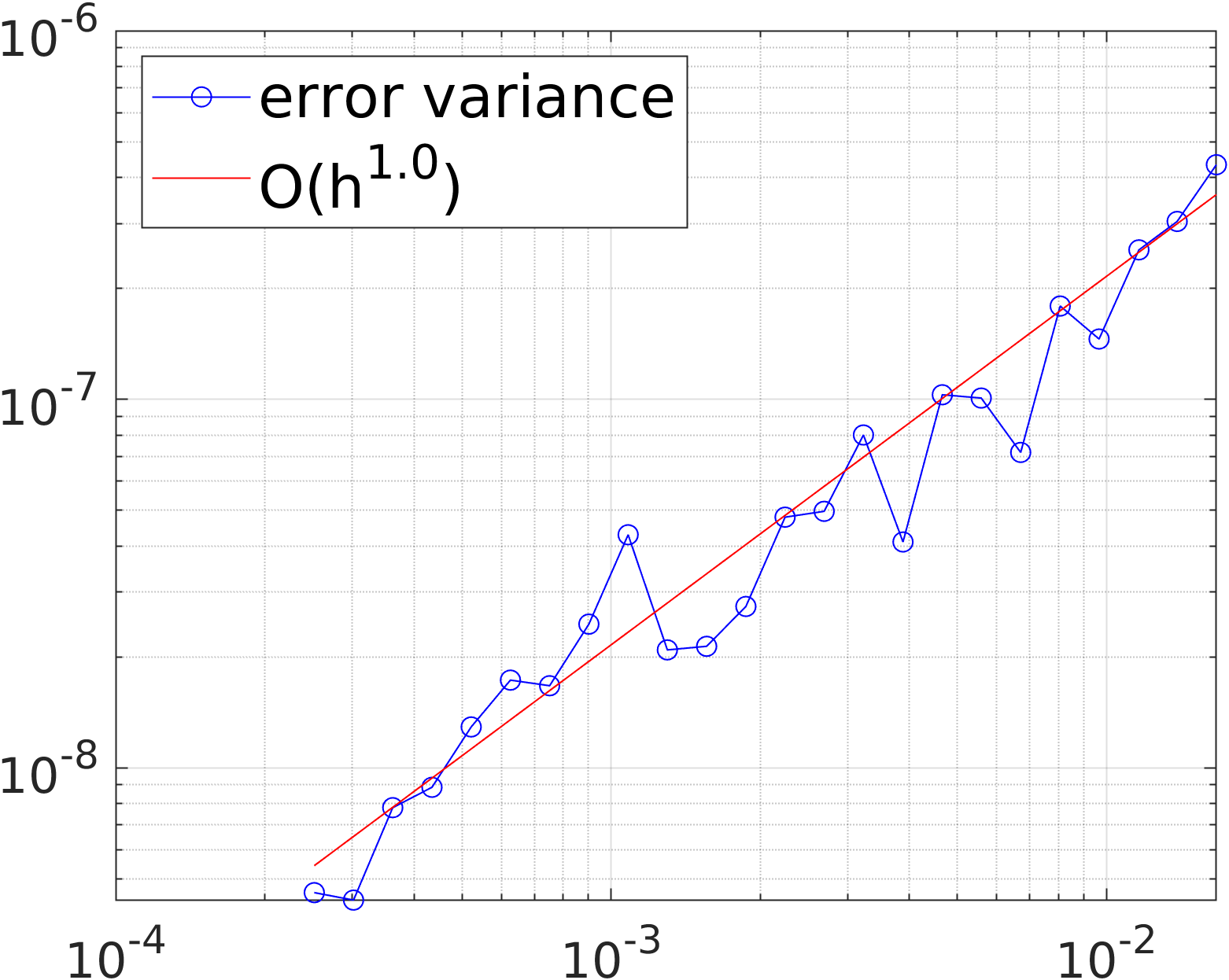}\qquad
    \includegraphics[scale=0.34]{ 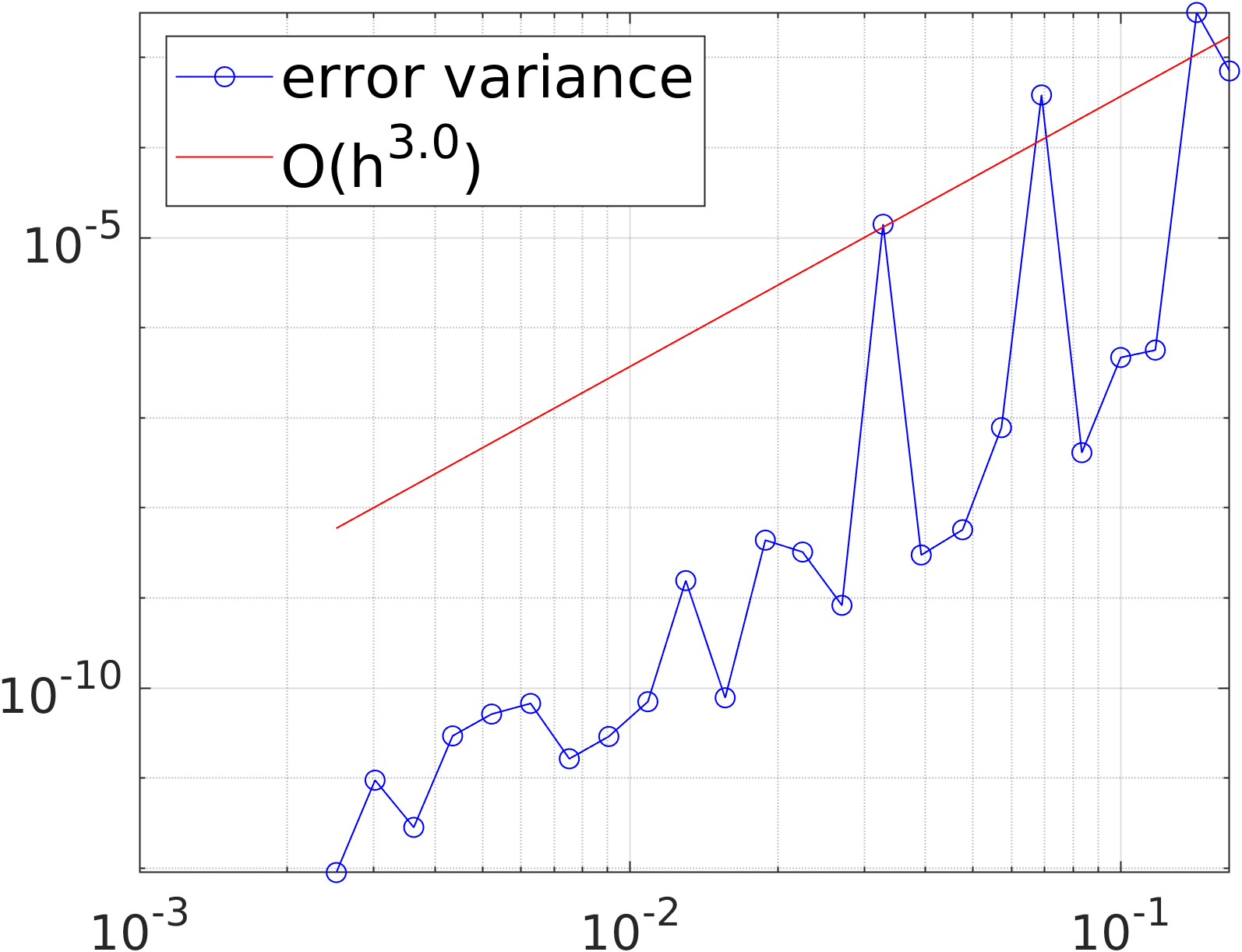}\qquad
    \includegraphics[scale=0.34]{ 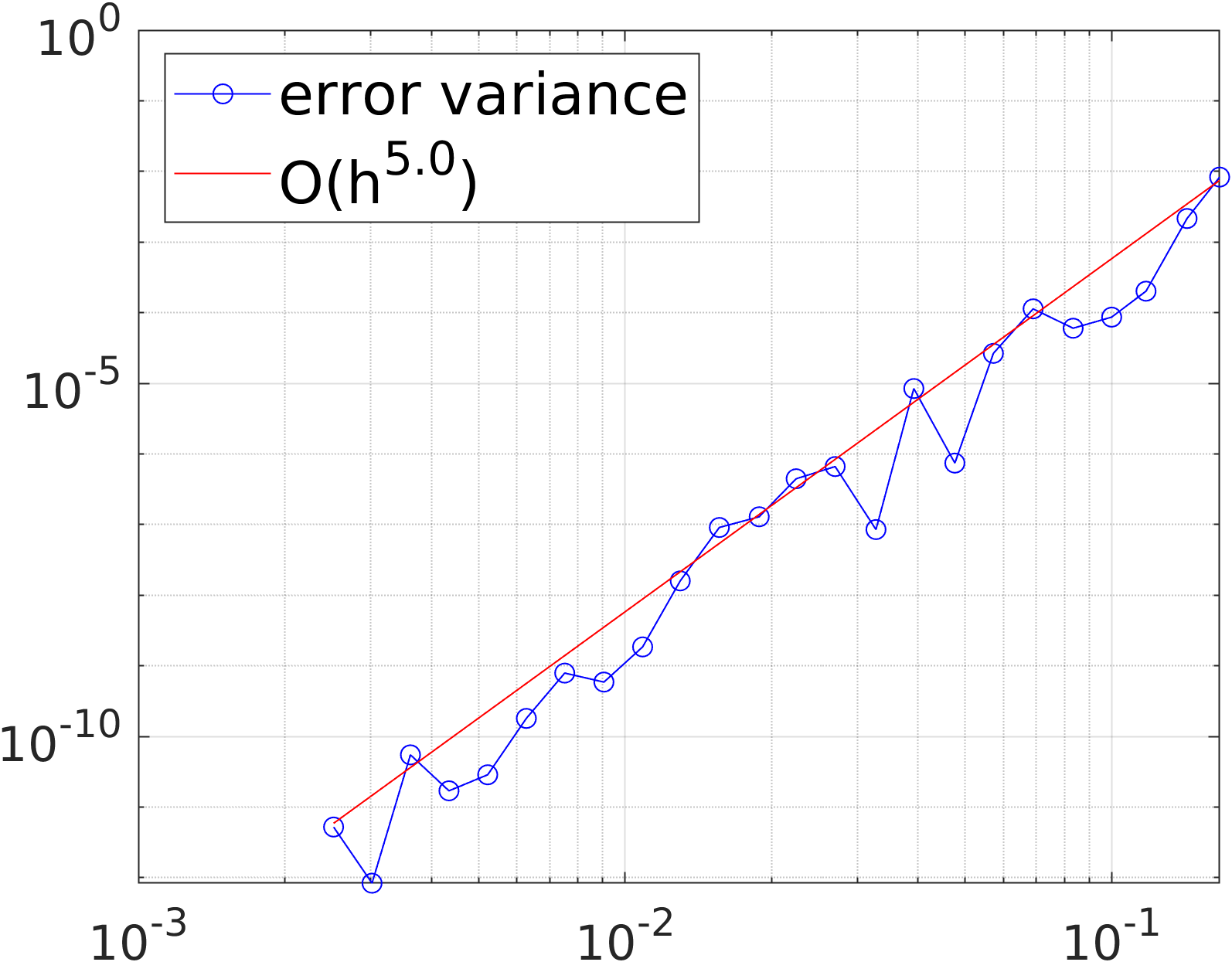}
    \caption{Variance of quadrature error with weight function $\theta_{\eps}^{\Delta}$ with respect to grid size $h$. Left: tube width $\eps=2h$. Middle: tube width $\eps=2h^{\frac{1}{2}}$. Right: tube width $\eps=0.1$. }
    \label{fig: error var 1}
\end{figure}
\begin{figure}[!htb]
    \centering
     \includegraphics[scale=0.34]{ 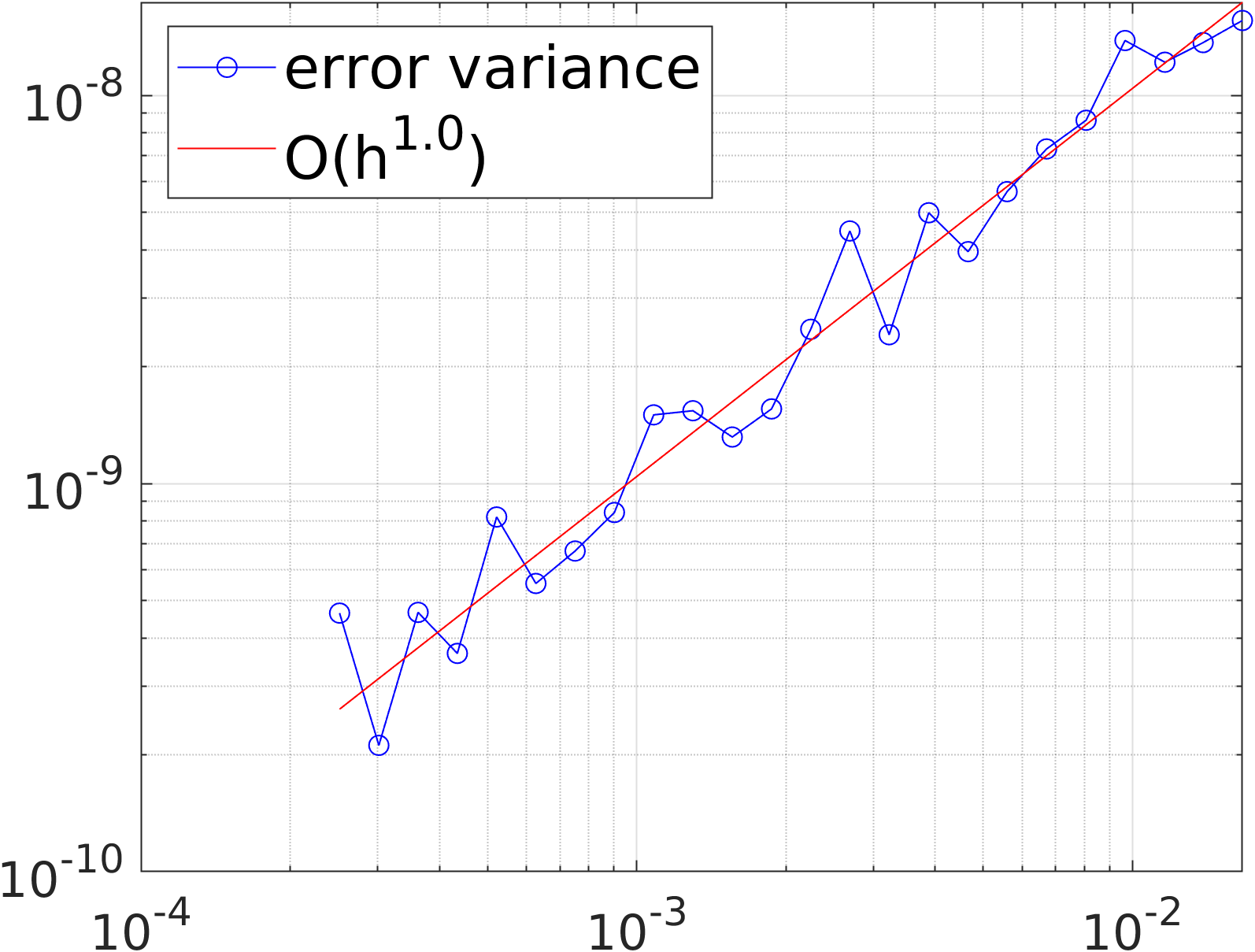}\qquad
    \includegraphics[scale=0.34]{ 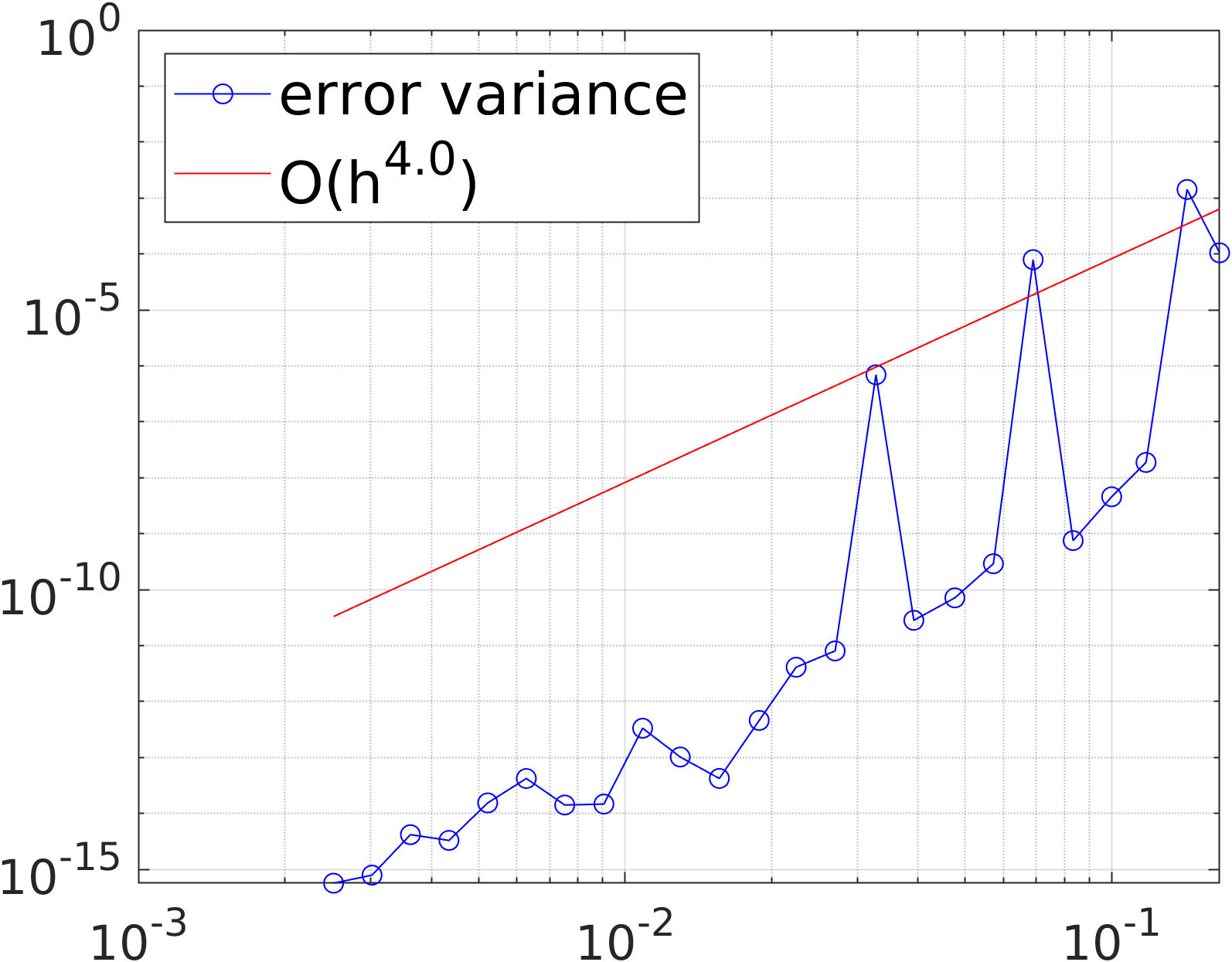}\qquad
    \includegraphics[scale=0.34]{ 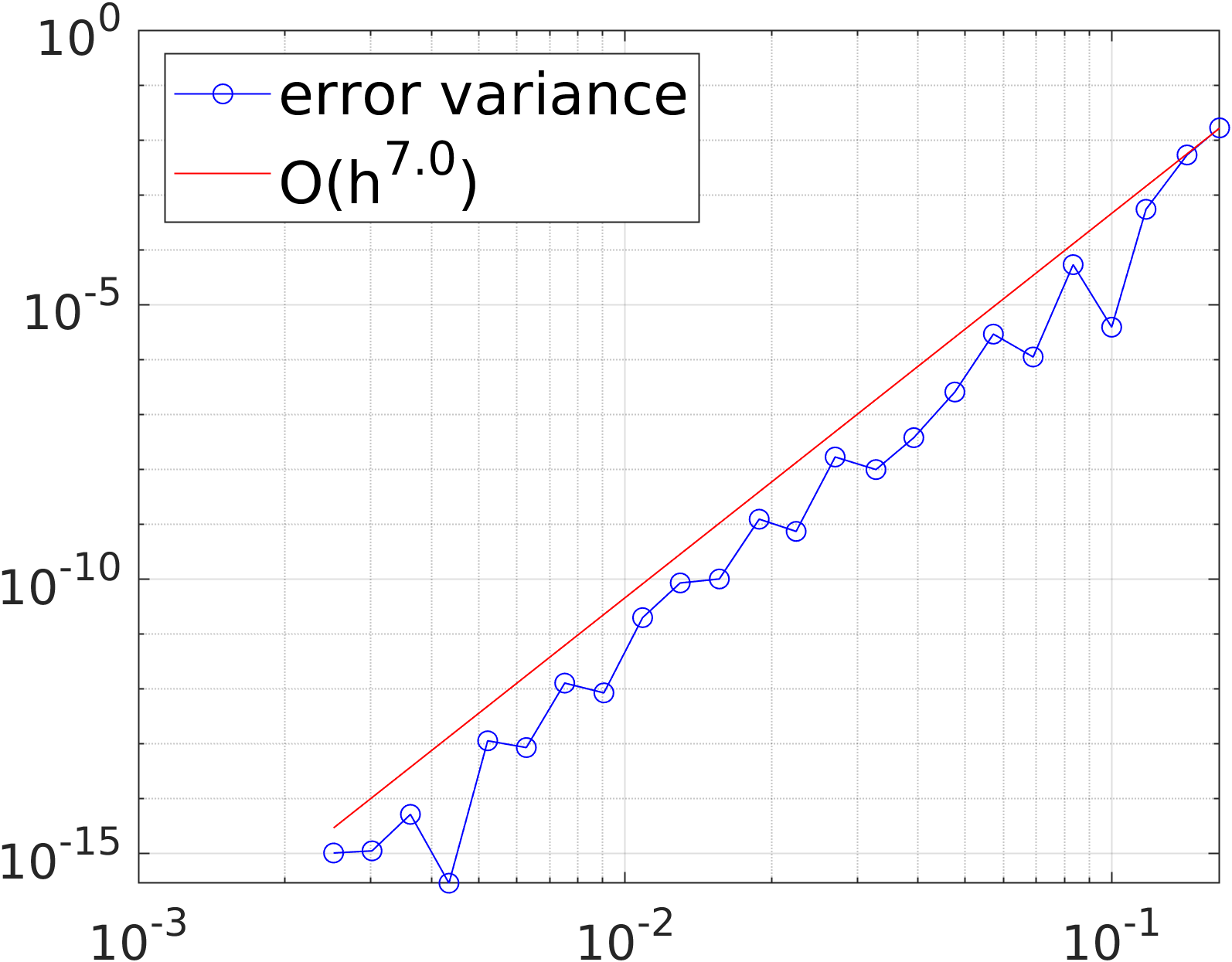}
    \caption{Variance of quadrature error with weight function $\theta_{\eps}^{\cos}$ with respect to grid size $h$. Left: tube width $\eps=2h$. Middle: tube width $\eps=2h^{\frac{1}{2}}$. Right: tube width $\eps=0.1$. }
    \label{fig: error var 2}
\end{figure}

\subsubsection{3D sphere}
The 3D experiments are performed on the sphere $\Gamma$:
\begin{equation*}
    \frac{(x - x_0)^2}{r^2} + \frac{(y - y_0)^2}{r^2} + \frac{(z - z_0)^2}{r^2} = 1
\end{equation*}
with $r = \frac{3}{4}$ and $(x_0, y_0, z_0)$ a randomly sampled point.  The test integrand function $f:\bbR^2 \mapsto \bbR$ in this case is 
\begin{equation*}
    f(x, y, z) = \cos(x^2 - y - z^3) \sin(y^2 - x^3 - z).
\end{equation*}
With the weight function $\theta_{\eps}^{\Delta}\in \cW_1$ (resp. $\theta_{\eps}^{\cos}\in \cW_2$), the quadrature errors are shown in Figure~\ref{fig: error 3d-1} (resp. Figure~\ref{fig: error 3d-2}) for $\alpha = 1$, $\alpha = \frac{1}{2}$ and $\alpha = 0$. The estimated error bounds are still consistent with the rate $\cO(h^{1 + 2(1-\alpha)})$ (resp. $\cO(h^{1 + 3(1-\alpha)})$). The corresponding variance of quadrature error is plotted in Figure~\ref{fig: var error 3d 1} (resp. Figure~\ref{fig: var error 3d 2}).
\begin{figure}[!htb]
    \centering
     \includegraphics[scale=0.34]{ 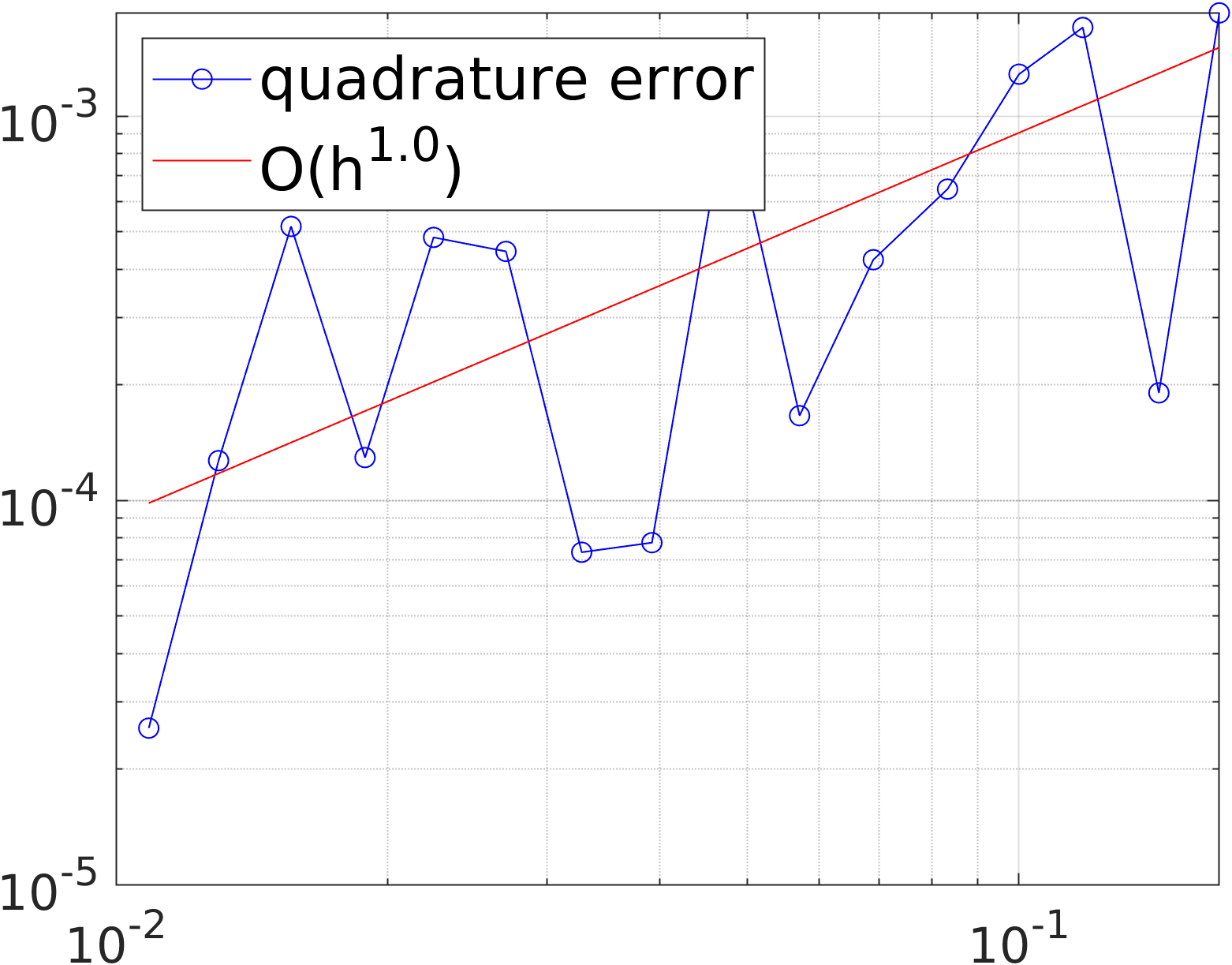}\qquad
    \includegraphics[scale=0.34]{ 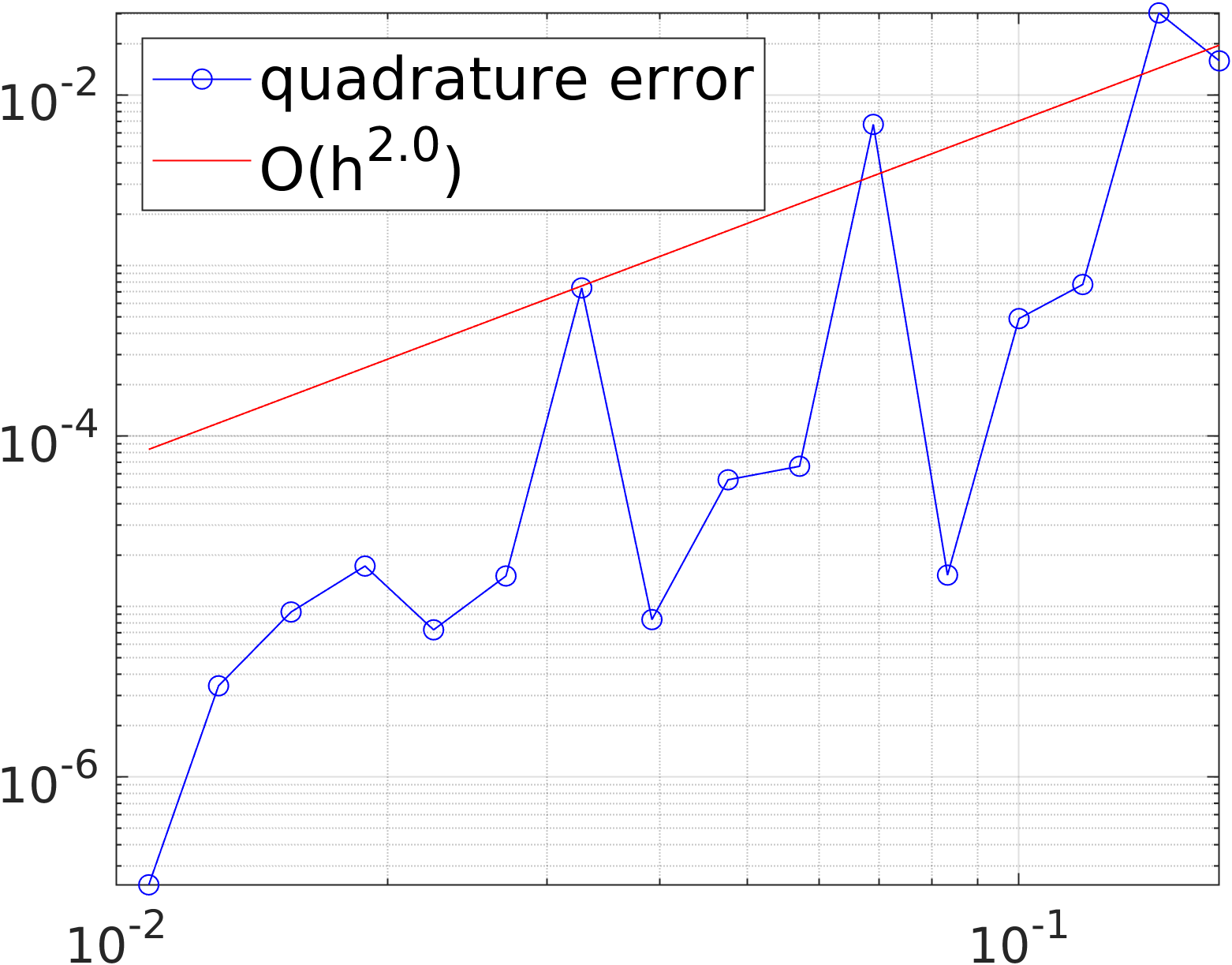}\qquad
    \includegraphics[scale=0.34]{ 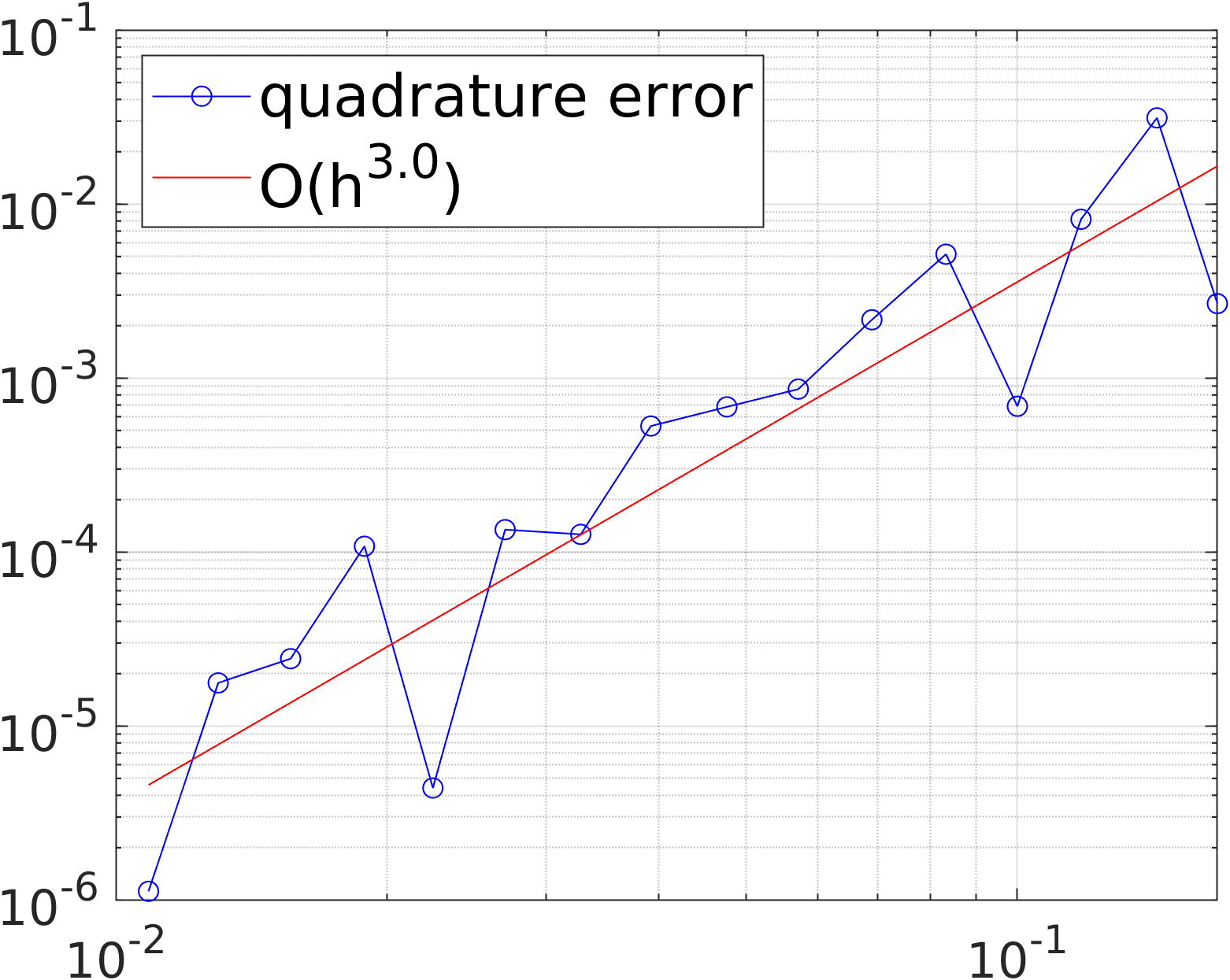}
    \caption{Quadrature error with weight function $\theta_{\eps}^{\Delta}$ with respect to grid size $h$. Left: tube width $\eps=2h$. Middle: tube width $\eps=2h^{\frac{1}{2}}$. Right: tube width $\eps=0.1$. }
    \label{fig: error 3d-1}
\end{figure}

\begin{figure}[!htb]
    \centering
     \includegraphics[scale=0.34]{ 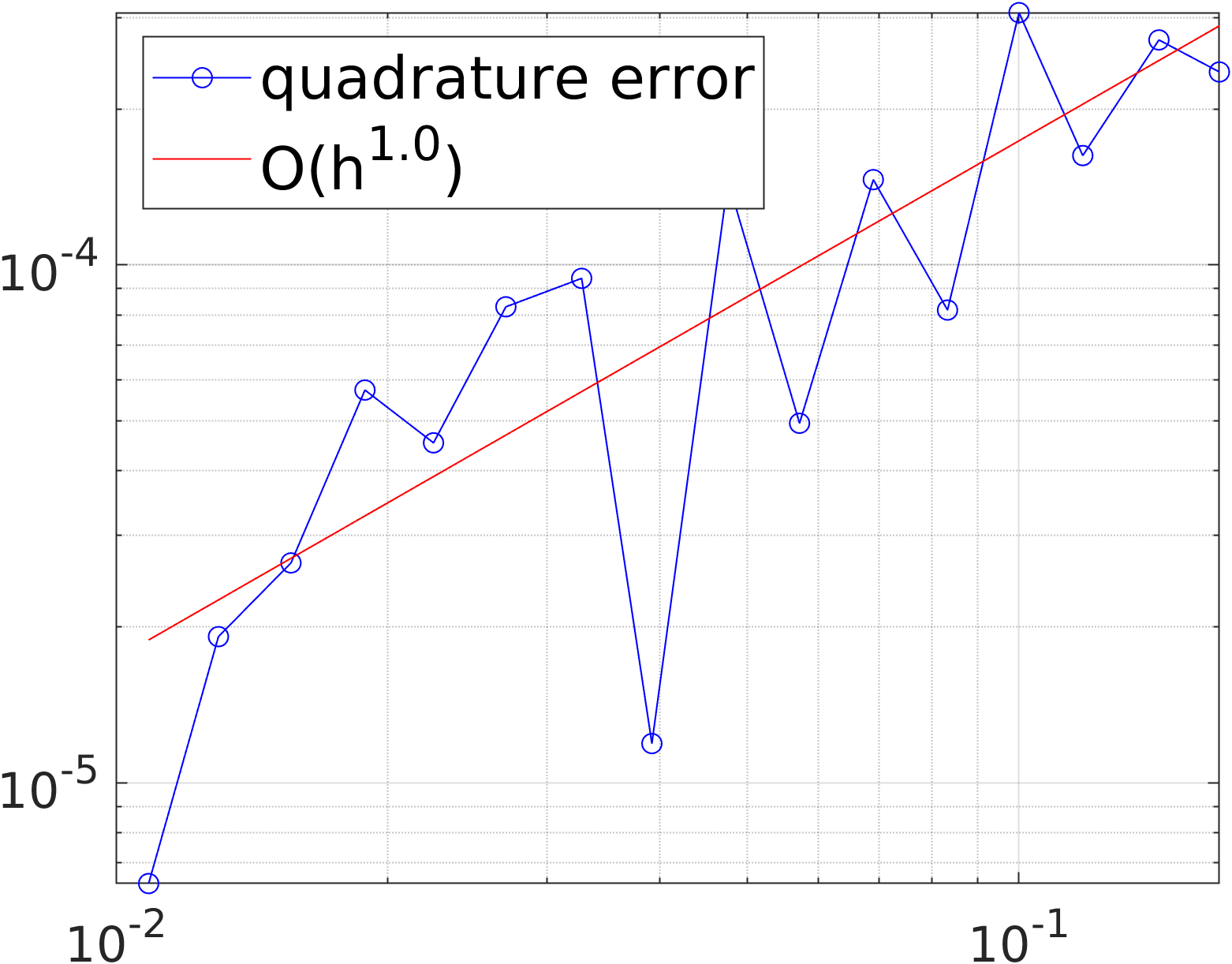}\qquad
    \includegraphics[scale=0.34]{ 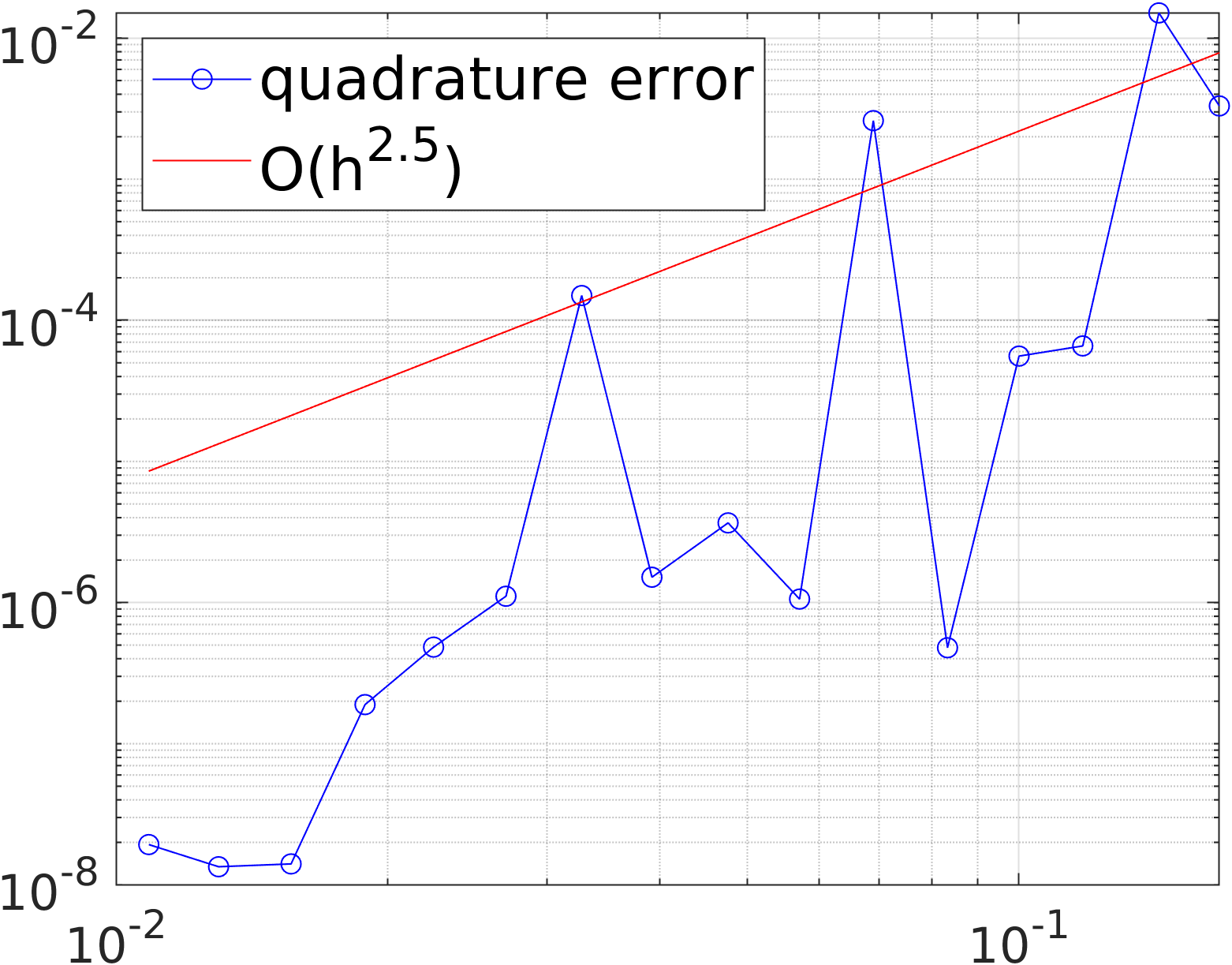}\qquad
    \includegraphics[scale=0.34]{ 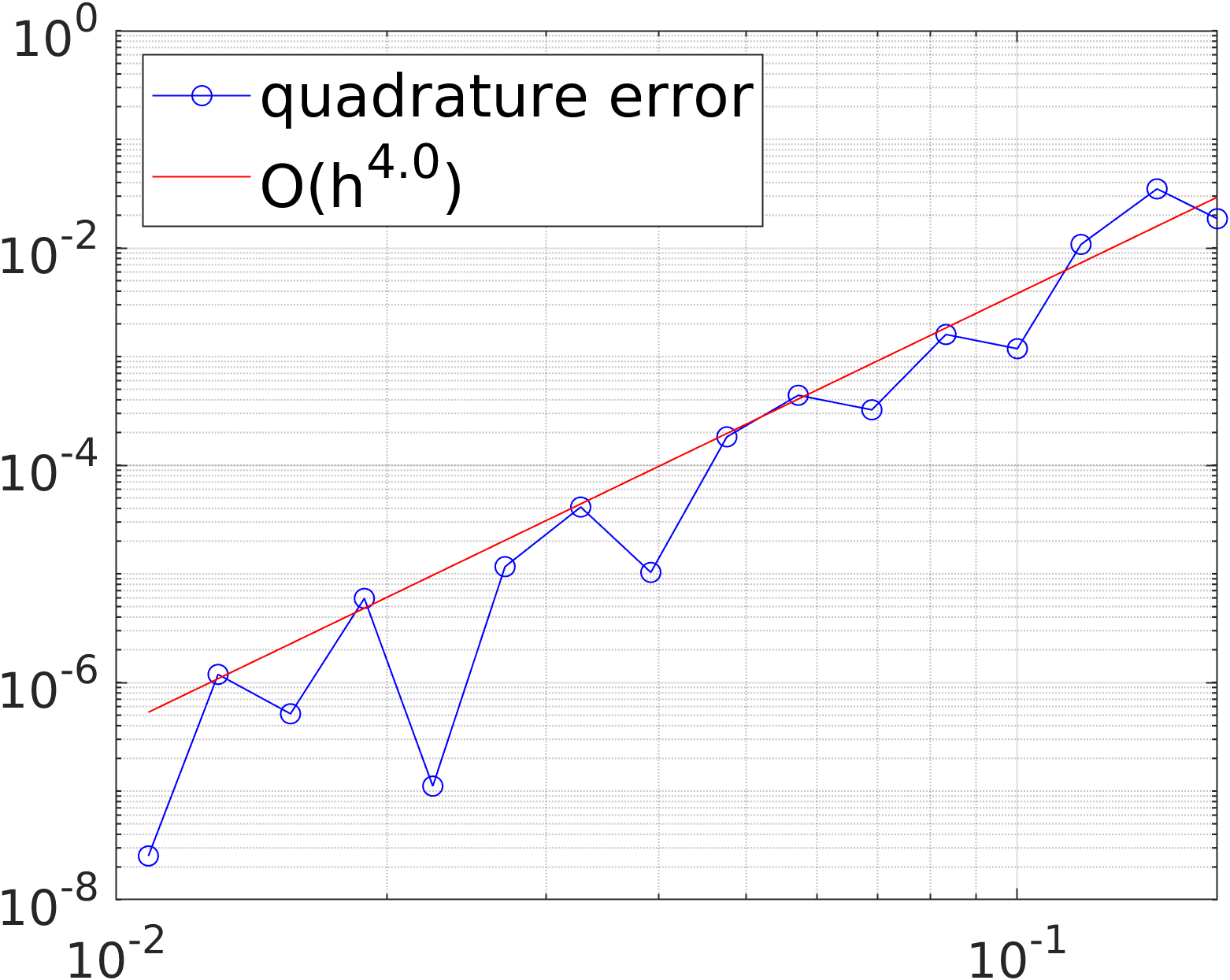}
    \caption{Quadrature error with weight function $\theta_{\eps}^{\cos}$ with respect to grid size $h$. Left: tube width $\eps=2h$. Middle: tube width $\eps=2h^{\frac{1}{2}}$. Right: tube width $\eps=0.1$. }
    \label{fig: error 3d-2}
\end{figure}

\begin{figure}[!htb]
    \centering
     \includegraphics[scale=0.34]{ 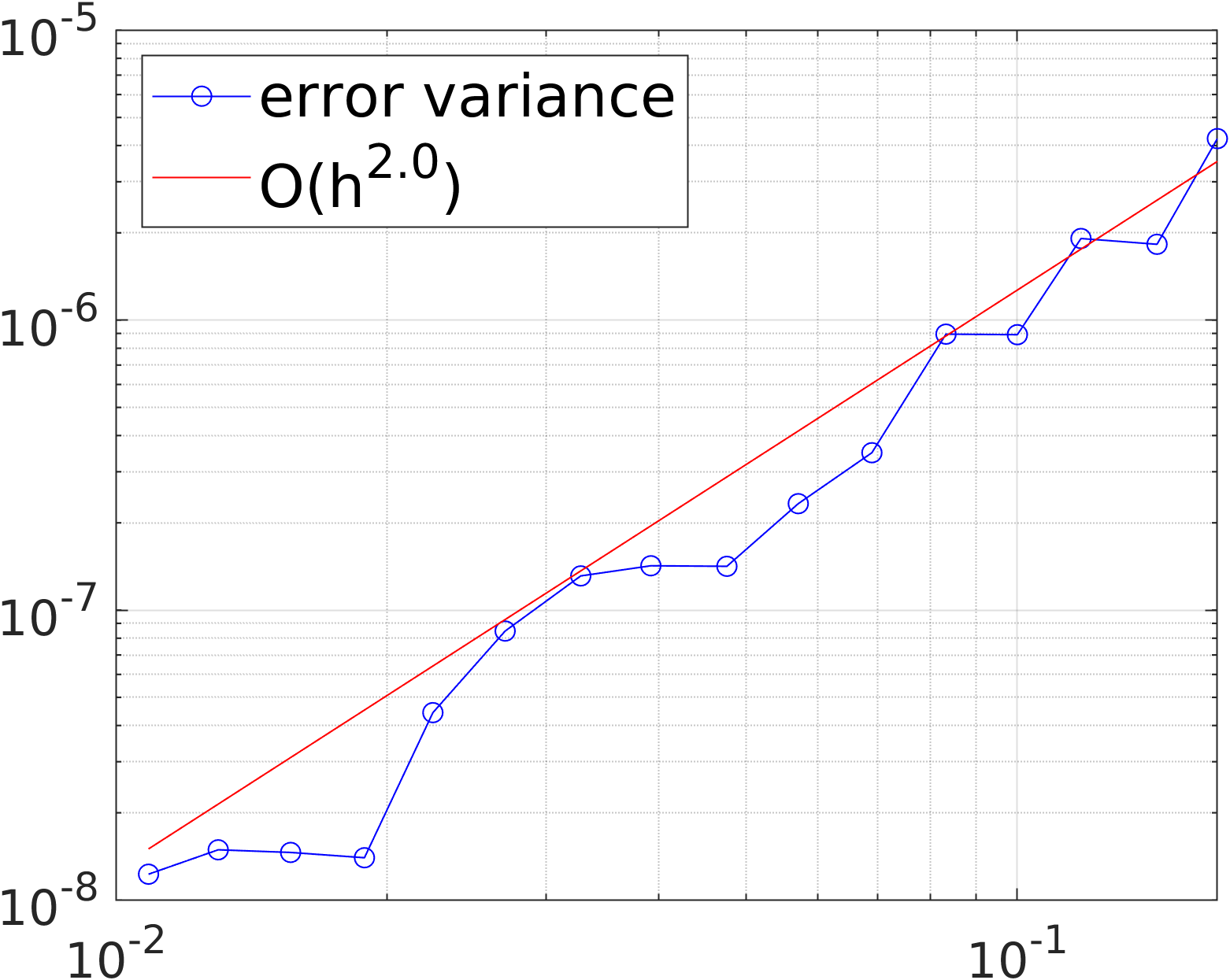}\qquad
    \includegraphics[scale=0.34]{ 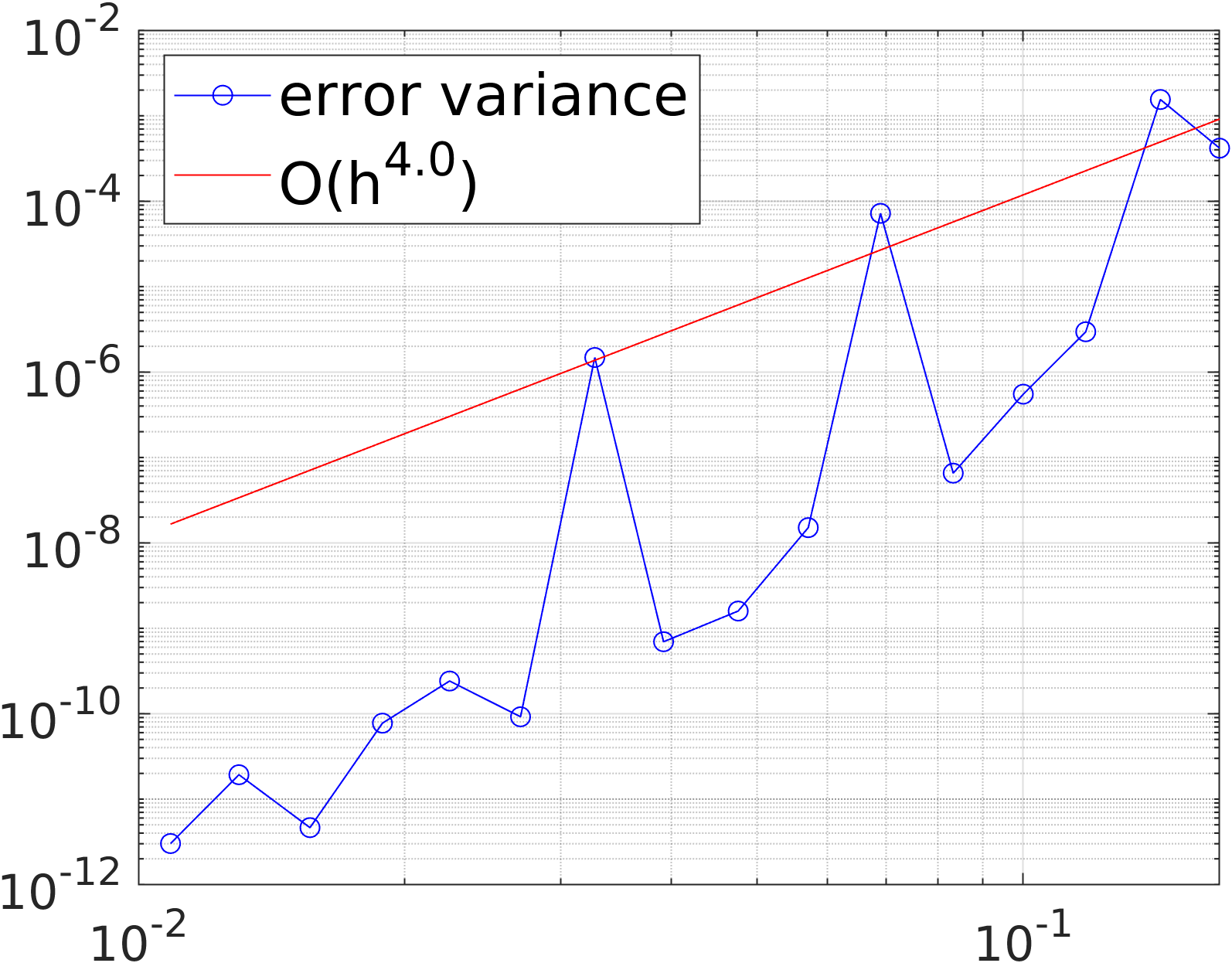}\qquad
    \includegraphics[scale=0.34]{ 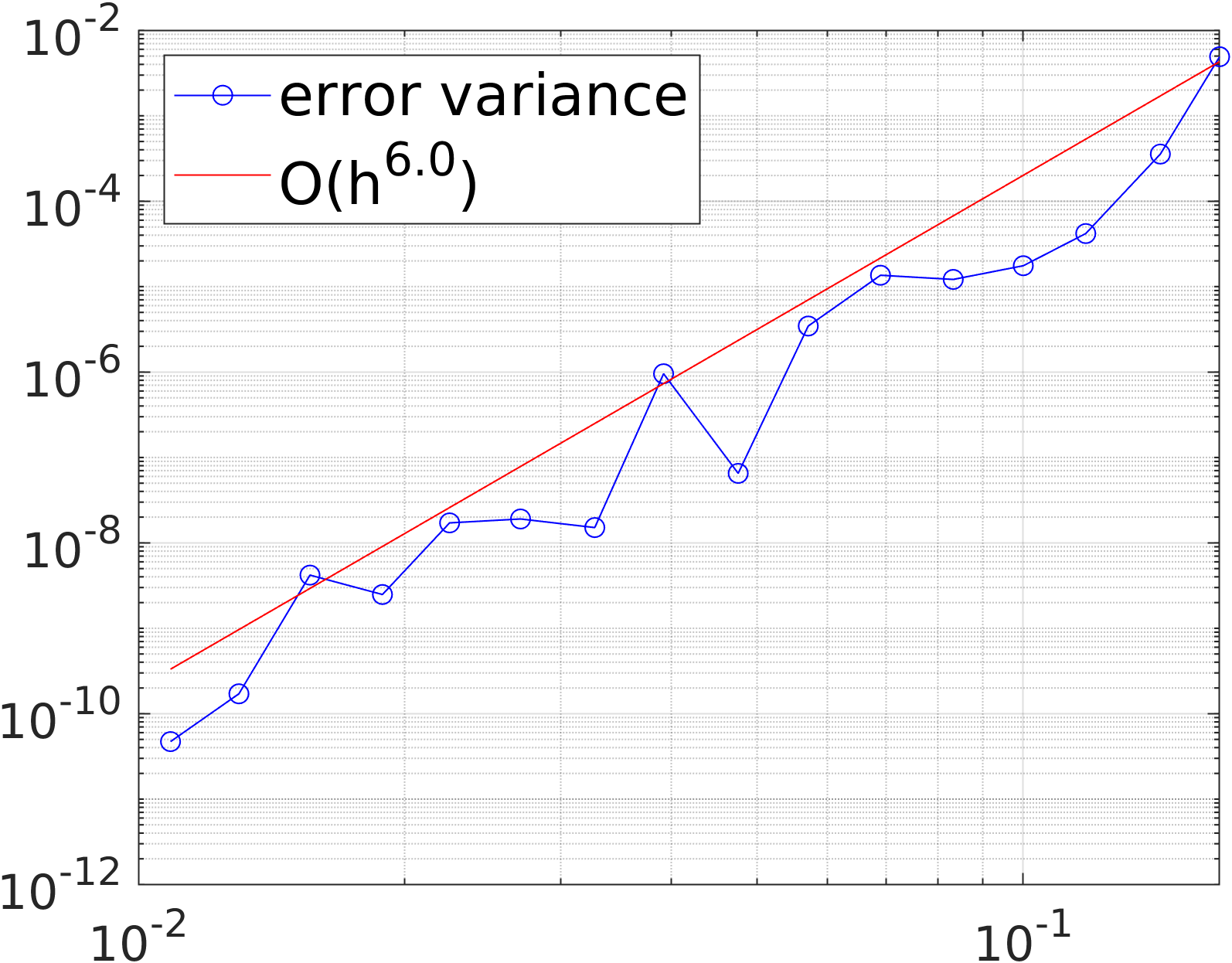}
    \caption{Variance of quadrature error with weight function $\theta_{\eps}^{\Delta}$ with respect to grid size $h$. Left: tube width $\eps=2h$. Middle: tube width $\eps=2h^{\frac{1}{2}}$. Right: tube width $\eps=0.1$. }
    \label{fig: var error 3d 1}
\end{figure}
\begin{figure}[!htb]
    \centering
     \includegraphics[scale=0.34]{ 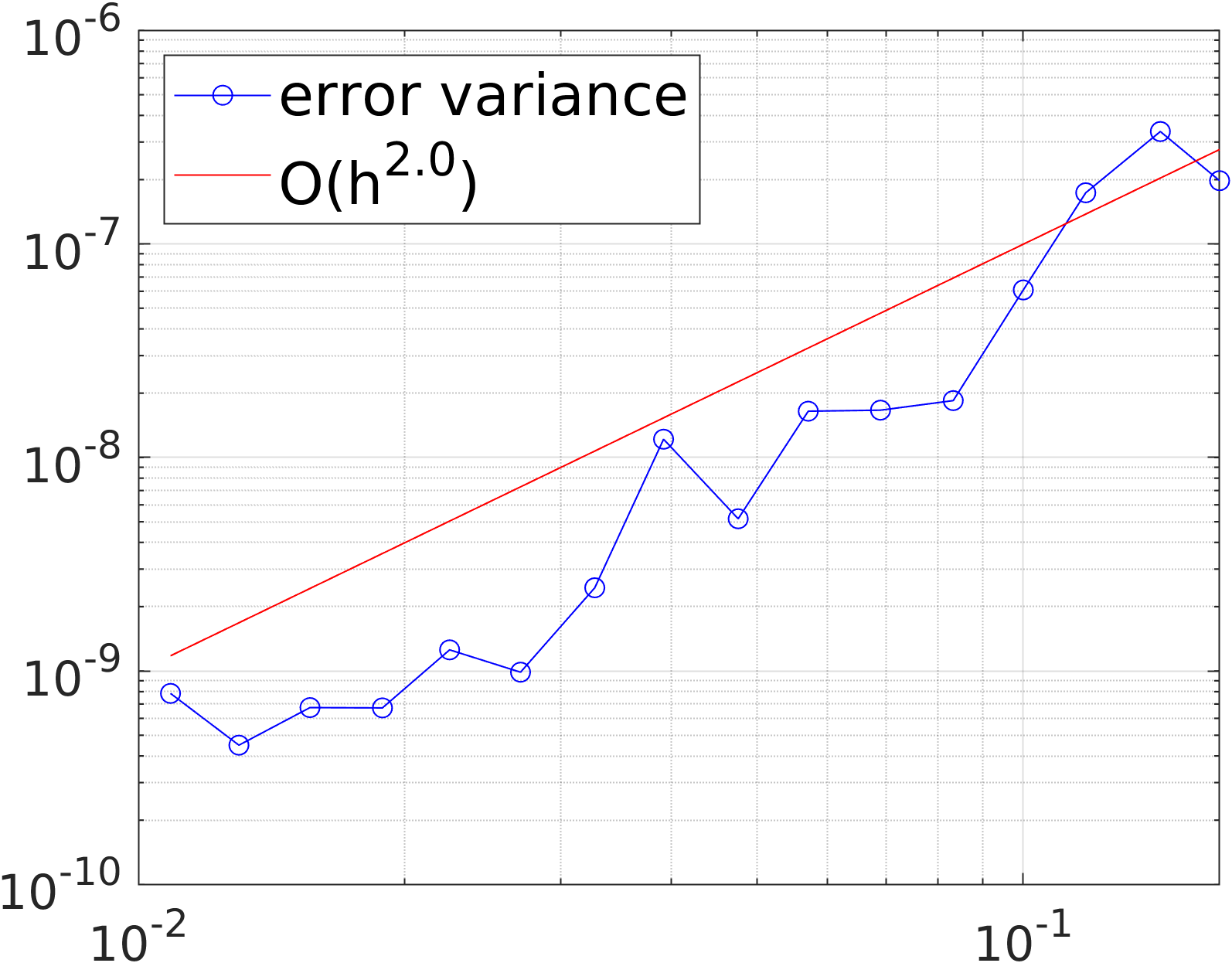}\qquad
    \includegraphics[scale=0.34]{ 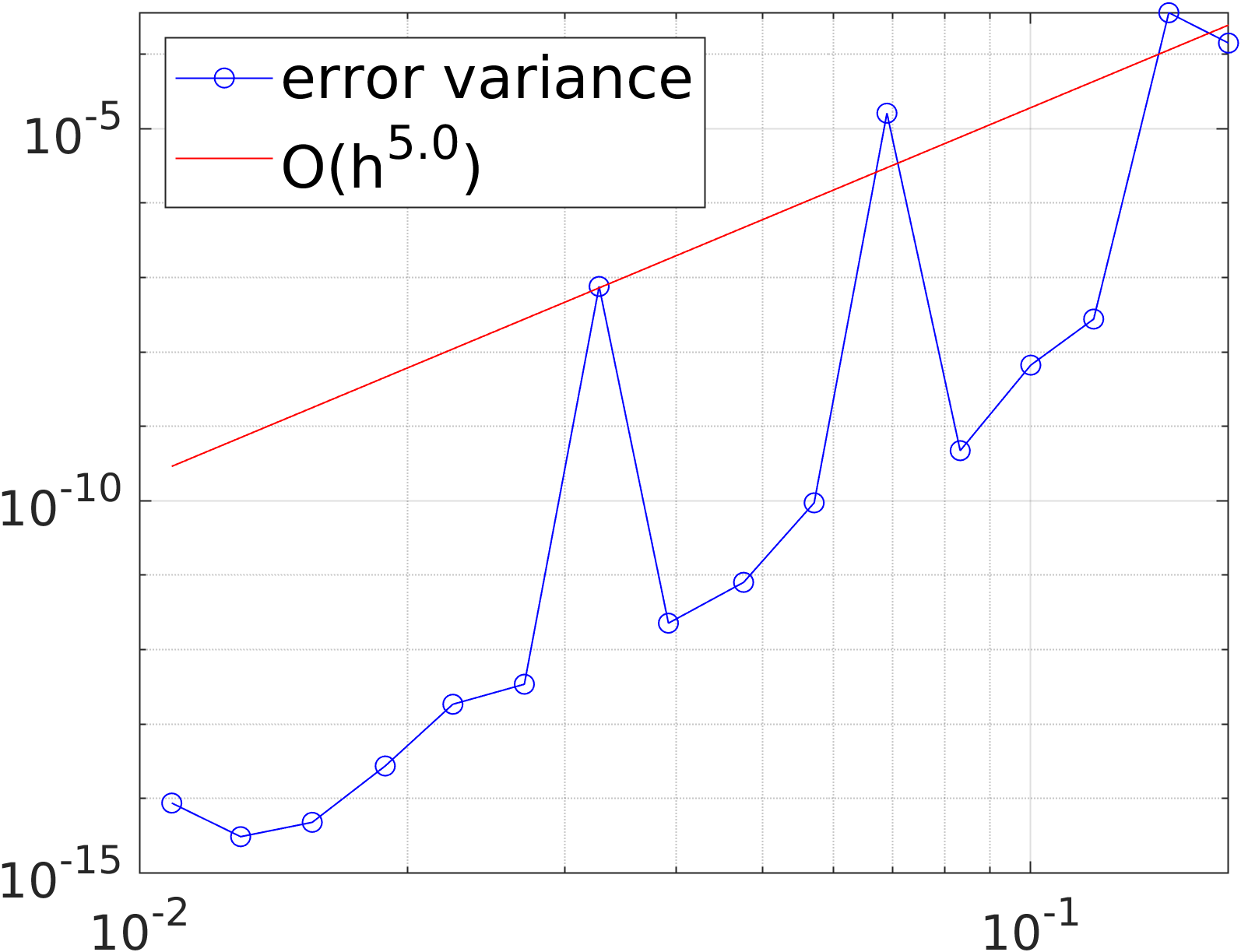}\qquad
    \includegraphics[scale=0.34]{ 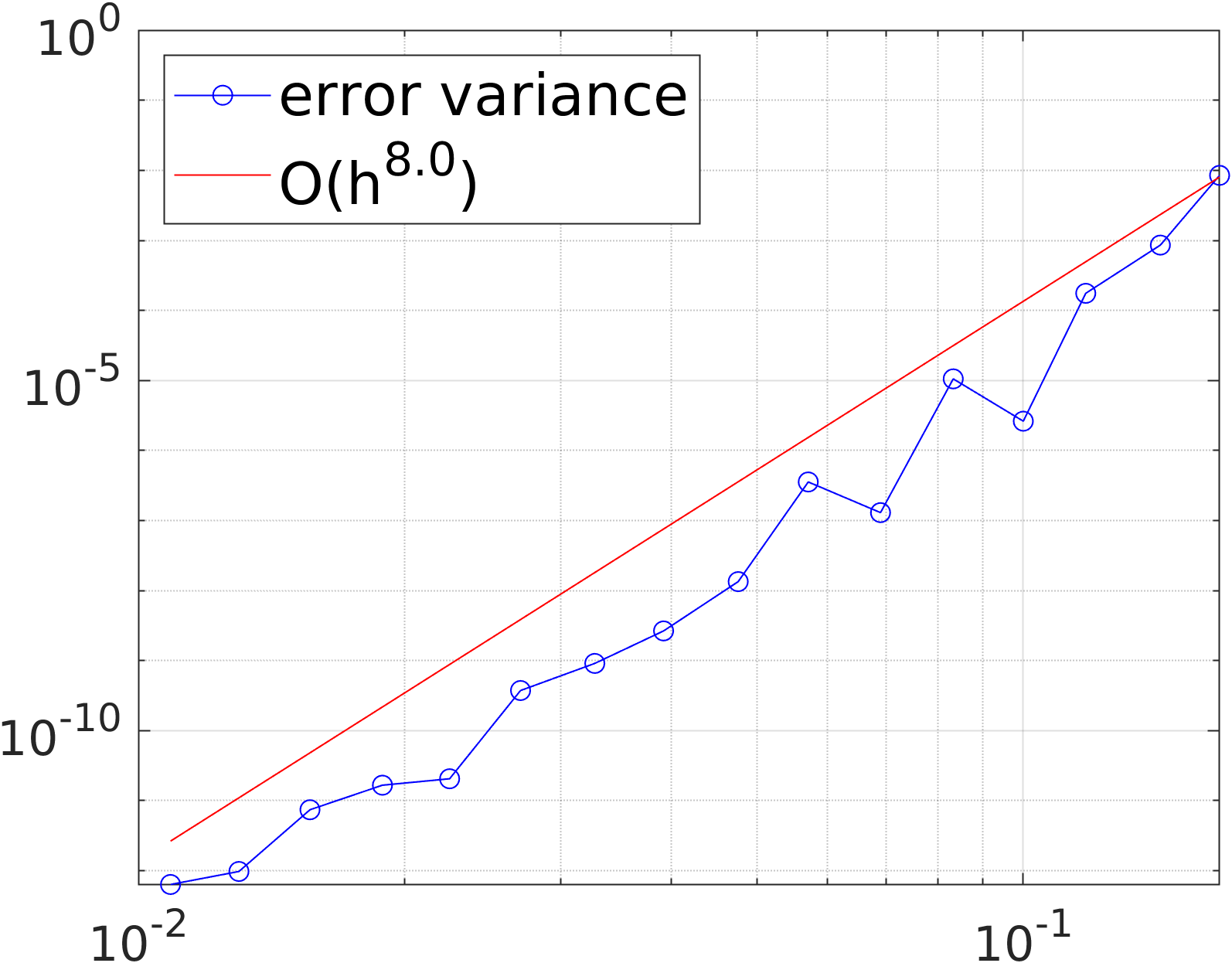}
    \caption{Variance of quadrature error with weight function $\theta_{\eps}^{\cos}$ with respect to grid size $h$. Left: tube width $\eps=2h$. Middle: tube width $\eps=2h^{\frac{1}{2}}$. Right: tube width $\eps=0.1$. }
    \label{fig: var error 3d 2}
\end{figure}

\begin{remark}
For the efficiency of computation, when the tube width $\eps= \Theta(h)$, one can apply the approximated Jacobian factor $J_{\eps}\approx 1$ instead of the accurate one if the dimension $d=2, 3$, without altering the error estimate $\cO(h^{\frac{d-1}{2}})$.  If the dimension $d=4, 5$, one can alternatively use the approximation $J_{\eps}(\bx, \eta) \approx 1 - d_{\Gamma}(\bx) \Delta d_{\Gamma}(\bx)$, which can be effectively computed by a local central difference scheme.
\end{remark}

\begin{remark}
It is worth noting that when the boundary is strongly convex, the quadrature error $\cO(h^{\frac{d-1}{2}})$ for a thin tube with width $\eps=\Theta(h)$ is somewhat equivalent to the error using a standard Monte Carlo method. This seems to imply certain equidistributed ``randomness'' of the lattice points inside the tube. However, this is still an open question even for spheres unless the tube width $\eps$ is at least $ \Theta(h^{\frac{11}{16}})$, which can be derived by combining the Theorem 1 of~\cite{duke1990representation} and the lattice count theorem in~\cite{heath1999lattice}.
\end{remark}

\subsection{Convex but not strongly convex boundaries}

It turns out that the results in this section can be modified slightly for closed, convex but not strongly convex, surface $\Gamma\in C^{\infty}$ with at least one positive principal curvature. In this case, the stationary phase estimate used in proving Lemma~\ref{LEM: STATIONARY PHASE} gives a degenerate result (see Remark~\ref{REMARK:Convex}), which leads to a modified version of Lemma~\ref{LEM: STATIONARY PHASE} in which one replaces $d$ with $\Lambda+1$ ($\Lambda$ being the number of positive principle curvature of $\Gamma$). Based on this revised version of Lemma~\ref{LEM: STATIONARY PHASE}, we can reproduce all the theorems in this section to have the following result.
\begin{corollary}\label{COR: GENERAL}
    Let $\Gamma\in C^{\infty}$ be closed and convex with at least $\Lambda > 0$ strictly positive principal curvatures everywhere. Then every theorem of Section~\ref{SEC: MAIN} and Section~\ref{SEC: STATS} holds with $d$ replaced by $(\Lambda+1)$.
\end{corollary}

As an extreme example, in the following Figure~\ref{fig: straight line}, we show that if the boundary $\Gamma$ has a segment component parallel to $y = x$, where $\Lambda = 0$, numerical quadrature with tube width $\eps=\Theta(h)$ will introduce an $\cO(1)$ error even when integrating a constant function. In addition, if one computes the quadrature with random shifts, the standard deviation is still $\cO(1)$. Such an argument can be easily adapted to segments with rational slopes. 
\begin{figure}[!htb]
    \centering
    \includegraphics[scale=0.4]{ 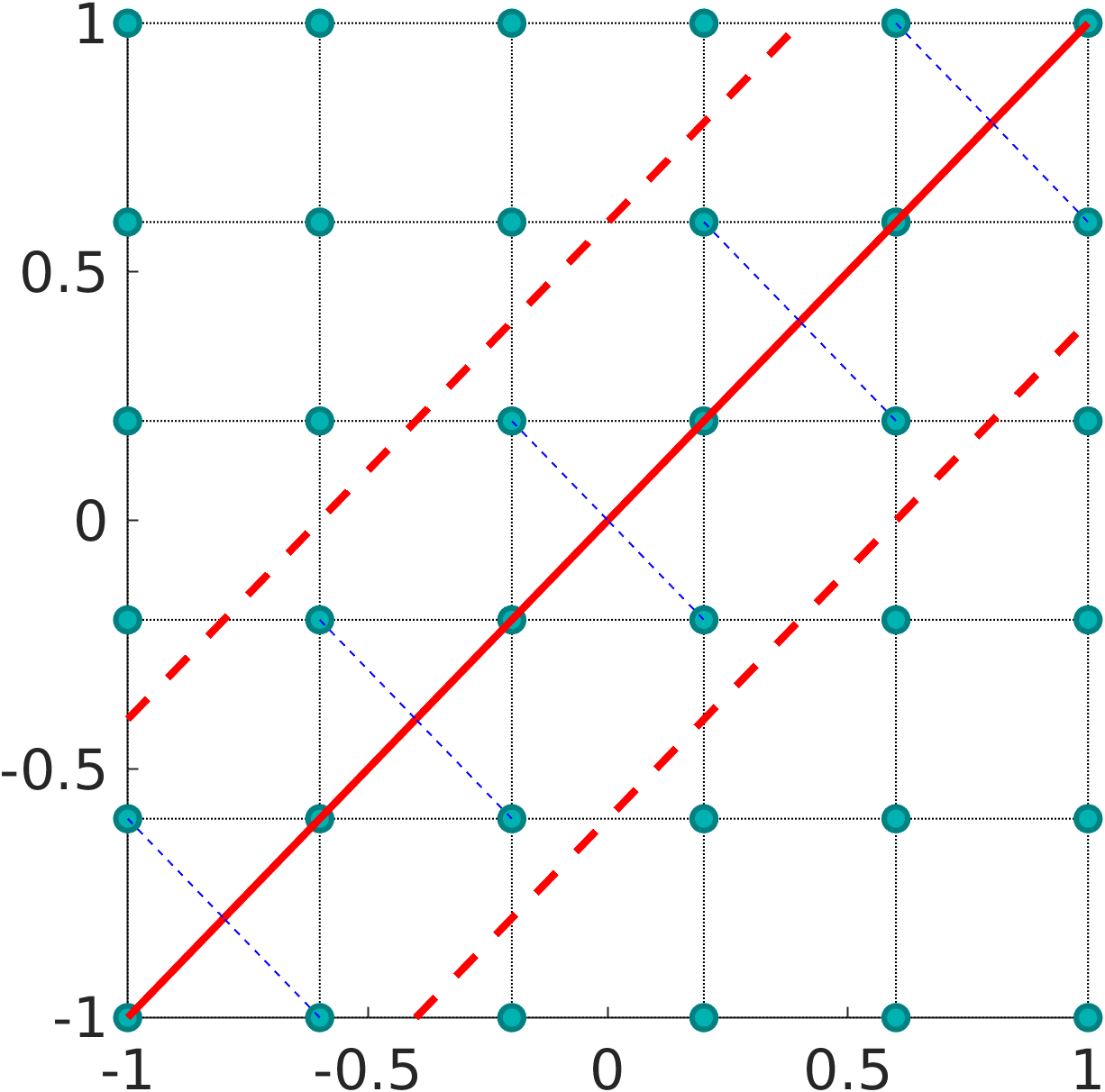}
    \caption{$\Gamma$ (red solid line) has a segment component as $y = x$. The tube, denoted by the red dash line, has a width $\eps = \frac{3}{2}h$. The green dots are lattice points, and the blue dash lines are projections to $\Gamma$.}
    \label{fig: straight line}
\end{figure}
\begin{remark}
Such a difference between smooth convex sets and polytopes has already been spotted in the famous lattice point problem. The early works trace back to Hardy and Littlewood~\cite{hardy1922some,hardy1922some2} over a century ago. The lattice point problem studies the discrepancy $$|t B \cap \bbZ^d| - \mathrm{meas}(B) t^d.$$
 The discrepancy could be regarded as random fluctuations if $B\subset \bbR^d$ has a smooth convex boundary. In contrast, if $B$ is a polytope, the leading term in the discrepancy may still behave polynomially~\cite{macdonald1971polynomials} in the variable $t$, which conceals the randomness behind the successive order. In addition, the fluctuations of randomness are closely related to the Diophantine approximation problem; see ~\cite{beck1994probabilistic,matousek1999geometric} and references therein.
\end{remark}


\section{Theories for smooth closed curves in \texorpdfstring{$\bbR^2$}{}}
\label{SEC: 3}

In this section, we will establish a general theory for the implicit boundary integral on smooth closed curves in $\bbR^2$. The case when the curve is closed and strongly convex is only a special case of the results in Section~\ref{SEC: 2}. We therefore focus only on the cases where the curvature vanishes at certain points of the curves. 

\subsection{Vanishing curvature at isolated points}\label{SEC: 3.1}
We start with a generic case where the curvature vanishes at a finite number of points.  It suffices to consider only one point with vanishing curvature using a partition of unity argument. We set this point $\bz$ as the origin. We assume that the curvature at $\bz$ vanishes to order $\kappa - 2$ for certain $\kappa > 2$,  which means \emph{locally} we can arrange the frame and represent $\Gamma$ as $(x, g(x))\subset \bbR^2$ such that
\begin{equation}\label{EQ: VANISH CUR}
    g(x) = |x|^{\kappa} h(x),  \quad h\in C^{\infty}[-r, r], \quad h(x)\neq 0\quad \forall x\in[-r,r]
\end{equation}
for certain $r > 0$. Then in the proof of Lemma~\ref{LEM: STATIONARY PHASE} (see Appendix~\ref{PRF: STATION}),  instead of gaining an additional decay factor $\cO(|\bzeta|^{1/2})$ for~\eqref{EQ: STA PHASE}, the factor becomes $\cO(|\bzeta|^{1/\kappa})$ in the spirit of Lemma~\ref{LEM: van der Corput Revised} (a generalized version of the van der Corput Lemma~\cite{stein1993harmonic}). Repeating the proofs in Section~\ref{SEC: MAIN}, one can obtain the following result.
\begin{corollary}
    Let $d=2$, and $\Gamma$ be a smooth closed curve with vanishing curvatures to order $\kappa-2$ ($\kappa>2$) at finitely many points. Then, when the tube width is $\eps=\cO(h^{\alpha})$, the quadrature error scales as $$|\cI f  - \cI_h f|=\cO(h^{\frac{1}{\kappa}+(q+1)(1-\alpha)})\,.$$
\end{corollary}
When $\alpha = 1$, the quadrature error is bounded by $\cO(h^{\frac{1}{\kappa}})$ which degenerates to $\cO(1)$ as $\kappa\to\infty$ (that is when the curve locally becomes straight, as in Figure~\ref{fig: straight line}). While this worst quadrature error may be unsatisfactory for practical uses, under random rigid transformations, the ``average'' quadrature error could be greatly improved over this bound.


\subsection{Variance of error under random rigid transformations}

The discussions in Section~\ref{SEC: STATS} suggest that the distribution of signed distances on lattice points to a strongly convex smooth boundary behaves ``almost'' random so that most rigid transformations (rotations and translations) of the lattice points will not change the quadrature error too much. In this section, we show that under random rigid transformations,  even with a finite number of points on $\Gamma$ with vanishing curvatures will \emph{not} alter the ``average'' quadrature error $\cO(h^{\frac{1}{2}+(q+1)(1-\alpha)})$.
\begin{theorem}\label{THM: DEGEN STAT}
    Let $\kappa\in (2, \infty)$ and $\Gamma$ be a convex closed $ C^{\infty}$ curve with finitely many points of vanishing curvature of maximal order $(\kappa-2)$. For any rotation $\eta\in \mathrm{SO}(2)$ and translation $\bxi\in [0, 1]^2$, let $\cI_h(f; \eta, \bxi)$ denote the implicit boundary integral with tube width $\eps=\Theta(h^{\alpha})$, $\alpha\in[0, 1]$ on the transformed boundary $\eta \Gamma + h \bxi$. Then, we have that 
    \begin{equation}
    \int_{[0, 1]^2}\int_{\eta\in \mathrm{SO}(2)} |\cI_h(f; \eta,\bxi) - \cI(f) |^2 d\eta d\bxi = \cO(h^{1+2(q+1)(1-\alpha)} )\,,
    \end{equation}
    where $d\eta$ is the normalized Haar measure on $\mathrm{SO}(2)$ and $q\ge 0$ is the regularity order of the weight function.
\end{theorem}
\begin{proof}
    Let $\cQ_{\Gamma}(\bx) := f(P_{\Gamma}\bx)\theta_{\eps}(d_{\Gamma}(\bx))$. The quadrature $\cI_h(f;\eta,\bxi)$ can be written as 
    \begin{equation*}
        \cI_h(f; \eta) = h^2\sum_{\bw\in \bbZ^2} \cQ_{\eta\Gamma + h\bxi}(h \bw) = h^2 \sum_{\bw\in \bbZ^2} \cQ_{\Gamma}(h\eta^{-1}(\bw - \bxi)).
    \end{equation*}
Similar to that in the proof of Theorem~\ref{THM: STAT 1}, we have
    \begin{equation*}
    \begin{aligned}
           \int_{\mathrm{SO}(2)} \int_{[0, 1]^2}|\cI(f; \eta,\bxi) - \cI(f) |^2 d\bxi d\eta &= \int_{\mathrm{SO}(2)} \sum_{\bw\in\bbZ^2, \bw\neq \bzero}|\widehat{\cQ}_{\Gamma}(h^{-1}\eta^{-1}\bw)|^2 d\eta.
    \end{aligned}
    \end{equation*}
    Next, we provide a slightly different estimate of $\widehat{\cQ}_{\Gamma}$ from the one in the proof of Lemma~\ref{LEM: STATIONARY PHASE}. Without loss of generality, we choose a partition of unity $\{\phi_j\}_{j=1}^N$ for $\Gamma$ such that each support $\supp \phi_j\subset \Gamma$ contains exactly one point with vanishing curvature to order $\kappa - 2$. The case where $\supp \phi_j$ does not contain any point with vanishing curvature is already handled in Theorem~\ref{THM: STAT 1}. Locally, the support of $\phi_j$ is represented by
    $$\supp\phi_j = \{(x, g_j(x)) \mid x\in (-r, r)\}\,, $$
    with $g_j(x)$ taking the form of ~\eqref{EQ: VANISH CUR}. Let $\bw = |\bw|(\sin\delta, \cos\delta)$. Then, according to~\eqref{EQ: STA PHASE}, if we set $\bx = (x(\delta), g_j(\delta))\in \Gamma$, then $\bx + s \bn(\bx)$ becomes a stationary point  if
    \begin{equation}\label{EQ:Sta Point Cond}
       \sin(\delta) + g'_j(x(\delta)) \cos(\delta) -  \frac{ s g''_j(x(\delta))}{(1 + |g'_j(x(\delta))|^2))^{3/2}} \left( \sin(\delta) + g'_j(x(\delta)) \cos(\delta)\right) = 0\,. 
    \end{equation}
    When $\eps $ is small enough, ~\eqref{EQ:Sta Point Cond} is equivalent to 
    \begin{equation*}
 \sin\delta + g'_j(x(\delta)) \cos\delta = 0\,.
    \end{equation*}
    This leads to $ |x(\delta)| = \Theta(|\delta |^{1/(\kappa - 1)})$ since locally $g'_j(x) = \Theta(|x|^{\kappa - 1})$. We estimate $\widehat \cQ_\Gamma$ in three different cases:
    \begin{enumerate}
        \item If $|x(\delta)| > r$, then there are no stationary points near the point with vanishing curvature. The estimate will be the same as that in Lemma~\ref{LEM: STATIONARY PHASE}, that is, 
        \begin{equation*}
            |\widehat{\cQ_{\Gamma} \phi_j}(\bw) |= \cO(\eps^{-q-1} |\bw|^{-1/2 - (q+1)})\,.
        \end{equation*}
        \item If $|x(\delta)| < |\bw|^{-\frac{\kappa - 2}{\kappa(\kappa - 1)}}$, the point is close to the origin. We can then use Lemma~\ref{LEM: van der Corput Revised}) (the revised van der Corput Lemma) and Lemma~\ref{LEM: STATIONARY PHASE} to conclude that
        \begin{equation*}
            |\widehat{\cQ_{\Gamma} \phi_j}(\bw)| = \cO(\eps^{-q-1} |\bw|^{-1/\kappa - (q+1)})\,.
        \end{equation*}
        \item If $ r \ge  |x(\delta)| \ge |\bw|^{-\frac{\kappa - 2}{\kappa(\kappa - 1)}} $, the point's Hessian is at order $\cO(|x(\delta)|^{\kappa - 2})$. We apply the standard stationary phase approximation with the additional factor $|x(\delta)|^{-(\kappa - 2)/2} = \Theta(|\delta|^{-\frac{\kappa - 2}{2\kappa - 2}})$ to get  
        \begin{equation*}
             |\widehat{\cQ_{\Gamma} \phi_j}(\bw)| = \cO(\eps^{-q-1} |\bw|^{-1/2 - (q+1)} |\delta|^{-\frac{\kappa - 2}{2\kappa - 2}})\,.
        \end{equation*}
    \end{enumerate}
    Therefore, we have 
    \begin{equation*}
    \begin{aligned}
            \int_{\mathrm{SO(2)}} |\widehat{\cQ_{\Gamma} \phi_j}(h^{-1}\eta^{-1}\bw)|^2 d\eta &= \cO\left( \left|\eps^{-q-1} |h^{-1}\bw|^{-1/2 - (q+1)}\right|^2 \right) \\ &\quad + \cO\left(| h^{-1} \bw|^{-\frac{\kappa - 2}{\kappa}} \left|\eps^{-q-1} |h^{-1}\bw|^{-1/\kappa - (q+1)}\right|^2 \right) \\ &\quad + \int_{|h^{-1}\bw|^{-\frac{\kappa - 2}{\kappa}}}^{\cO(1)} \cO\left( \left|\eps^{-q-1} |h^{-1}\bw|^{-1/2 - (q+1)} |\delta|^{-\frac{\kappa - 2}{2\kappa - 2}}\right|^2 \right)d\delta\,.
    \end{aligned}
    \end{equation*}
    The first two terms are both bounded by $\cO\left(\eps^{-2q-2} h^{1 + 2(q+1)} |\bw|^{-1- 2(q+1)}\right) $, and the last integral is bounded by
    \begin{equation*}
        \cO(\eps^{-2q-2}|\bw|^{-1-2(q+1)} h^{1+2(q+1)})\int_{|h^{-1}\bw|^{-\frac{\kappa - 2}{\kappa}}}^r  |\delta|^{-\frac{\kappa - 2}{\kappa - 1}} d\delta =  \cO(\eps^{-2q-2}|\bw|^{-1-2(q+1)} h^{1+2(q+1)})\,.
    \end{equation*}
    Taking $\eps = \cO(h^{\alpha})$ then gives us 
    \begin{equation*}
    \begin{aligned}
         \int_{\mathrm{SO}(2)} \sum_{\bw\in\bbZ^2, \bw\neq \bzero}|\widehat{\cQ_{\Gamma} \phi_j}(h^{-1}\eta^{-1}\bw)|^2 d\eta &= \sum_{\bw\in\bbZ^2, \bw\neq \bzero}\cO(\eps^{-2q-2}|\bw|^{-1-2(q+1)} h^{1+2(q+1)}) \\
         &= \cO(h^{1+2(q+1)(1-\alpha)} )\,.
    \end{aligned}
    \end{equation*}
    The last step is to use the Cauchy-Schwartz inequality and the finiteness of $N$ to conclude that
    \begin{equation*}
    \begin{aligned}
          \int_{\mathrm{SO}(2)} \sum_{\bw\in\bbZ^2, \bw\neq \bzero}|\widehat{\cQ_{\Gamma}}(h^{-1}\eta^{-1}\bw)|^2 d\eta &=    \int_{\mathrm{SO}(2)} \sum_{\bw\in\bbZ^2, \bw\neq \bzero}|\sum_{j=1}^N\widehat{\cQ_{\Gamma} \phi_j}(h^{-1}\eta^{-1}\bw)|^2 d\eta \\
          &\le N  \int_{\mathrm{SO}(2)} \sum_{\bw\in\bbZ^2, \bw\neq \bzero}|\widehat{\cQ_{\Gamma} \phi_j}(h^{-1}\eta^{-1}\bw)|^2 d\eta\\
          &=\cO(h^{1+2(q+1)(1-\alpha)} )\,.
    \end{aligned}
    \end{equation*}
    The proof is complete.
\end{proof}

In fact, one can work a little harder to remove the convexity requirement in the above theorem. 
The main idea is that, instead of using the representation~\eqref{EQ: VANISH CUR} for the curve, which sustains the convexity, we can represent the curve locally as $(x, g(x))\subset \bbR^2$ for $x\in (-r, r)$ that
\begin{equation}\label{EQ: VANISH CUR 2}
     g(x) =  \sgn(x) |x|^{\kappa} h(x),\quad h(x) \neq 0\quad\forall x\in (-r,r),\quad h\in C^2(-r, r)\,.
\end{equation}
 The proof of Lemma~\ref{LEM: van der Corput Revised} (the revised van der Corput Lemma) shows both~\eqref{EQ: VANISH CUR} and~\eqref{EQ: VANISH CUR 2} can be handled by the same procedure; see Remark~\ref{RMK:Nonconvex}. We can thus have the following corollary.
\begin{corollary}\label{COR: GENERAL CURVE}
   The same estimate in Theorem~\ref{THM: DEGEN STAT} holds for any closed boundary $\Gamma\in C^{\infty}$ if it only has finitely many points with vanishing curvature, whose maximal order is $(\kappa - 2)$ for $\kappa\in (2, \infty)$. 
\end{corollary}

\subsection{Numerical experiments}
Here, we present some numerical simulations to verify the variance estimates in   Theorem~\ref{THM: DEGEN STAT}  and Corollary~\ref{COR: GENERAL CURVE} for the quadrature error under random rigid transformations. 

\subsubsection{Convex curve with vanishing curvatures}\label{SEC: VAR VANISH}
To numerically verify the variance estimate in Theorem~\ref{THM: DEGEN STAT}, we consider the convex curve
\begin{equation}\label{EQ: VANISH CUR EQ}
    \frac{(x-x_0)^4}{r^4} + \frac{(y-y_0)^2}{r^2} =1
\end{equation}
with $r=\frac{3}{4}$ and a random center $(x_0, y_0)$. The curvature vanishes to order two at $(x_0, y_0\pm r)$; See Figure~\ref{fig: vanish curvature}.  The integrand function is selected as
\begin{equation*}
    f(x, y) = \cos(x^2 - y) \sin(y^2 - x^3).
\end{equation*}
\begin{figure}[!htb]
    \centering
    \includegraphics[scale=0.4]{ 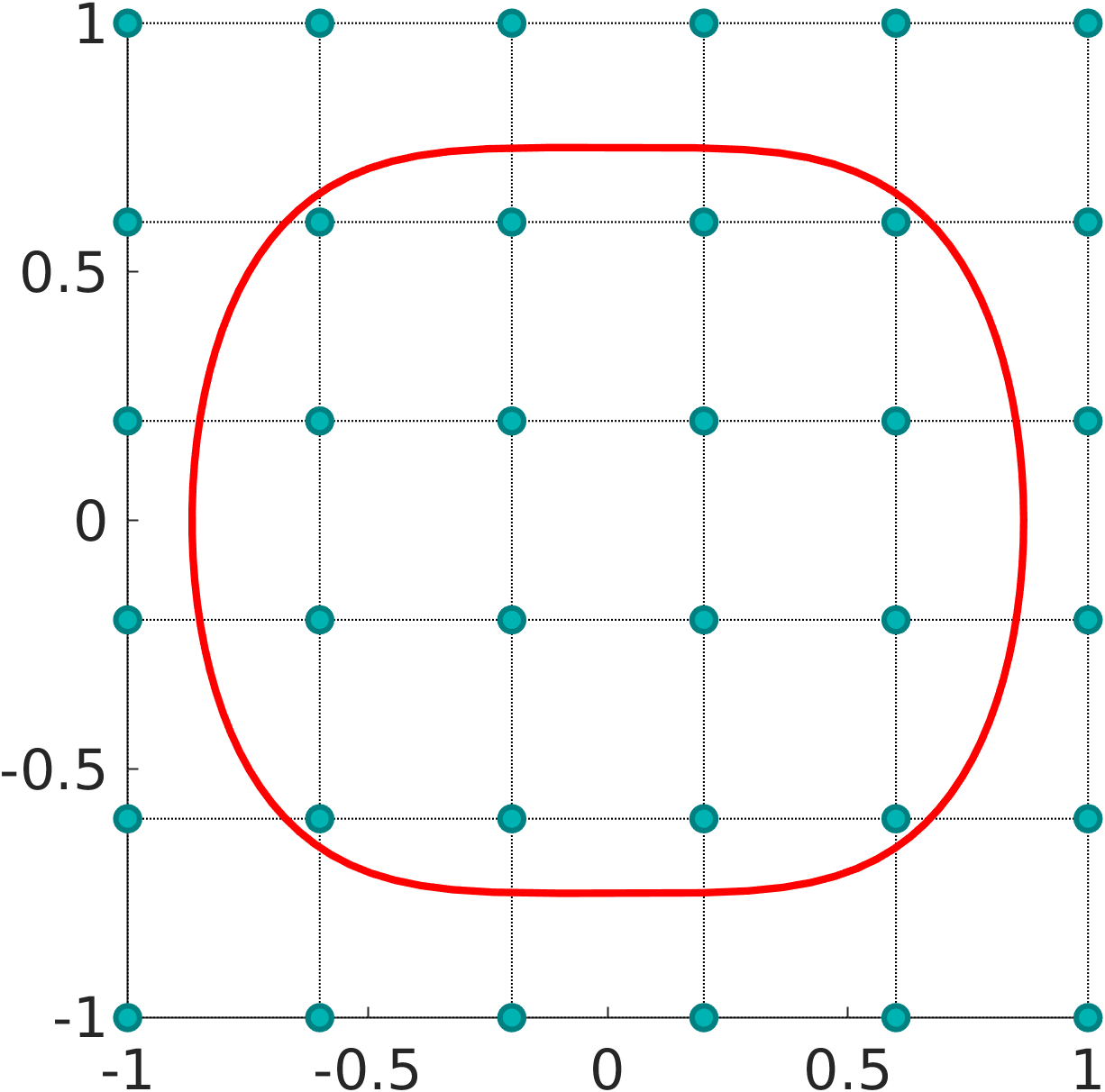}
    \caption{The smooth boundary with vanishing curvatures to order two in~\eqref{EQ: VANISH CUR EQ}.}
    \label{fig: vanish curvature}
\end{figure}
The weight function is $\theta_{\eps}^{\Delta}\in \cW_1$. Each experiment is performed with 32 random rigid transformations independently. The decay rates of $\cO(h^{5-4\alpha})$ and $\cO(h^{7-6\alpha})$ have been observed in Figure~\ref{fig: degen var decay 1} and Figure~\ref{fig: degen var decay 2} for different tube widths $\eps = \Theta(h^{\alpha})$, respectively.
\begin{figure}[!htb]
    \centering
    \includegraphics[scale=0.34]{ 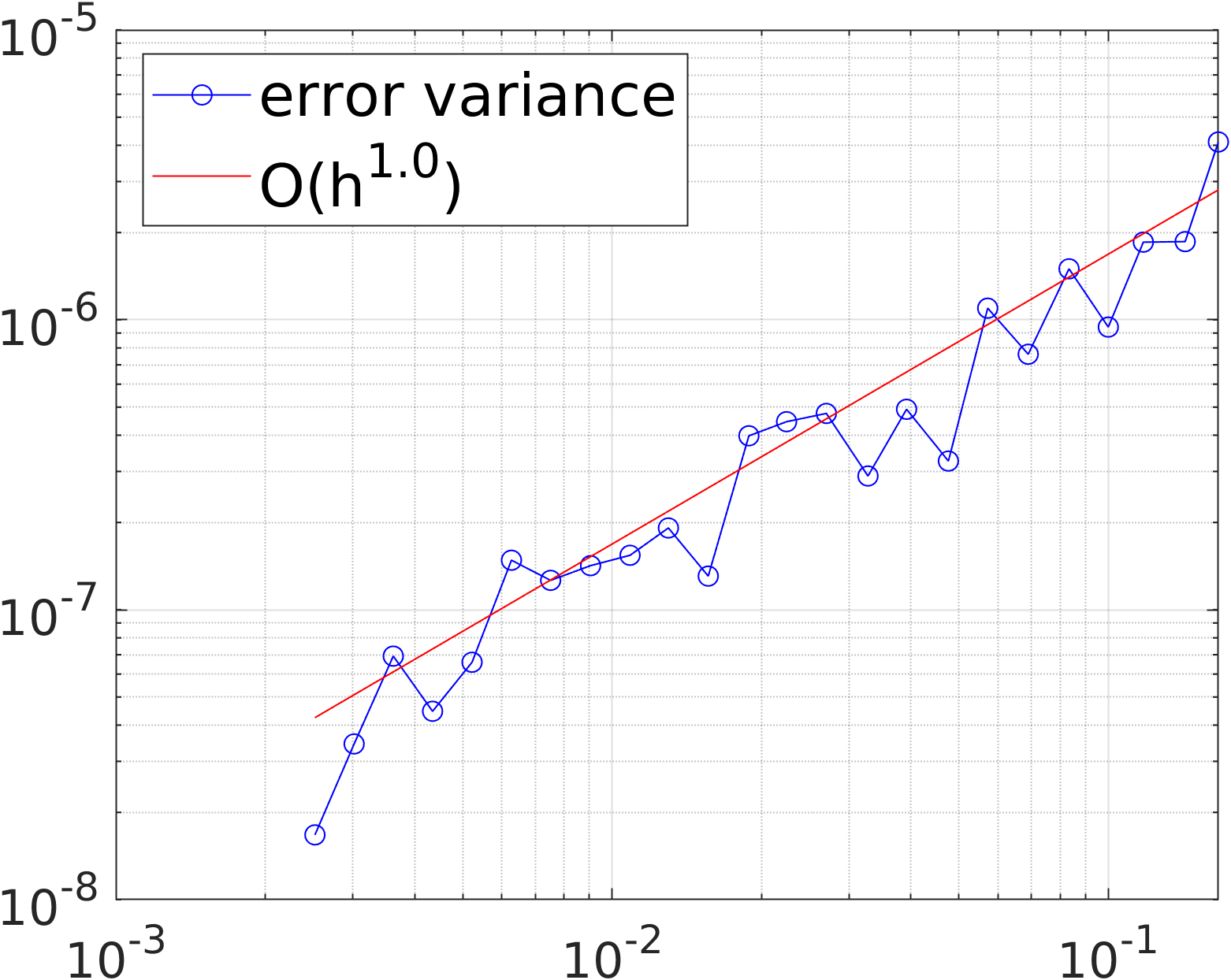}\qquad
    \includegraphics[scale=0.34]{ 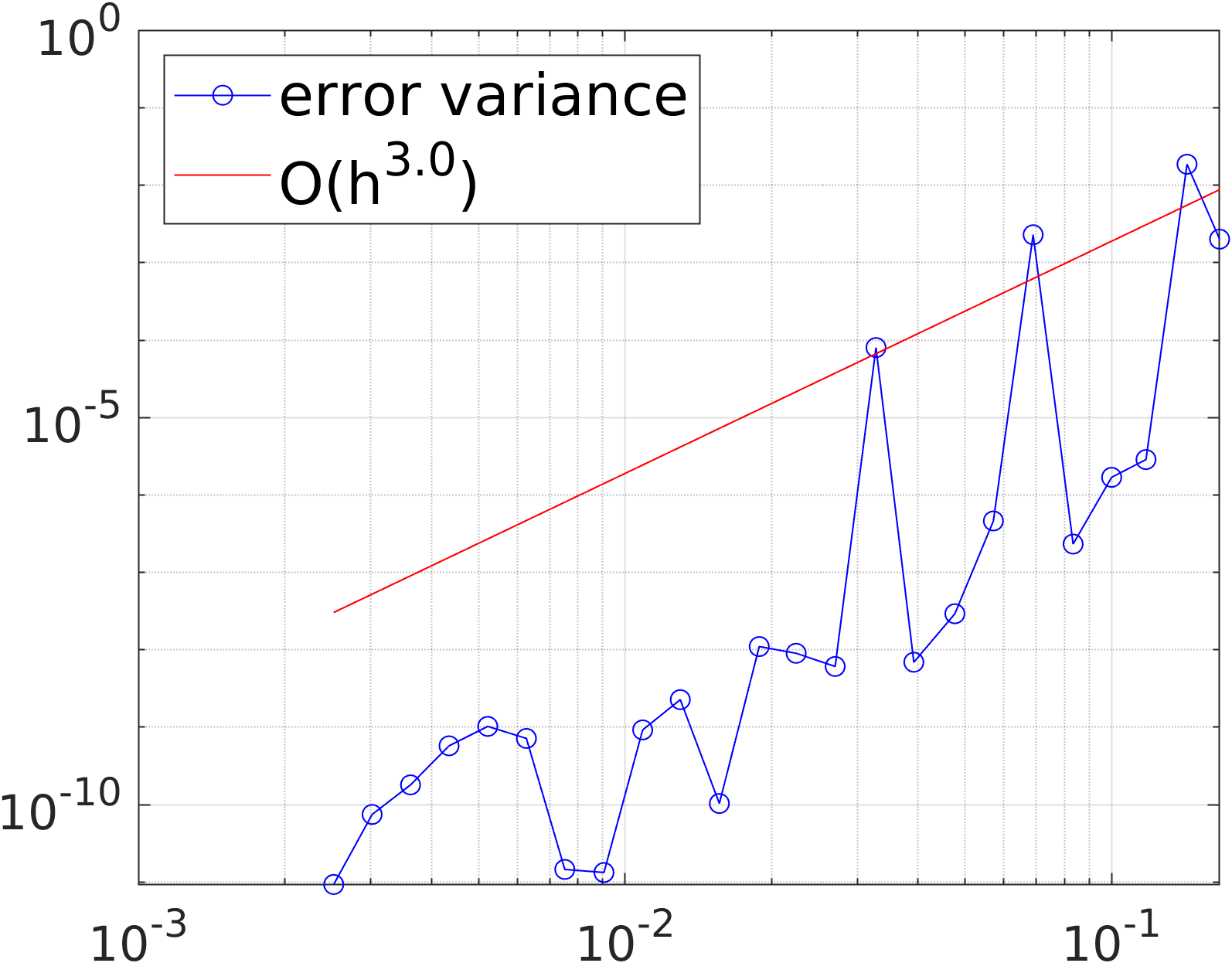}\qquad
    \includegraphics[scale=0.34]{ 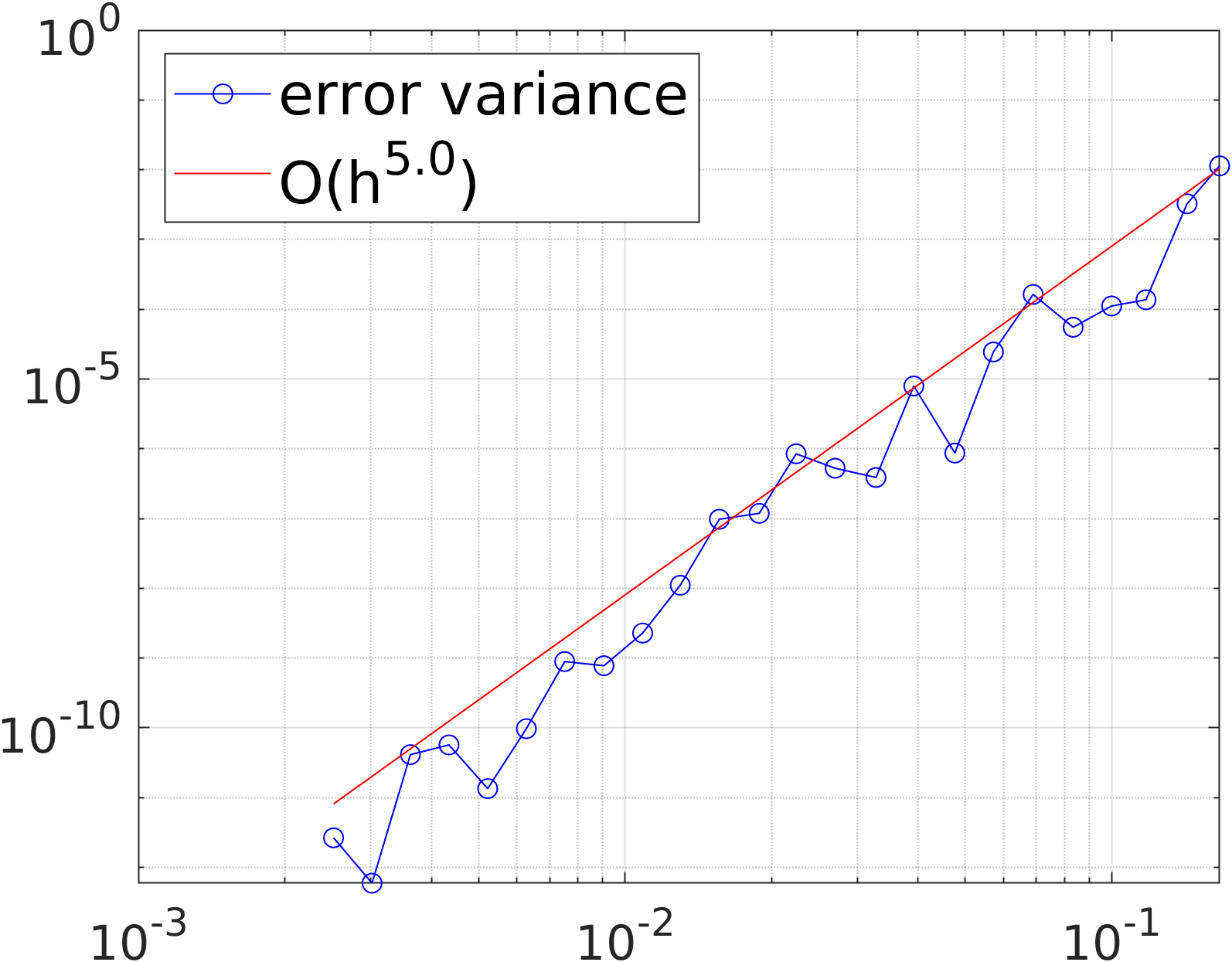}
    \caption{Variance of quadrature error for weight function $\theta_{\eps}^{\Delta}$. Left: $\eps = 2h$. Middle: $\eps = 2h^{\frac{1}{2}}$. Right: $\eps = 0.1$.}
    \label{fig: degen var decay 1}
\end{figure}
\begin{figure}[!htb]
    \centering
    \includegraphics[scale=0.34]{ 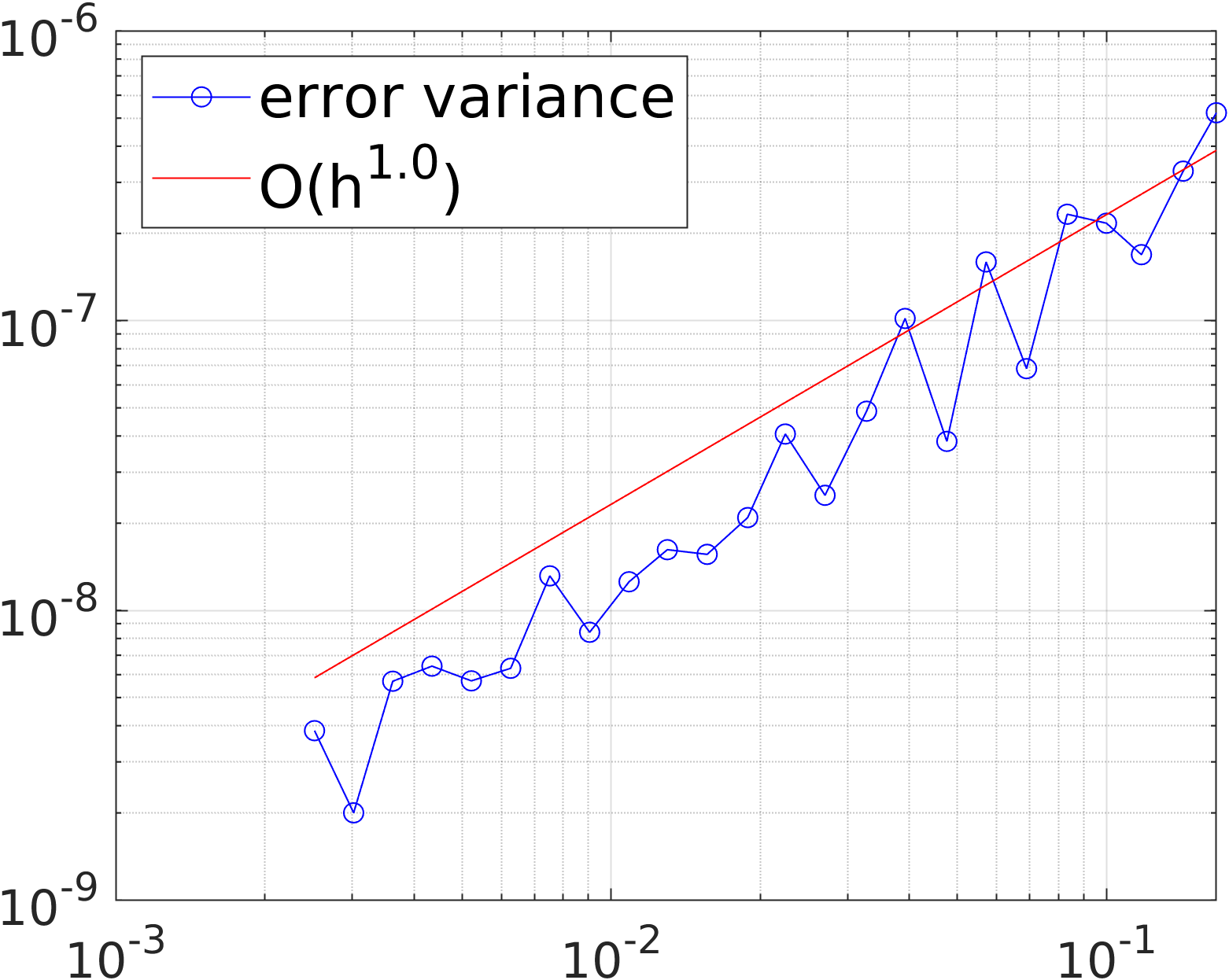}\qquad
    \includegraphics[scale=0.34]{ 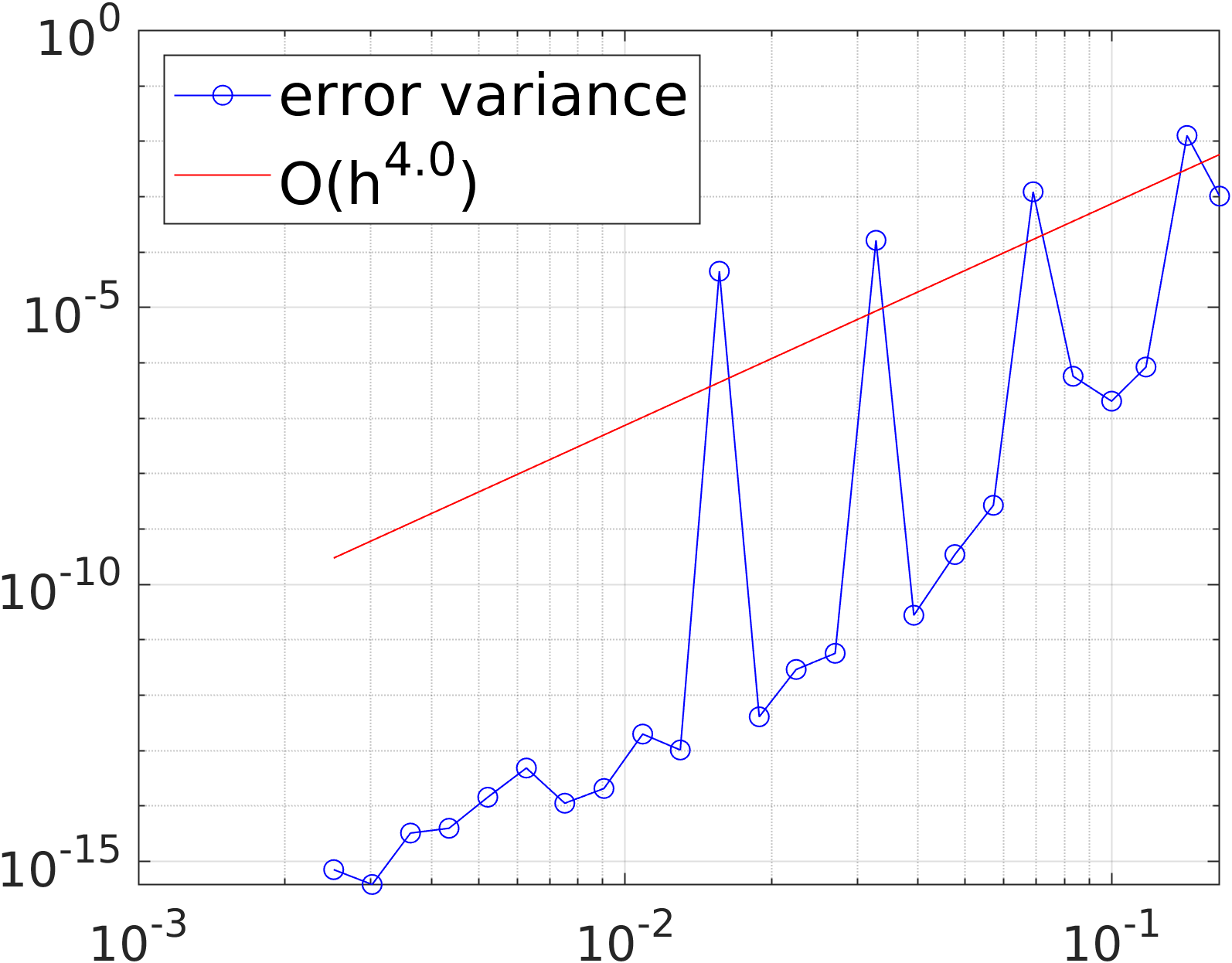}\qquad
    \includegraphics[scale=0.34]{ 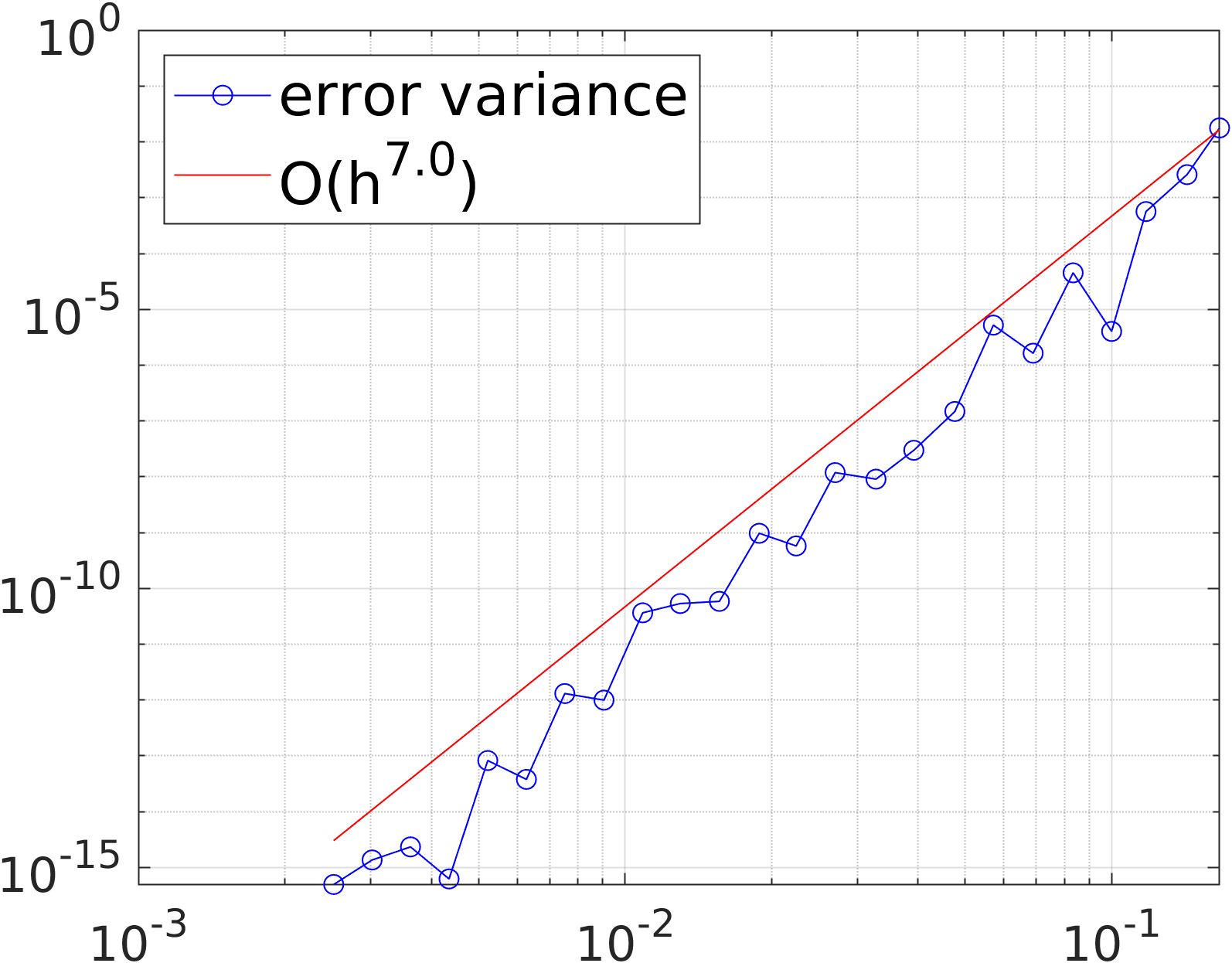}
    \caption{Variance of quadrature error for weight function $\theta_{\eps}^{\cos}$. Left: $\eps = 2h$. Middle: $\eps = 2h^{\frac{1}{2}}$. Right: $\eps = 0.1$.}
    \label{fig: degen var decay 2}
\end{figure}

\subsubsection{Star shape}\label{SEC: VAR STAR}
For the case of non-convex curves, we take the star-shaped curve in polar coordinates:
\begin{equation}\label{EQ: STAR EQ}
    \rho(\theta) = R + r \cos(m\theta)
\end{equation}
with parameters $R=0.75$, $r = 0.2$, and $m=3$; see Figure~\ref{fig: star-shape}.
\begin{figure}[!htb]
    \centering
    \includegraphics[scale=0.4]{ 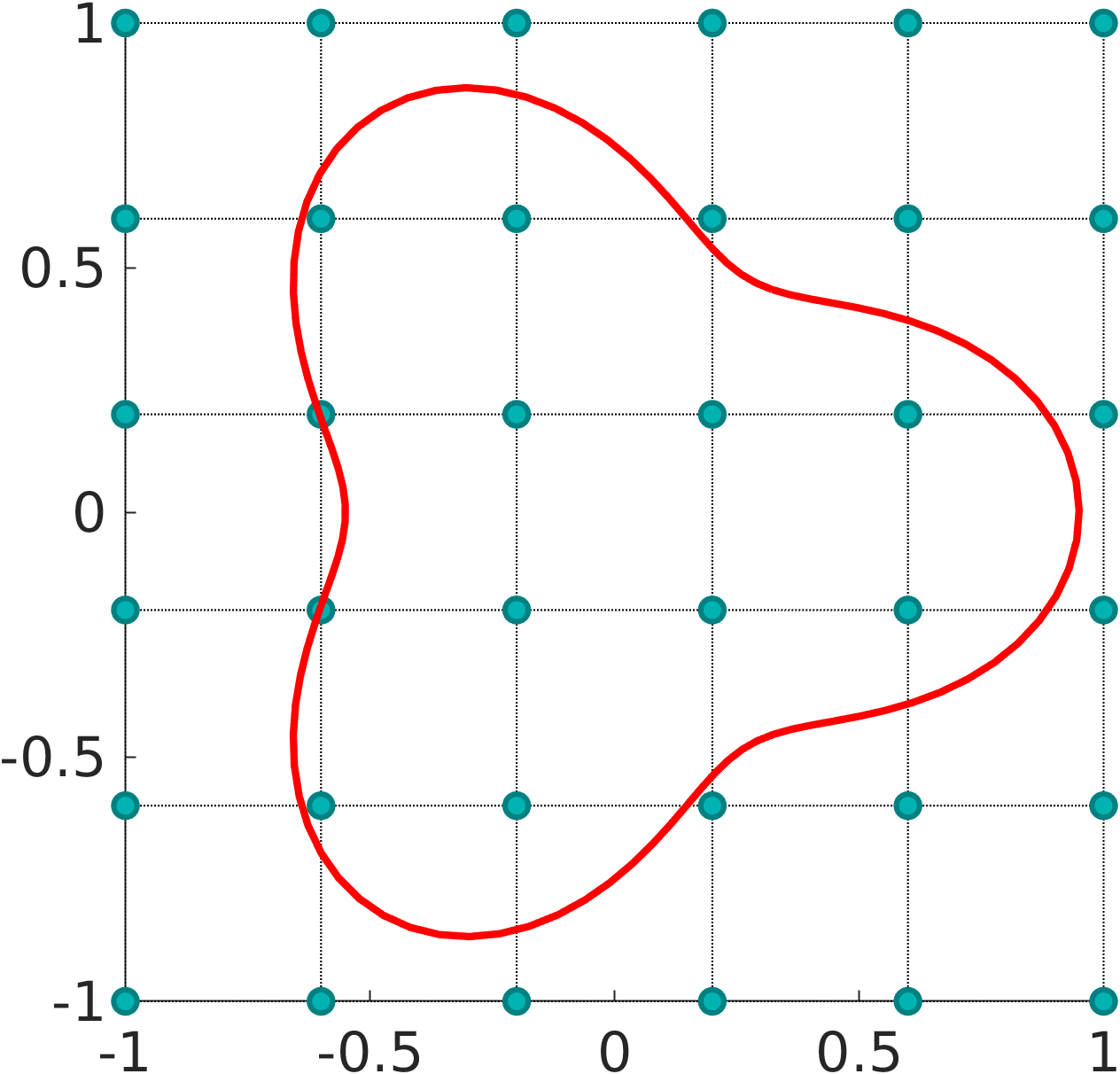}
    \caption{Star-shape boundary with parameter $R = 0.75$, $r = 0.2$, and $m=3$ in~\eqref{EQ: STAR EQ}.}
    \label{fig: star-shape}
\end{figure}
The curve is non-convex but consists of finitely many points with vanishing curvatures. We use it to verify the result in Corollary~\ref{COR: GENERAL CURVE}. The integrand function is selected as
\begin{equation*}
    f(x, y) = \cos(x^2 - y) \sin(y^2 - x^3).
\end{equation*}
The weight function is $\theta_{\eps}^{\Delta}\in \cW_1$. Each experiment is performed with 32 random rigid transformations independently. The decay rate of $\cO(h^{5-4\alpha})$ is observed in Figure~\ref{fig: star var decay 1}.
\begin{figure}[!htb]
    \centering
    \includegraphics[scale=0.34]{ 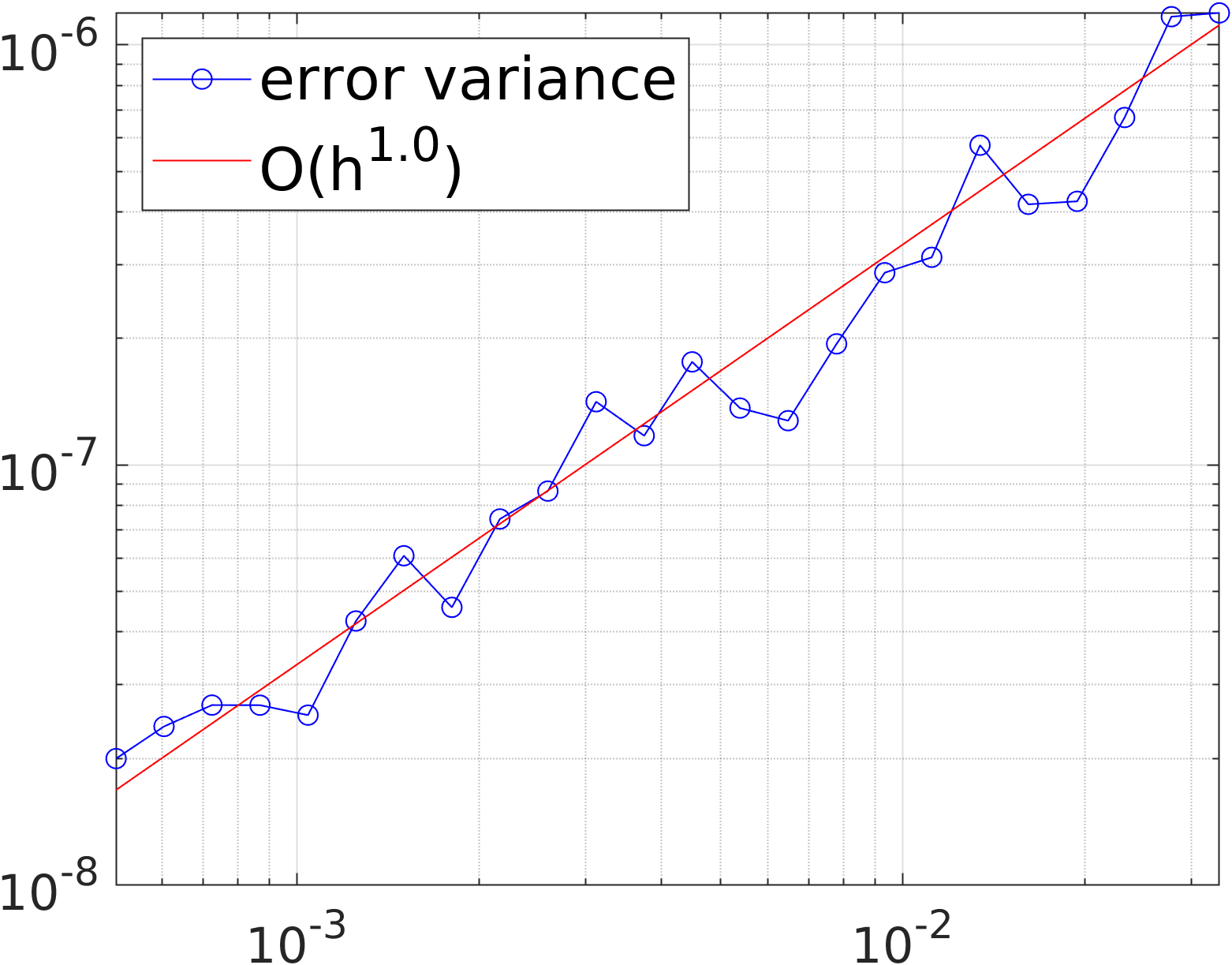}\qquad
    \includegraphics[scale=0.34]{ 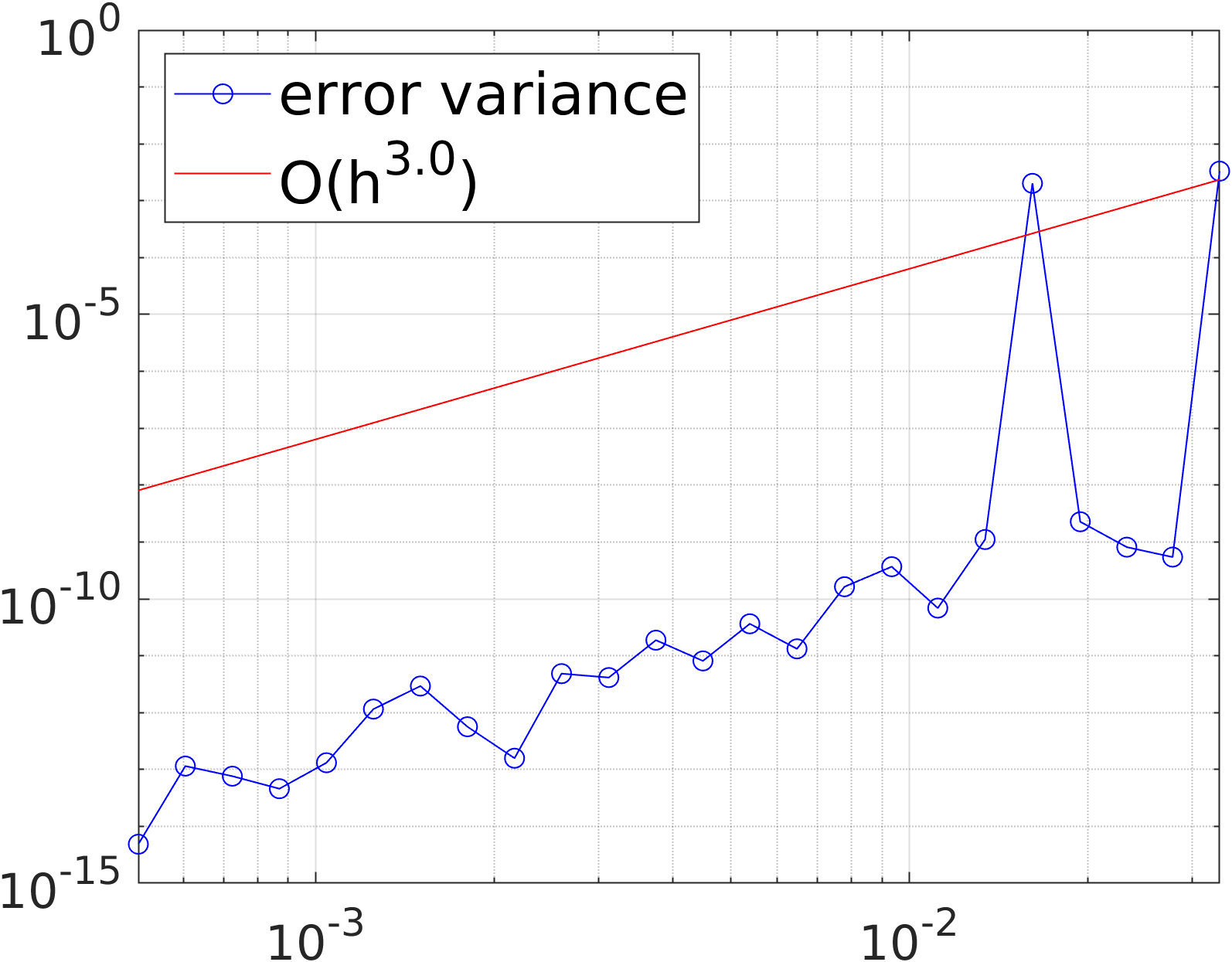}\qquad
    \includegraphics[scale=0.34]{ 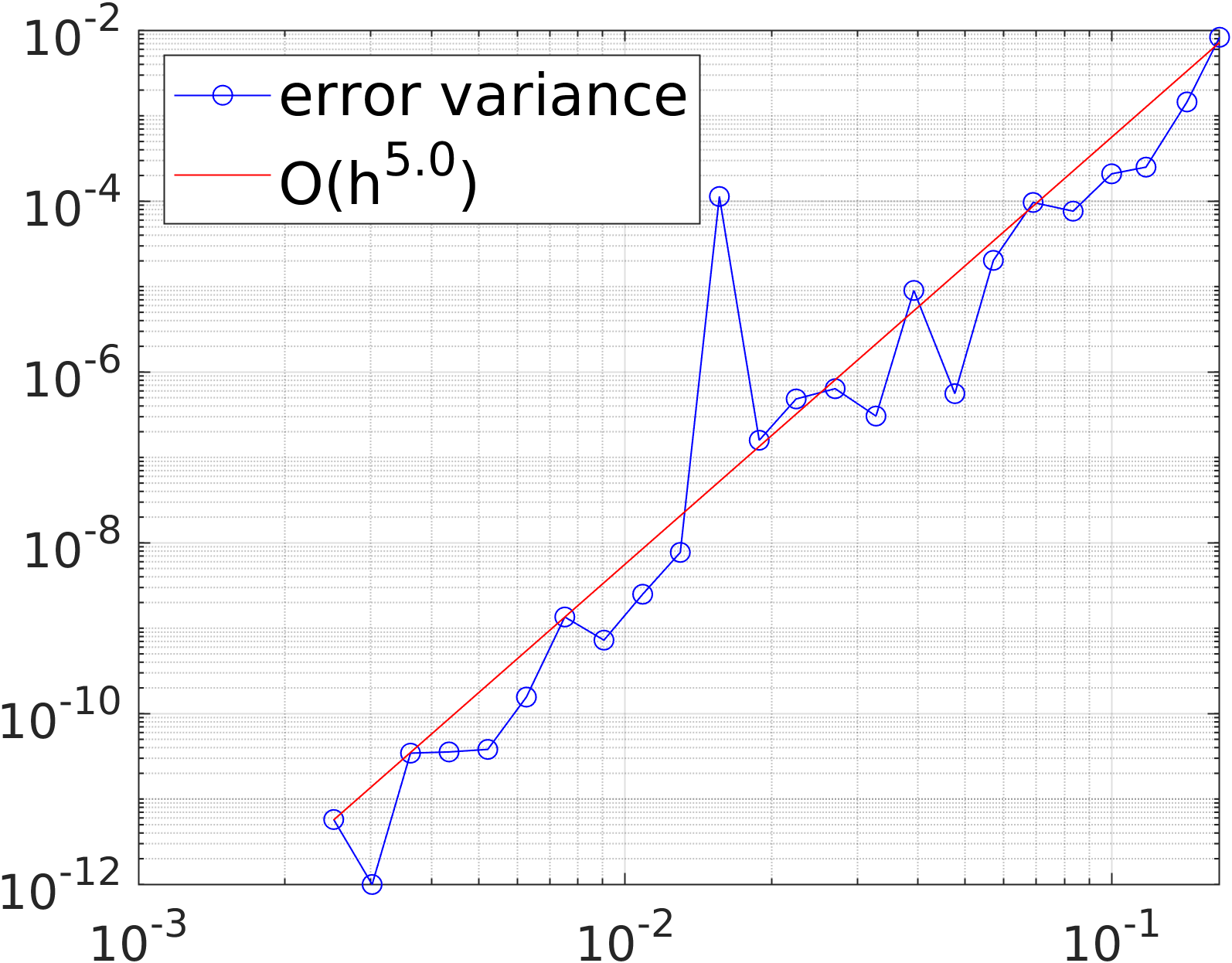}
    \caption{Variance of quadrature error for weight function $\theta_{\eps}^{\Delta}$. Left: $\eps = 2h$. Middle: $\eps = 2h^{\frac{1}{2}}$. Right: $\eps = 0.1$.}
    \label{fig: star var decay 1}
\end{figure}

\section{Further discussions on open curves}
\label{SEC: 4}

In this section, we consider the quadrature error on open curves instead of closed ones. For technical reasons in the derivation of the theory, mainly the fact that in the case of open curves, we cannot construct a partition of unity with $C_0^{\infty}$ functions, we have only weaker results compared to the ones in Section~\ref{SEC: 2} and Section~\ref{SEC: 3}. 

We consider two groups of curves: (i)  curves with finitely many inflection points (curvature zero) and (ii) line segments.

\subsection{General curves}
 Let $\Gamma$ be a collection of curves that $ \Gamma = \bigcup_{k=1}^m \Gamma_k$,
where each curve $\Gamma_k\in C^{\infty}$ has only a finite number of points with vanishing curvature, whose maximal order is $(\kappa - 2)$, $\kappa \in (2, \infty)$. Let $\cQ_k(\bx):= f(P_{\Gamma_k}\bx)\theta_{\eps}(d_{\Gamma_k}(\bx)) J_{\eps}(\bz, d_{\Gamma_k}(\bx))$, we define generalized implicit boundary integral on the collection $\Gamma$ as 
\begin{equation}
\begin{aligned}
    \cI_h (f) := h^2 \sum_{k=1}^m \sum_{\bz\in (h\bbZ)^2\cap T_{k, \eps}} \cQ_k(\bz),
\end{aligned}
\end{equation}
where $T_{k,\eps} = \{\bx + t\bn_k(\bx) \mid (\bx, t) \in \Gamma_{k}\times [-\eps, \eps]\}$ is the segmented tube associated with $\Gamma_k$ and $\bn_k(\bx)$ is the unit normal vector at $\bx\in\Gamma_k$. Clearly, the smooth closed boundary in Corollary~\ref{COR: GENERAL CURVE} is a special case for $m=1$. 
\begin{theorem}\label{LEM: GENERAL CURVES}
For any rotation $\eta\in \mathrm{SO}(2)$ and translation $\bxi\in [0, 1]^2$, we denote the generalized implicit boundary integral with tube width $\eps=\Theta(h^{\alpha})$ on the transformed curve collection $\eta \Gamma + h \bxi$ as $\cI_h(f; \eta, \bxi)$, then if there are no points with vanishing curvature or there exist finitely many points with vanishing curvature to the maximal order $\kappa - 2$ for $\kappa\in [2, \frac{3+\sqrt{5}}{2})$, then 
$$\int_{\mathrm{SO(2)}}\int_{[0, 1]^2} |\cI(f) - \cI_h(f; \eta, \bxi)|^2 d\bxi d\eta  = 
    \cO(h^{\frac{2}{\kappa} + {\frac{2}{\kappa - 1}(1-\alpha)} }).
    $$
\end{theorem}
\begin{proof}
Without loss of generality, we may assume $\Gamma_k$ has at most one point with vanishing curvature at its endpoint, otherwise one can split the curve into shorter pieces. Let
       $$\Gamma_k \coloneqq \{(x, g_k(x)) \mid x\in [0, r]\}$$
     such that $g_k(x) = |x|^{\kappa} h(x)$, $h\in C^{\infty} [0,r]$ and $h(x) > 0$ for all $x\in [0,r]$. 
    When $\kappa=2$, then there are no points with vanishing curvature.  Using the derivation in Appendix~\ref{PRF: STATION}, the Fourier transform of $\cQ_k \chi_{T_{k,\eps}}$ becomes 
    \begin{equation}\label{EQ: FUBINI}
    \begin{aligned}
    \widehat{\cQ_k \chi_{T_{k,\eps}}}(\bzeta) &= \int_{T_{k,\eps}} \cQ_k (\bx) e^{-2\pi i \bx\cdot \bzeta} d\bx \\
    &= \int_{-\eps}^{\eps}  \theta_{\eps}(s)  \int_{\Gamma_k} f(\bx') e^{-2\pi i (\bx'+s \bn(\bx'))\cdot \bzeta}  d\sigma  ds\\
    &=\int_{-\eps}^{\eps}  \theta_{\eps}(s)  \int_{0}^r f(x) e^{-2\pi i \phi(x, \zeta_1, \zeta_2, s)}  \sqrt{1 + |g'_k(x)|^2}dx  ds 
    \end{aligned}
    \end{equation}
    where $d\sigma$ is the Lebesgue measure on $\Gamma_k$, $\chi_{T_{k,\eps}}$ denotes the characteristic function on $T_{k,\eps}$, and 
    \begin{equation*}
        \phi(x, \zeta_1, \zeta_2, s) =  x \zeta_1 + g_k(x) \zeta_2 + s \frac{-g'_k(x)\zeta_1 + \zeta_2}{\sqrt{1 + |g'_k(x)|^2}}\,.
    \end{equation*}
    Let $\{\chi_l\}_{l\ge 1}$ be a partition of unity over the interval $[0, r]$. It is clear that we only have to estimate $\widehat{\cQ_k \chi_{T_{k,\eps}}}$ for each integrand $f(x)\chi_{l}$ instead of $f(x)$. Therefore, we may further assume $f(r) = 0$. Let us define the auxiliary function $G(x, s)$ that 
    \begin{equation*}
        G(x, s) = \int_0^x e^{-2\pi i \phi(t, \zeta_1, \zeta_2, s)} dt\,.
    \end{equation*}
    Then, using integration by parts on~\eqref{EQ: FUBINI}, as well as the the assumption that $f(r) = 0$, we obtain
    \begin{equation}\label{EQ: CQK}
    \begin{aligned}
    \widehat{\cQ_k \chi_{T_{k,\eps}}}(\bzeta) &=\int_{-\eps}^{\eps}  \theta_{\eps}(s)  \int_{0}^r f(x) e^{-2\pi i \phi(x, \zeta_1, \zeta_2, s)}  \sqrt{1 + |g'_k(x)|^2}dx  ds \\
    &=\int_{-\eps}^{\eps}  \theta_{\eps}(s)  \left[ -\int_{0}^r  G(x, s)  \frac{d}{dx}\left( f(x)\sqrt{1 + |g'_k(x)|^2} \right) dx  \right] ds\,.
    \end{aligned}
    \end{equation}
    Let $\bx' = (x, g_k(x))$ and $\bn(\bx') = (\sin\psi(x), \cos \psi(x))$, then $\psi(x) = -\arctan(g'(x)) = \Theta(|x|^{\kappa - 1})$ and $\psi\in[-\psi_0, 0]$ for certain $\psi_0\in (0, \frac{\pi}{2})$.
    In the following, we use the notation $\bzeta = |\bzeta|(\sin\delta, \cos\delta)$ and estimate the integral in the following form
    \begin{equation}\label{EQ: EJ}
      \cE :=  \int_{-\eps}^{\eps}  \theta_{\eps}(s) \int_{0}^{r}  G(x,s) \frac{d}{dx}\left( f(x)\sqrt{1 + |g'_k(x)|^2} \right) dx ds .
    \end{equation}
    We split the integral in $x$ into two different regimes.\\[1ex]   
\noindent {\bf Case 1.} On $\cI_1 := \{x\in [0, r] \mid  |\cos(\psi(x) - \delta) |\le | \eps\bzeta|^{-1}\} $, we can use Lemma~\ref{LEM: van der Corput Revised} to find that $|G(x,s)| = \cO(|\bzeta|^{-1/\kappa})$, and the contribution on $\cI_1$ is then
        \begin{equation*}
        \begin{aligned}
     \cE_1 &= \int_{-\eps}^{\eps}  \theta_{\eps}(s) \int_{\cI_1} G(x,s) \frac{d}{dx}\left( f(x)\sqrt{1 + |g'_k(x)|^2} \right) dx ds  = \cO\left(|\cI_1||\bzeta|^{-1}\right) \\
     &=\begin{cases}
         \cO\left( (\eps |\bzeta|)^{\frac{-1}{\kappa - 1}} |\bzeta|^{-1/\kappa} \right)  \quad& \text{ if } |\sin(\delta - \frac{\pi}{2})|\le (\eps |\bzeta|)^{-1}\,,\\ 
          \cO\left( |\sin(\delta - \frac{\pi}{2})|^{-\frac{\kappa - 2}{\kappa - 1}} (\eps|\bzeta|)^{-1} |\bzeta|^{-1/\kappa} \right) \quad & \text{ if }  |\sin(\delta - \frac{\pi}{2})| > (\eps |\bzeta|)^{-1}\,.
     \end{cases}
        \end{aligned}
        \end{equation*}
        \noindent {\bf Case 2.} On $\cI_2 := \{x\in [0, r]\mid  |\cos(\psi(x) - \delta)| \in[  | \eps\bzeta|^{-1} , 1] \} $, we use integration by parts $(q+1)$ times on~\eqref{EQ: EJ}. The contribution on $\cI_2$ can then be estimated as
        \begin{equation*}
        \begin{aligned}
        \cE_2
        &= - \sum_{l=1}^L \int_{\cI_2}  \left[\frac{1}{2\pi i \bn(\bx')\cdot \bzeta}  \right]^{q+1} \left( G(x, s) \frac{d^q}{ds^q} \theta_{\eps} \Big|_{s_l^{+}}^{s_l^{-}} \right) \frac{d}{dx}\left[f(x)\sqrt{1 + |g'_k(x)|^2}  \right] dx  \\
       & \quad +  \sum_{l=1}^L \int_{s_l}^{s_{l+1}} \frac{d^{q+1}}{ds^{q+1}} \theta_{\eps}(s)\int_{\cI_2} \left[\frac{1}{2\pi i \bn(\bx')\cdot \bzeta}  \right]^{q+1} G(x,s)\frac{d}{dx}\left[ f(x)  \sqrt{1 + |g'_k(x)|^2}\right]  dx ds.
        \end{aligned}
       \end{equation*}
       Since $|\bn(\bx')\cdot \bzeta| = |\bzeta||\cos(\psi(x) - \delta)|$, we can then apply the estimate $G(x,s) = \cO(|\bzeta|^{-1/\kappa})$ again. This leads to
        \begin{equation*}
        \begin{aligned}
       \cE_2 &= \cO\left(\eps^{-(q+1)} |\bzeta|^{-(q+1 + 1/\kappa)}\right) \int_{\cI_2} \frac{1}{|\cos (\delta - \psi(x))|^{q+1}} dx \\
       &=  \cO\left(\eps^{-(q+1)} |\bzeta|^{-(q+1+1/\kappa)}\right) \int_{(\eps |\bzeta|)^{-1}}^{\frac{\pi}{2}} \frac{|\frac{\pi}{2} - \delta - \tau|^{-\frac{\kappa - 2}{\kappa - 1}}} {|\sin \tau |^{q+1}} d\tau  \quad (\text{let } \tau = \frac{\pi}{2} - \delta + \psi(x) )\\
       &\le  \cO\left(\eps^{-(q+1)} |\bzeta|^{-(q+1+1/\kappa)}\right) \int_{(\eps |\bzeta|)^{-1}}^{\frac{\pi}{2}} \frac{| (\eps |\bzeta|)^{-1} - \tau|^{-\frac{\kappa - 2}{\kappa - 1}}} {|\sin \tau |^{q+1}} d\tau \\
       &=  \cO\left( |\eps\bzeta|^{-\frac{1}{\kappa - 1}} |\bzeta|^{-\frac{1}{\kappa}}\right)\,.
        \end{aligned}
         \end{equation*}
The above calculations show that
\begin{equation}\label{EQ:CQK 1-2}
\begin{aligned}
     \int_0^{2\pi} |\cE_1|^2 d\delta &= \cO\left( |\bzeta|^{-2/\kappa} \left[ |\eps\bzeta|^{-2}+ |\eps\bzeta|^{-\frac{\kappa + 1}{\kappa - 1}}\right] \right)\,,\\
     \int_0^{2\pi} |\cE_2|^2 d\delta &= \cO(|\eps\bzeta|^{-\frac{2}{\kappa - 1}} |\bzeta|^{-2/\kappa})\,.
\end{aligned}
\end{equation}
To estimate~\eqref{EQ: CQK}, we observe that
\begin{equation}\nonumber
\begin{aligned}
     \int_{\textrm{SO(2)}} |\widehat{\cQ_k \chi_{T_{k,\eps}}}(\eta^{-1}\bzeta) |^2 d\eta &= \frac{1}{2\pi} \int_0^{2\pi} |\cE_1 + \cE_2|^2 d\delta \\&\le \frac{1}{4\pi} \int_0^{2\pi} \left( |\cE_1|^2 + |\cE_2|^2\right) d\delta\,.
\end{aligned}
\end{equation}
This, together with~\eqref{EQ:CQK 1-2}, gives
\begin{equation}\nonumber
      \int_{\textrm{SO(2)}} |\widehat{\cQ_k \chi_{T_{k,\eps}}}(\eta^{-1}\bzeta) |^2 d\eta = \begin{cases}
\cO(|\eps\bzeta|^{-2}|\bzeta|^{-\frac{2}{\kappa}} + |\eps\bzeta|^{-\frac{2}{\kappa - 1}}|\bzeta|^{-\frac{2}{\kappa}})\quad &\text{if }\kappa \in [2, 3]\,,\\
\cO(|\eps\bzeta|^{-\frac{\kappa+1}{\kappa - 1}}|\bzeta|^{-\frac{2}{\kappa}} + |\eps\bzeta|^{-\frac{2}{\kappa - 1}}|\bzeta|^{-\frac{2}{\kappa}})\quad& \text{if } \kappa > 3\,.
      \end{cases}
\end{equation}
Since $\frac{2}{\kappa - 1} \le \frac{\kappa + 1}{\kappa - 1}$ and $\frac{2}{\kappa - 1}\le 2$, we only need the second term. The variance of quadrature error is therefore bounded by 
\begin{equation}\nonumber
\begin{aligned}
    \sum_{\bzeta\in \bbZ^2 - \{\bzero\}} \int_{\textrm{SO(2)}} |\widehat{\cQ_k \chi_{T_{k,\eps}}}(\eta^{-1}h^{-1}\bzeta) |^2 d\eta &= \cO\left(\int_{\rho\ge 1} (\eps h^{-1}\rho )^{-\frac{2}{\kappa - 1}} (h^{-1}\rho)^{-\frac{2}{\kappa}} \rho d\rho\right) \\
        &=\cO(h^{\frac{2}{\kappa} + \frac{2}{\kappa - 1}(1-\alpha)}) \,,
        \end{aligned}
\end{equation}
where we need $\kappa<\frac{3 + \sqrt{5}}{2}\approx 2.618$ for the integral to be finite. This finishes the proof.
\end{proof}
We emphasize that the result in Theorem~\ref{LEM: GENERAL CURVES} is valid for smooth open curves as well as closed ones. However, as we commented before, the above estimates are not as strong as the one in Corollary~\ref{COR: GENERAL CURVE} for smooth closed boundaries with any $\kappa \in [2, \infty)$, due to the fact that in the current case, when cannot assume that the partition of unity we have is a smooth and vanishing-on-support partition. It is, however, possible to generalize the estimate in Theorem~\ref{LEM: GENERAL CURVES} to larger $\kappa$ by a careful computation of~\eqref{EQ: CQK}. In particular, as we will see later in the numerical experiments in Section~\ref{SEC: EXP 4}, the estimates in Theorem~\ref{LEM: GENERAL CURVES} appear to the sharp for curves without inflection points ($\kappa = 2$). However, for $\kappa > 2$, the estimate in Theorem~\ref{LEM: GENERAL CURVES} may still be far from being optimal. When the tube width $\eps=\Theta(h)$, the estimates in Theorem~\ref{LEM: GENERAL CURVES} and Corollary~\ref{COR: GENERAL CURVE} coincide for $\kappa = 2$, which implies open and closed curves may have the same variance of quadrature errors $\cO(h)$. 




\subsection{Segments with irrational slope}


The illustration of Figure~\ref{fig: straight line} shows that the worst scenario for straight segments is $\cO(1)$. There is, however, a dense subset of slopes that the quadrature error could be much better. We now show that if the slope of the straight segment is quadratically irrational (which is dense in $\bbR$), then the quadrature error can be reduced to $\cO(h^{2-\alpha}|\log h|)$ for a width of $\eps = \cO(h^{\alpha})$. The proof is based on the concept of low-discrepancy sequences. Let us first recall some of the necessary tools.

\begin{definition}
    The discrepancy of a set $P = \{\bx_1, \dots, \bx_N\}\subset \bbR^d$ is defined by 
    \begin{equation}
        D_{N}(P) = \sup_{B\in J} \left| \frac{|P\cap B| }{N} - \mathrm{meas}(B)\right|
    \end{equation}
    where $\mathrm{meas}(A)$ is the Lebesgue measure of $A$,  $J$ is the set of $d$-dimensional boxes of the form $  \prod_{i=1}^d [a_i, b_i)$, $0\le a_i < b_i \le 1$. 
\end{definition}
The following Theorem~\ref{THM: LOW DIS} shows that a set with small discrepancy serves as a quadrature rule of equal weights for functions of bounded variation. This quadrature rule is also called the quasi-Monte-Carlo method~\cite{bertrandias1960calcul,zaremba1966good,sloan1994lattice,hlawka1962angenaherten}, and it outperforms the standard Monte-Carlo methods in accuracy if $\{\bx_i\}_{i\ge 1}$ is a low-discrepancy sequence~\cite{kuipers2012uniform}.
\begin{theorem}[Koksma-Hlawka Inequality~\cite{brandolini2013koksma}]\label{THM: LOW DIS}
Let $f$ have bounded variation $V(f)$, then 
\begin{equation}\label{EQ: QMC}
    \left|\frac{1}{N}\sum_{i=1}^N f(\bx_i) - \int_{[0, 1]^d} f(\bx) d\bx \right| \le V(f) D_N(\bx_1,\dots, \bx_N).
\end{equation}
\end{theorem}
However, the general $d$-dimensional  lattice $h\bbZ^d$ only has a discrepancy of $\cO(h)$. Hence, a direct application of the Koksma-Hlawka Inequality to the weight function $\theta_{\eps}(\bx)$ only implies an error bound of $\cO(\frac{h}{\eps})$. We will show below in Theorem~\ref{THM: LINE INT} that the error can improve if the lattice points are chosen appropriately. To choose such lattice points, we need the following result.
\begin{lemma}[{\cite[Theorem 2.3.3]{huxley1996area}}]\label{LEM: DIS}
    Let $R\subset\bbR^2$ be a convex polygon $P_1P_2\dots P_n$, $P_n = P_0$. Let the slope of side $P_{i-1} P_{i}$ be $\alpha_{i}$ and assume that $\alpha_i$ can be written in continued fraction  
    \begin{equation*}
        \alpha_{i} = a_{0, i} + 1/(a_{1, i} + \cdots ), 
    \end{equation*}    
    with convergents $\frac{r_{k, i}}{q_{k, i}}$, $k\in\bbN$. Then the difference between the lattice point count inside $R$ and area of $R$ is bounded by $\sum_{i=1}^n \rho(P_{i-1}P_i)$,
    where $\rho(P_{i-1}P_i) = \sum_{s=0}^k a_s + \frac{|P_{i-1}P_i|+1}{q_{k,i}}$, $|P_{i-1}P_i|$ being the length of the side, and $k$ being the largest integer that $q_{k,i}\le |P_{i-1}P_i|+1$.
\end{lemma}
This result allows us to show the following.
\begin{corollary}\label{COR: DIS}
    Let $R$ be a rectangle with side lengths $s$ and $t$, and the corresponding slopes $\alpha$ and $\beta$. If the terms of continued fraction of $\alpha$ and $\beta$ are uniformly bounded, then $||R\cap h\bbZ^d| - h^{-2} st |$ is bounded by $\cO(1) + \cO(\log(h^{-1}s+1)) + \cO(\log(h^{-1}t+1))$.
\end{corollary}
\begin{proof}
For continued fraction of $\alpha$ (and the same for $\beta$), the convergents $r_k/q_k$ satisfy $q_{k+2} > 2 q_k$. Therefore, $k \le \frac{2\log (s+1)}{\log 2} + 1$; see~\cite{huxley1996area} for more detailed discussion. The conclusion then follows directly from Lemma~\ref{LEM: DIS}. A similar conclusion can be found in~ \cite[Theorem A3]{hardy1922some}. 
\end{proof}

We are now ready to prove the first main result of this section.
\begin{theorem}\label{THM: LINE INT}
Let $\Gamma\subset \bbR^2$ be a segment with a quadratically irrational slope. Then the quadrature error of integrating the function $f\in C(\Gamma)$, with tube width $\eps = \Theta(h^{\alpha})$ ($\alpha \in [0, 1]$), is given as
\begin{equation}
        | \cI(f) - \cI_h(f) |= \cO( h^{2-\alpha}\log(h))\,,
\end{equation}
where $\cI_h(f)$ is defined as 
\begin{equation}\nonumber
       \cI_h (f) := h^2\sum_{\bz\in (h\bbZ)^2\cap T_{\eps}} f(\bz) \theta_{\eps}(\bz),    
\end{equation}
with $T_{\eps} = \{\bx + t \bn\mid (\bx, t) \in \Gamma\times [-\eps, \eps]\}$ and $\bn$ being the unit normal vector for the segment.
\end{theorem}
\begin{proof}
    For simplicity, let $\Gamma$ be the segment sitting on $x$-axis and rotate, instead, the lattice to $p\bv + q\bv^{\perp}$ where $\bv = (v_1, v_2)$ is a unit vector and $(p,q)\in\bbZ^2$. We denote by $R := \Gamma\times [-\eps, \eps]$ the rectangular tube around $\Gamma$. The quadrature becomes
    \begin{equation*}
       \cI_{h} (f) = h^2\sum_{(p, q)\in\bbZ^2} \chi_{\Gamma} f( h(p\bv + q\bv^{\perp})\cdot \be_1 )\theta_{\eps}( h(p\bv + q\bv^{\perp})\cdot \be_2 )\,, 
    \end{equation*}
    where $\chi_{\Gamma}$ is the characteristic function on $\Gamma$. Let $P:= R \cap \{h(p\bv+q\bv^{\perp})\mid (p,q)\in\bbZ^2\}$, and denote by $N = |P|$. Then the discrepancy of $P$ (with scaling) is
    \begin{equation*}
    \begin{aligned}
            D_{N}(P) &= \sup_{B\in J} \left| \frac{|P \cap B|}{N} \text{meas}(R) - \mathrm{meas}(B)\right|,\quad J = \left\{  [a_1, b_1)\times [a_2, b_2)\subset R\right\}\,.
    \end{aligned}
    \end{equation*}
    Since the continued fraction of a quadratically irrational number is periodic, it is uniformly bounded. By Corollary~\ref{COR: DIS},  $N = h^{-2}\text{meas}(R) + \cO(\log(h) )$. We can derive that 
\begin{equation}\nonumber
    \begin{aligned}
        D_N(P) &\le \sup_{B\in J}\left| \frac{h^{-2}\mathrm{meas}(B) + \cO(1) + \cO(\log({h}^{-1}\mathrm{diam}(B)+1))}{N}\text{meas}(R) - \mathrm{meas}(B)\right|\\
        &\le \left| \frac{{h}^{-2}\text{meas}(R)}{N} - 1\right| + \sup_{B\subset J}\left|\frac{\cO(1) + \cO(\log({h}^{-1}\mathrm{diam}(B)+1))}{N}\text{meas}(R)\right|\\
        &= \frac{\cO({h}^2 \log({h}^{-1}))}{\mathrm{meas}(R)}\,.
    \end{aligned}
\end{equation}
By Lemma~\ref{LEM: DIS}, we have
\begin{equation}\label{EQ: CELL EST}
    \left|\frac{1}{N}\sum_{\bz\in P} f(\bz)\theta_{\eps}(\bz) - \frac{1}{\mathrm{meas}(R)}\int_{R} f(\bx) \theta_{\eps}(\bx) d\bx \right| \le D_{N}(P) V_{R}(f\theta_{\eps})\,,
\end{equation}
where $V_{R}(f\theta_{\eps}) = \cO(\eps^{-1})$ is the total variation of $f\theta_{\eps}$ on $R$. Multiplying $N h^2$ to~\eqref{EQ: CELL EST} then gives us
\begin{equation}\nonumber
\begin{aligned}
        \left|{h}^{2}\sum_{\bz\in P} f(\bz)\theta_{\eps}(\bz) -  \frac{N{h}^2}{\mathrm{meas}(R)}\int_{R} f(\bx) \theta_{\eps}(\bx) d\bx \right| &\le N {h}^{2} D_{N}(P) V_{R}(f\theta_{\eps})\,.
\end{aligned}
\end{equation}
Using the fact that $N  h^2 = \textrm{meas}(R) + \cO({h}^2 \log(h^{-1}) )$ and $\textrm{meas}(R) = \cO(\eps) = \cO(h^{\alpha})$ ($\alpha\in [0, 1]$), we obtain the desired error estimate
\begin{equation*}
    \begin{aligned}
    |\cI_h(f) - \cI(f)| &=   \left|{h}^{2}\sum_{\bz\in P} f(\bz)\theta_{\eps}(\bz) - \int_{R} f(\bx) \theta_{\eps}(\bx) d\bx \right| \\
       &\le N{h}^{2} D_{N}(P) V_{R}(f\theta_{\eps}) + \cO\left(\frac{{h}^2\log({h}^{-1})}{\textrm{meas}(R)}\right) \int_{R} f(\bx)\theta_{\eps}(\bx) d\bx \\
        &= \cO(\eps^{-1} h^2\log h^{-1})\,.
\end{aligned}
\end{equation*}
This completes the proof.
\end{proof}

In Theorem~\ref{THM: LINE INT}, the requirement that the line segment has a quadratically irrational slope is mainly used to get the boundedness of the continued fraction needed in the proof. This requirement can be replaced with other conditions that would ensure the boundedness of the continued fraction. In fact, by a slight modification in the proof of Corollary~\ref{COR: DIS} and Theorem~\ref{THM: LINE INT}, we can derive the following version of the theorem.
\begin{corollary}
    Let $\Gamma\subset \bbR^2$ be a segment with slope $\beta$ whose continued fraction is
    $[b_0; b_1, b_2, \cdots]$. If $|b_k|\le C k^p$, and tube width $\eps = \Theta(h^{\alpha})$ ($\alpha \in [0, 1]$), then the quadrature error of integrating the characteristic function $f\in C(\Gamma)$ is 
    \begin{equation}
        | \cI(f) - \cI_h(f) |= \cO( h^{2-\alpha}\log^{p+1}(h))\,.
    \end{equation}
\end{corollary}

The variance of error can be characterized as follows.
\begin{theorem}\label{THM: STAT 2}
    Let $\Gamma\subset \bbR^2$ be a segment of unit length. For any rotation $\eta \in \mathrm{SO}(2)$ and translation $\bxi\in [0, 1]^2$, we denote by $\cI_h(f; \eta, \bxi)$ the implicit boundary integral with tube width $\eps=\Theta(h^{\alpha})$ ( $\alpha\in[0, 1]$) on the transformed segment $\eta \Gamma + h \bxi$. We have
    \begin{equation}
     \int_{[0, 1]^2}\int_{\eta\in \mathrm{SO}(2)} |\cI_h(f; \eta,\bxi) - \cI(f) |^2 d\eta d\bxi = \begin{cases}
          \cO(h^{3-\alpha}) , & \alpha \in [0, \frac{2q}{2q+1})\,, \\
          \cO(h^{1 + (2q+2)(1-\alpha)}), & \alpha \in [\frac{2q}{2q+1}, 1]\,,
      \end{cases}     
    \end{equation}
    where $d\eta$ is the normalized Haar measure on $\mathrm{SO}(2)$ and $q\ge 0$ is the regularity order of the weight function.
\end{theorem}
\begin{proof}
    We denote the rectangular support of $\cQ_{\Gamma}(\bx):= f(\bx)\theta_{\eps}(\bx)$ by $T_{\eps}$. If $\beta$ is the angle between the segment $\Gamma$ and $\bxi$, then 
    \begin{equation}\nonumber
    \begin{aligned}
        \widehat{\cQ}_{\Gamma}( \bxi ) &= \int_{T_{\eps}} \cQ_{\Gamma}(\bx) e^{-2\pi i \bxi\cdot \bx} d\bx =\widehat{f}(|\bxi|\cos\beta) \widehat{\theta}_{\eps}(|\bxi|\sin\beta)\,.
    \end{aligned}
    \end{equation}
    The following generic estimates of Fourier transforms can be derived using integration by parts, 
    \begin{equation}\label{EQ: FT EST}
        \begin{aligned}
            \widehat{f}(|\bxi|\cos\beta) &= \cO\left(\frac{1}{1 + |\bxi||\cos\beta|}\right)\,,\\ 
            \widehat{\theta}_{\eps}(|\bxi|\sin\beta) &= \cO\left( \left(\frac{1}{1 + \eps|\bxi||\sin\beta| }\right)^{q+1}\right)\,,
        \end{aligned}
    \end{equation}
    where $q$ is the regularity order of $\theta_{\eps}$. Therefore, there exists a constant $C > 0$ that 
    \begin{equation}\label{EQ: ROT L2}
    \begin{aligned}
            \int_{\mathrm{SO(2)}} \left|\widehat{\cQ}_{\Gamma}(\eta^{-1}\bxi)\right|^2 d\eta & = \frac{1}{2\pi}\int_0^{2\pi} \left|\widehat{f}(|\bxi|\cos\beta) \widehat{\theta}_{\eps}(|\bxi|\sin\beta)\right|^2 d\beta \\
            &\le C \int_{0}^{2\pi} \left(\frac{1}{1+|\bxi| |\cos\beta|}\right)^2 \left(\frac{1}{1 + \eps |\bxi| |\sin\beta|}\right)^{2(q+1)} d\beta \\
            & = 4C \int_{0}^{\frac{\pi}{2}} \left(\frac{1}{1+|\bxi| |\cos\beta|}\right)^2 \left(\frac{1}{1 + \eps |\bxi| |\sin\beta|}\right)^{2(q+1)} d\beta \,.
    \end{aligned} 
    \end{equation}
    We now split the integral into $\{0\le \beta\le \frac{1}{\eps|\bxi|}\}$, $\{\frac{1}{\eps|\bxi|}\le \beta\le \frac{\pi}{4}\}$, $\{\frac{\pi}{4}\le \beta\le \frac{\pi}{2}-\frac{1}{|\bxi|}\}$, and $\{\frac{\pi}{2}-\frac{1}{|\bxi|}\le \beta \le \frac{\pi}{2}\}$ to get
    \begin{equation}\nonumber
    \begin{aligned}
        \int_{0}^{\frac{1}{\eps|\bxi|}} \left(\frac{1}{1+|\bxi| |\cos\beta|}\right)^2 \left(\frac{1}{1 + \eps |\bxi| |\sin\beta|}\right)^{2(q+1)} d\beta &= \cO\left(\frac{1}{\eps|\bxi|^3}\right) \,, \\
         \int_{\frac{1}{\eps|\bxi|}}^{\frac{\pi}{4}} \left(\frac{1}{1+|\bxi| |\cos\beta|}\right)^2 \left(\frac{1}{1 + \eps |\bxi| |\sin\beta|}\right)^{2(q+1)} d\beta &= \cO\left(\frac{1}{\eps|\bxi|^3}\right)\,,\\
        \int_{\frac{\pi}{4}}^{\frac{\pi}{2}-\frac{1}{|\bxi|}} \left(\frac{1}{1+|\bxi| |\cos\beta|}\right)^2 \left(\frac{1}{1 + \eps |\bxi| |\sin\beta|}\right)^{2(q+1)} d\beta &= \cO\left(\frac{1}{\eps^{2q+2}|\bxi|^{2q+3}}\right) \,,\\
        \int_{\frac{\pi}{2}-\frac{1}{|\bxi|}}^{\frac{\pi}{2}} \left(\frac{1}{1+|\bxi| |\cos\beta|}\right)^2 \left(\frac{1}{1 + \eps |\bxi| |\sin\beta|}\right)^{2(q+1)} d\beta &= \cO\left(\frac{1}{\eps^{2q+2}|\bxi|^{2q+3}}\right) \,.\\
    \end{aligned}
    \end{equation}
    Putting these together, we have
    \begin{equation}\nonumber
    \begin{aligned}
          \sum_{\bzeta\in\bbZ^2-\{\bzero\}}\int_{\mathrm{SO(2)}} \left|\widehat{\cQ}_{\Gamma}(\eta^{-1}h^{-1} \bxi)\right|^2 d\eta &=  \cO\left(\int_{\rho \ge 1} \frac{h^3}{\eps\rho^3} \rho d\rho \right) + \cO\left(\int_{\rho \ge 1} \frac{h^{2q+3}}{\eps^{2q+2}\rho^{2q+3}}  \rho d\rho \right) \,,
    \end{aligned}
    \end{equation}
    which then leads to 
    \begin{equation}\nonumber
    \begin{aligned}
      \int_{\mathrm{SO}(2)} \int_{[0, 1]^2}|\cI(f; \eta,\bxi) - \cI(f) |^2 d\bxi d\eta =       \begin{cases}
          \cO(h^{3-\alpha}) , & \alpha \in [0, \frac{2q}{2q+1})\,, \\
          \cO(h^{1 + (2q+2)(1-\alpha)}), & \alpha \in [\frac{2q}{2q+1}, 1]\,.
      \end{cases}     
    \end{aligned}
    \end{equation}
    This is the desired result.
\end{proof}

\begin{remark}
  It should be noted that if $f\in C_0^{\infty}(\Gamma)$, then by a slight modification in~\eqref{EQ: FT EST}, one can further improve the above estimate.  The result can be adapted to any polytope boundary $\Gamma$ in high dimensions by combining the techniques in~\cite{brandolini1997average,tarnopolska1979number,chen2014panorama}. In our setting, we can use the divergence theorem to reduce the Fourier transform into faces of lower dimensions plus the volumetric integrals with potentially smaller magnitudes. The normal vector on each face stays constant, and eventually, the Fourier transform can be decomposed into finitely many one-dimensional Fourier transforms. 
\end{remark}

\subsection{Numerical experiments}
\label{SEC: EXP 4}

Here are some numerical simulations to verify the quadrature error estimates and the corresponding variance estimates in Theorem~\ref{LEM: GENERAL CURVES}, Theorem \ref{THM: LINE INT},  and Theorem~\ref{THM: STAT 2}.

\subsubsection{Semicircle}\label{SEC: SEMICIRCLE}
We verify the estimates in Theorem~\ref{LEM: GENERAL CURVES} with a semi-circle. The quadrature on a semi-circle can be fulfilled by setting the integrand function as zero on half of the circle. Let the circle be 
\begin{equation*}
    \frac{(x - x_0)^2}{r^2} + \frac{(y - y_0)^2}{r^2} = 1
\end{equation*}
with $r = \frac{3}{4}$ and $(x_0, y_0)$ is a randomly sampled point. We take the integrand as 
\begin{equation*}
    f(x, y) = \begin{cases}
        \cos(x^2 - y) \sin(y^2 - x^3) &\text{ if } y \ge y_0\,,\\
        0&\text{ if } y < y_0\,.
    \end{cases}
\end{equation*}
The weight function $\theta_{\eps}^{\Delta}\in \cW_1$ and we sample the random rigid transformation 64 times independently. The variance of quadrature error has an upper bound in Theorem~\ref{LEM: GENERAL CURVES} as $\cO(h^{3-2\alpha})$. As we can see in Figure~\ref{fig: semi-circle decay}, for  $\alpha=\frac{1}{2}$ and $\alpha = 0$, the estimated error bounds $\cO(h^{2})$ and $\cO(h^{3})$ appears consistent with the numerical results. However, for $\alpha=1$, the theoretical estimate is $\cO(h)$ while we observed a transition of decay rate from quadratic to linear. This discrepancy likely happened because in the proof of Theorem~\ref{LEM: GENERAL CURVES}, we applied directly the uniform bound $G(x,s) = \cO(|\bzeta|^{-\frac{1}{2}})$, while there are cases where a smaller bound $G(x,s) = \Theta(|\bzeta|^{-1})$ can occur. Then, the variance will consist of two components with decay rates $\cO(h)$ and $\cO(h^2)$, respectively, and the quadratic decay rate dominates when the grid size is relatively large.
\begin{figure}[!htb]
    \centering
    \includegraphics[scale=0.34]{ 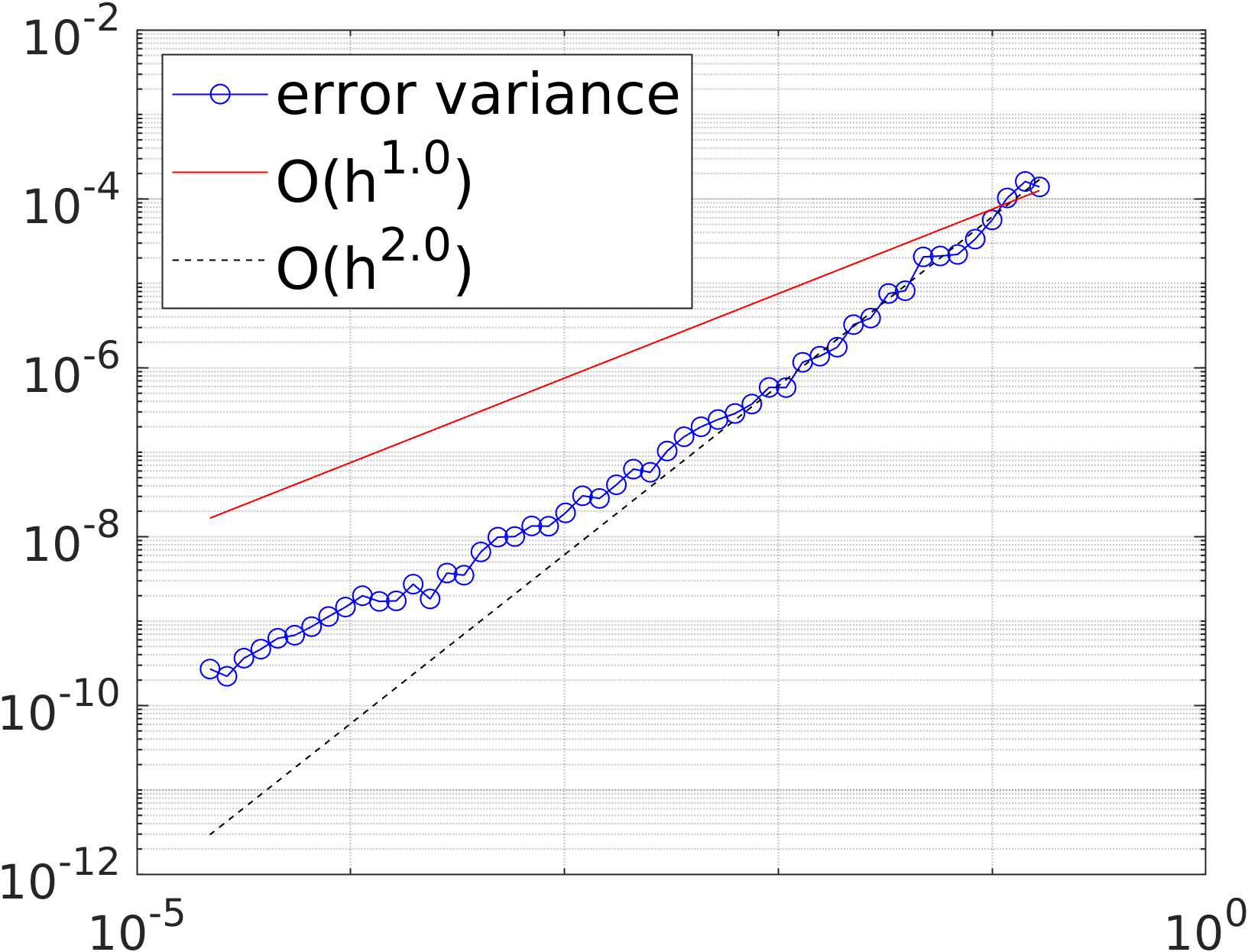}\qquad 
    \includegraphics[scale=0.34]{ 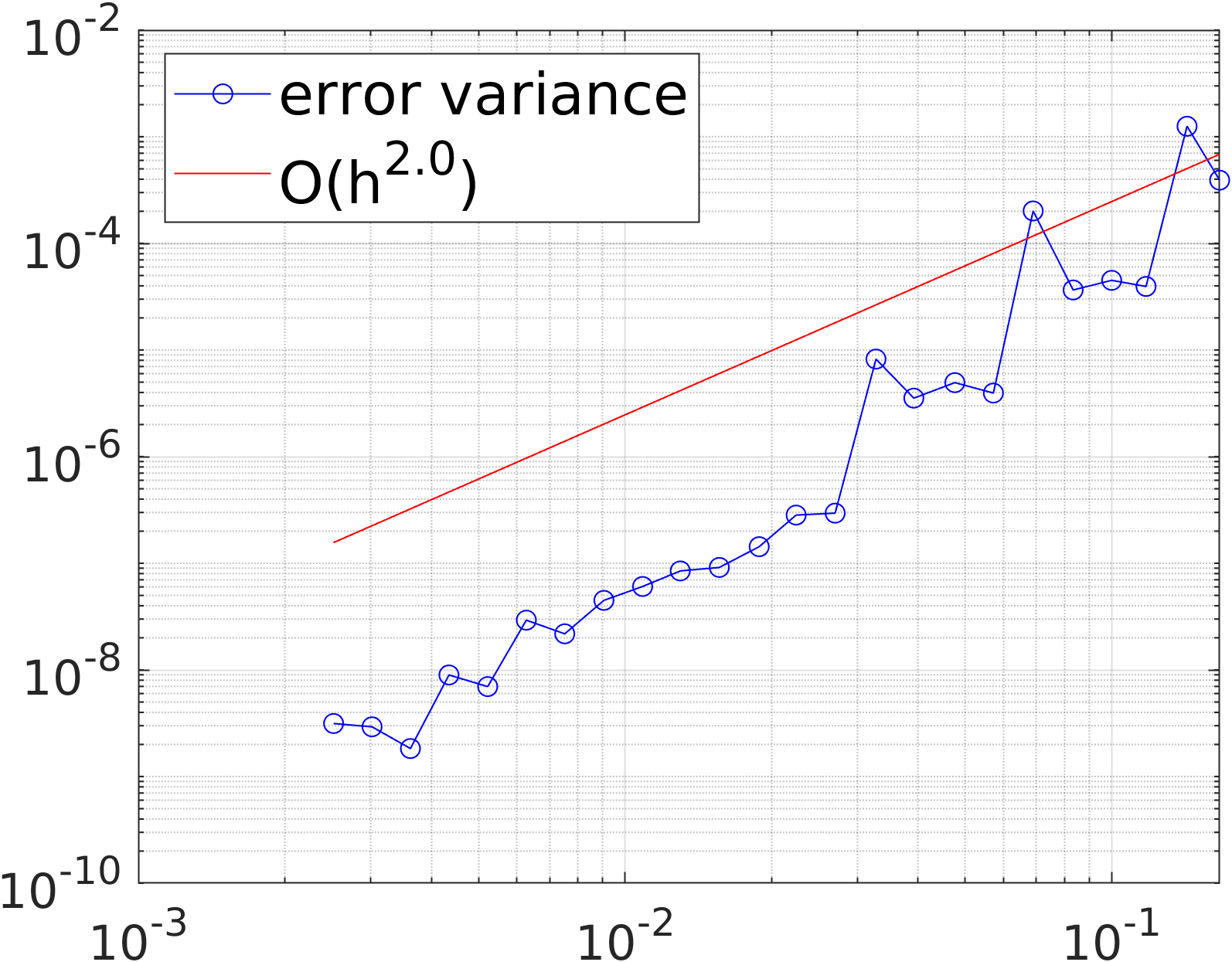}\qquad 
    \includegraphics[scale=0.34]{ 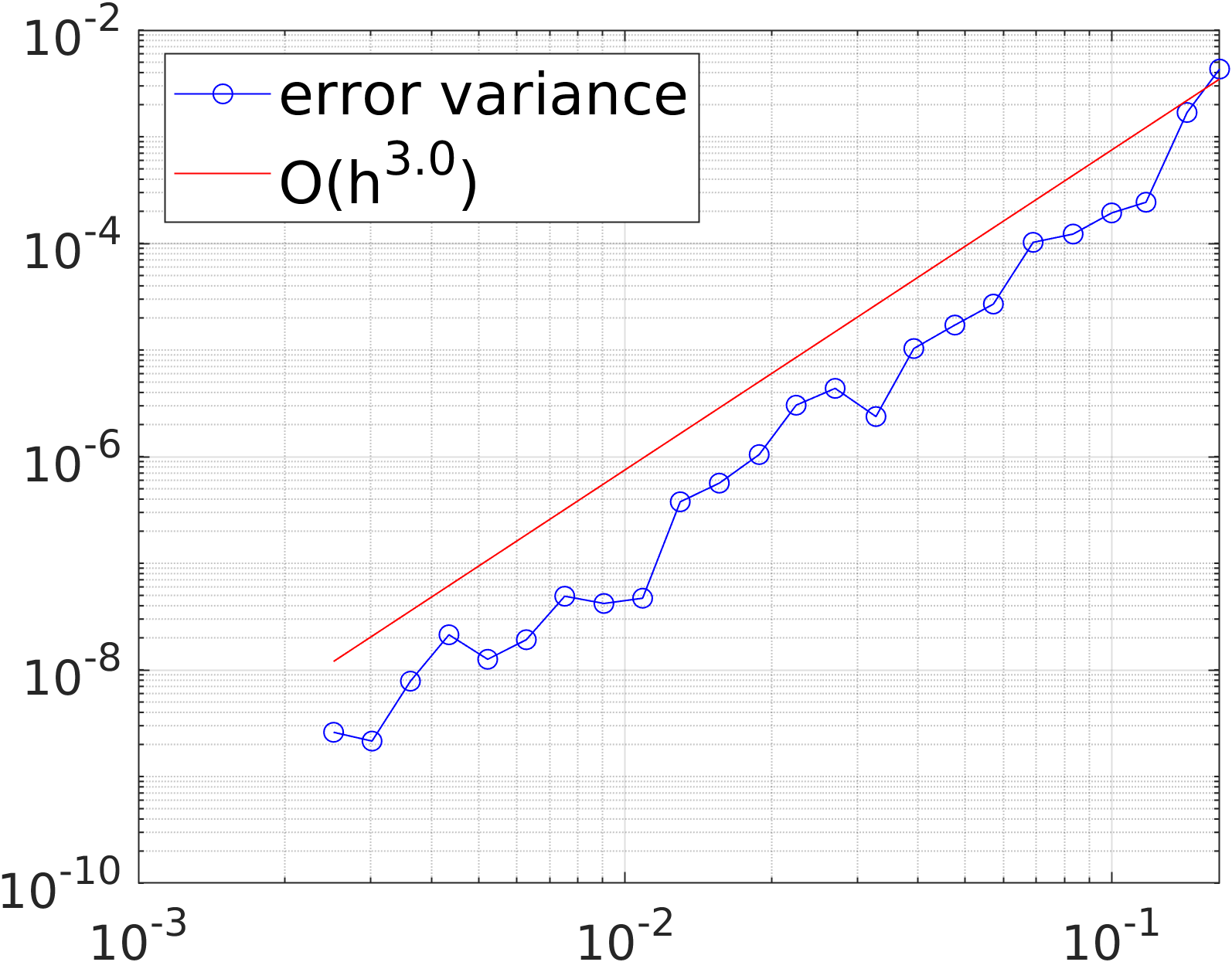} 
    \caption{Variance of quadrature error with weight function $\theta_{\eps}^{\Delta}$. Left: $\eps = 2h$. Middle: $\eps = 2h^{\frac{1}{2}}$. Right: $\eps = 0.1$.}
    \label{fig: semi-circle decay}
\end{figure}

\subsubsection{Segment with quadratically irrational slope}\label{SEC: ERROR SEG IRR}
    In this experiment, we verify the result in Theorem~\ref{THM: LINE INT} on the segment of unit length from the line $y = \gamma x$ with choices of $\gamma$ are $\sqrt{2}$ and $\frac{1+\sqrt{5}}{2}$, with continued fraction $[1;2,2,2,\cdots]$ and $[1;1,1,1,\cdots]$, respectively. Both slopes have uniformly bounded terms. The test integrand and the weight function are $f(\bx) =|\bx|^2$ and  $\theta_{\eps}^{\cos}\in \cW_2$, respectively.
    While the fluctuations are large, the upper bound of the quadrature error decays roughly $\cO(h^{2-\alpha})$ for various choices of $\eps = \Theta(h^{\alpha})$, $\alpha\in [0, 1]$; see Figure~\ref{fig: line decay 1} and Figure~\ref{fig: line decay 2}.
    \begin{figure}[!htb]
        \centering
        \includegraphics[scale=0.34]{ 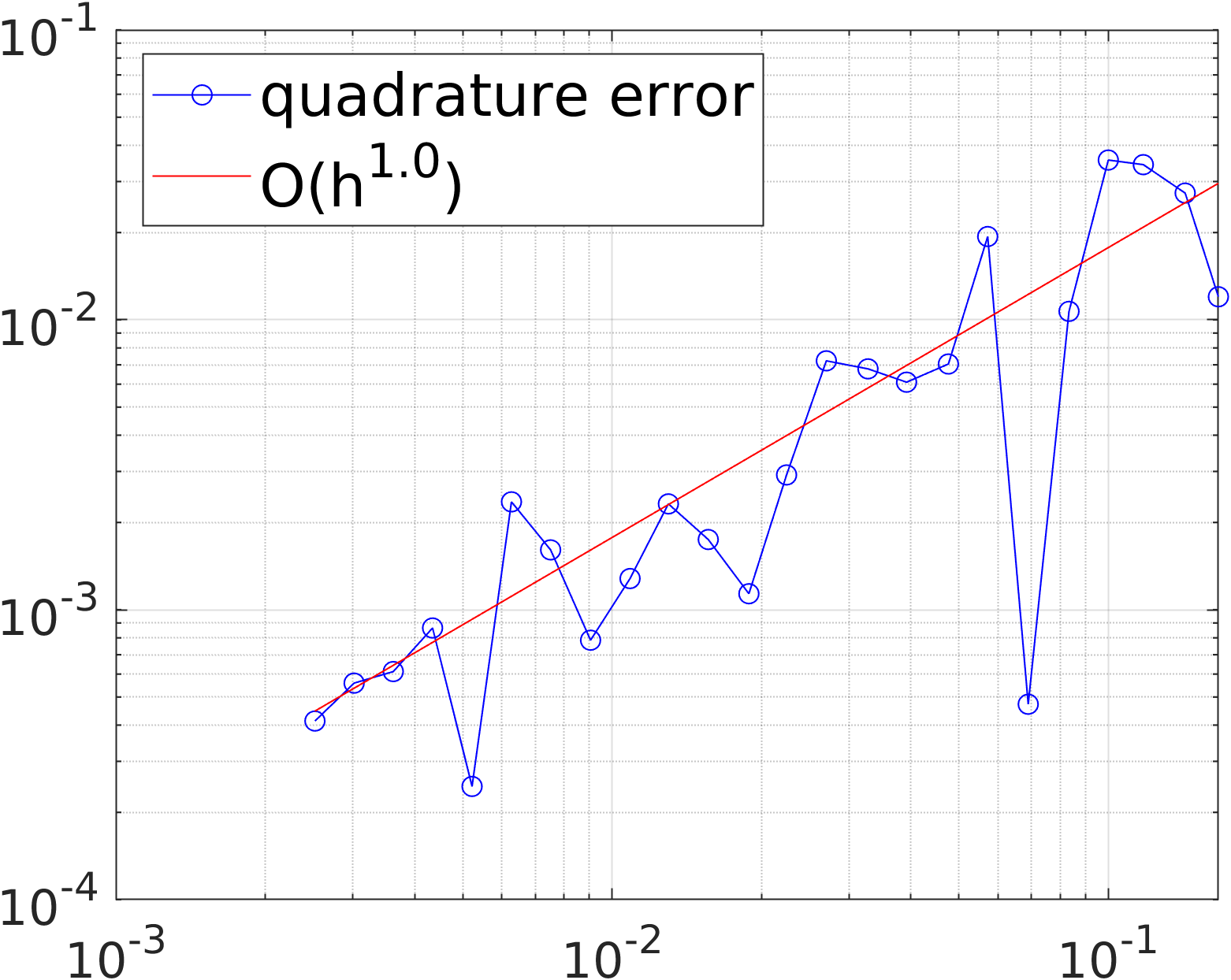}\qquad
        \includegraphics[scale=0.34]{ 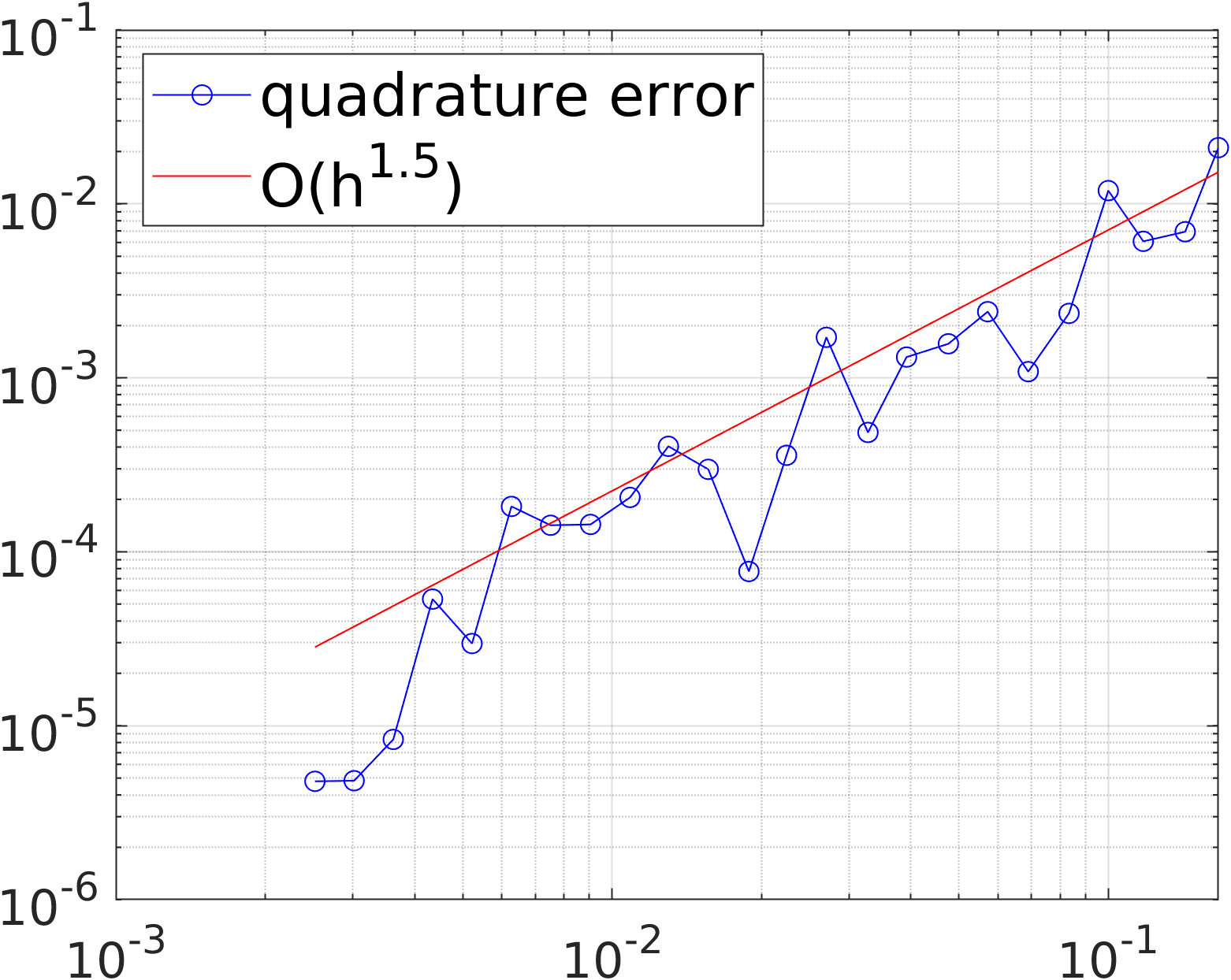}\qquad
        \includegraphics[scale=0.34]{ 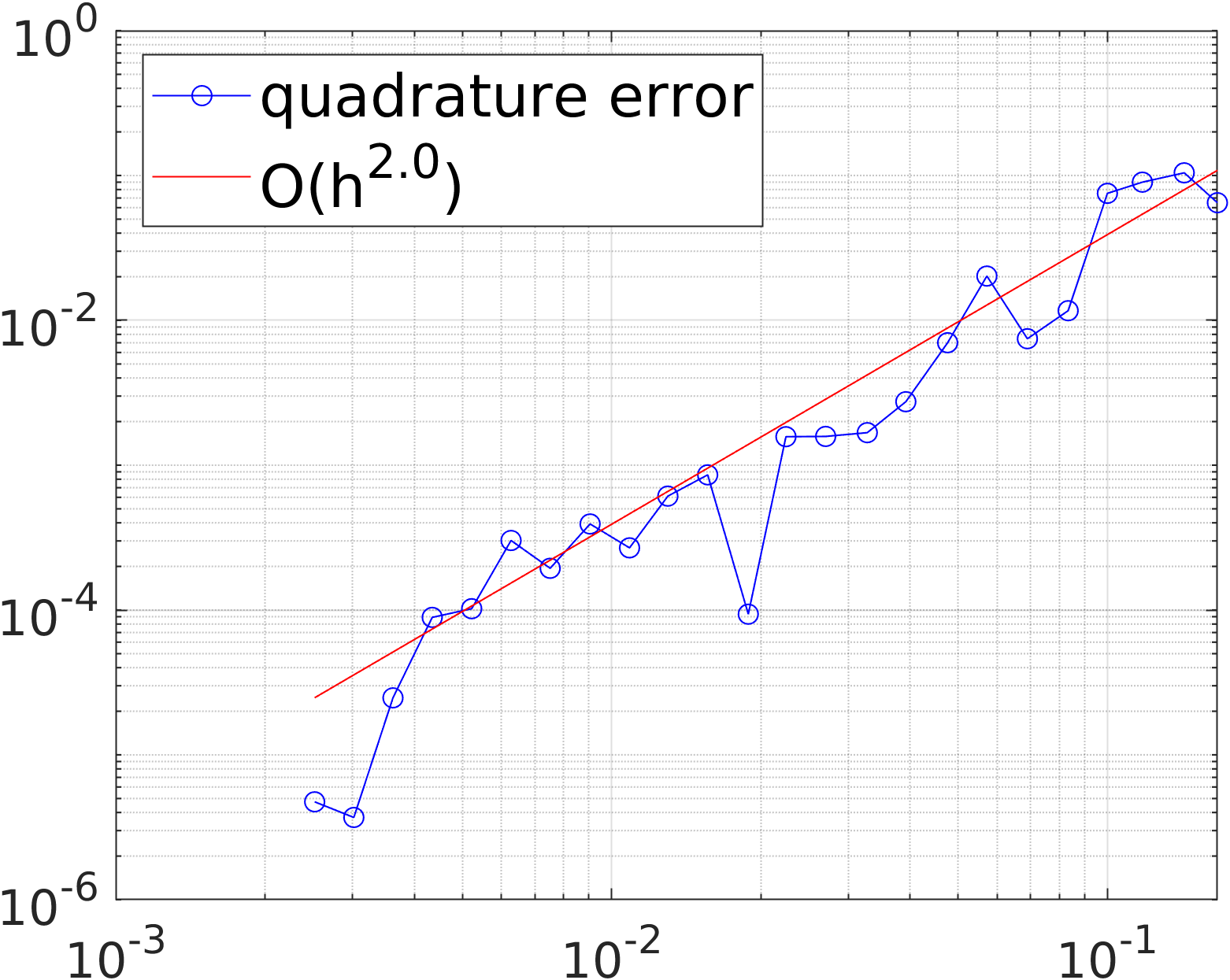}
        \caption{The error decay for different tube widths and $\gamma = \sqrt{2}$. Left: $\eps = 2h$. Middle: $\eps = 2h^{\frac{1}{2}}$. Right: $\eps = 0.1$.}
        \label{fig: line decay 1}
    \end{figure}
    \begin{figure}[!htb]
        \centering
        \includegraphics[scale=0.34]{ 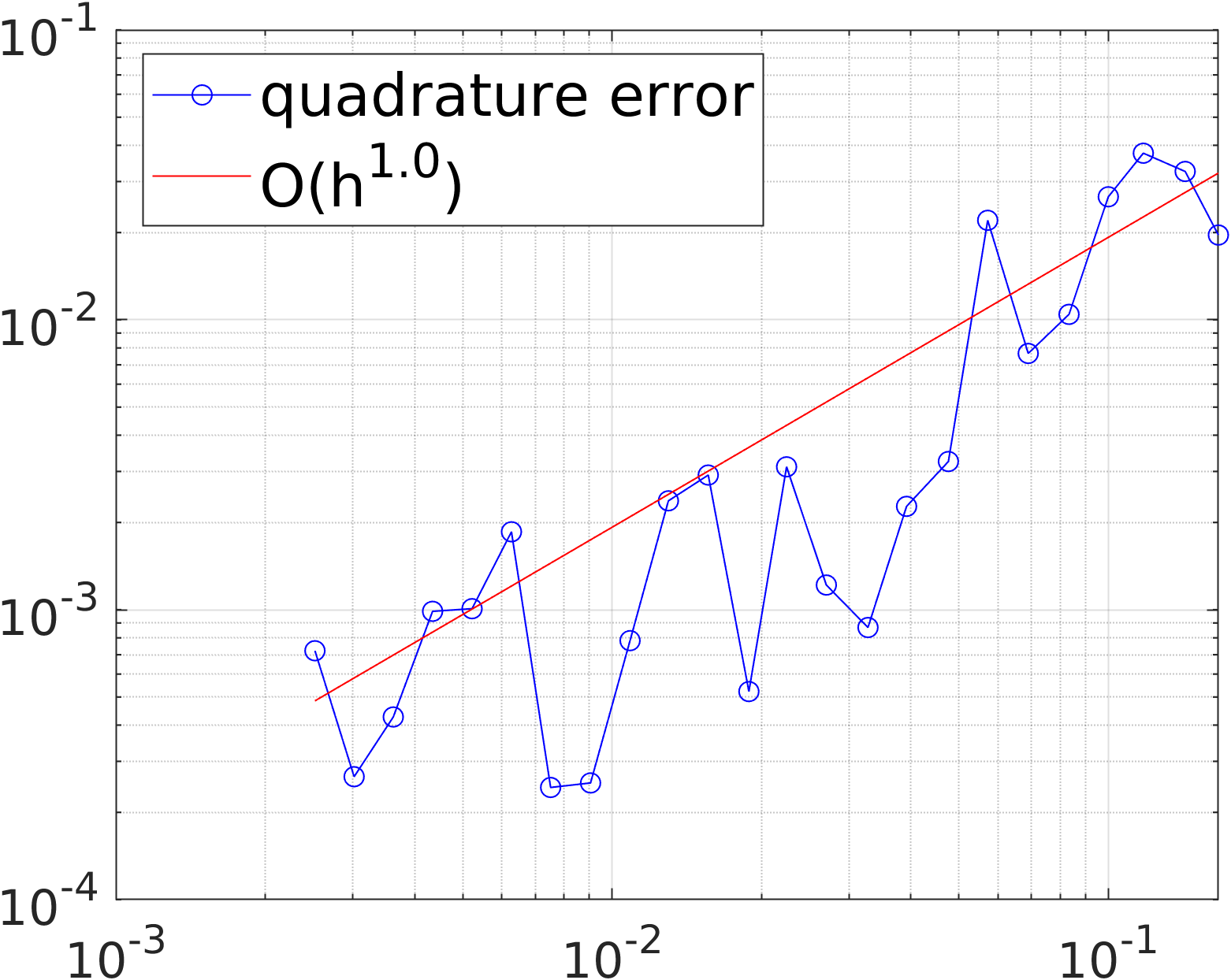}\qquad
        \includegraphics[scale=0.34]{ 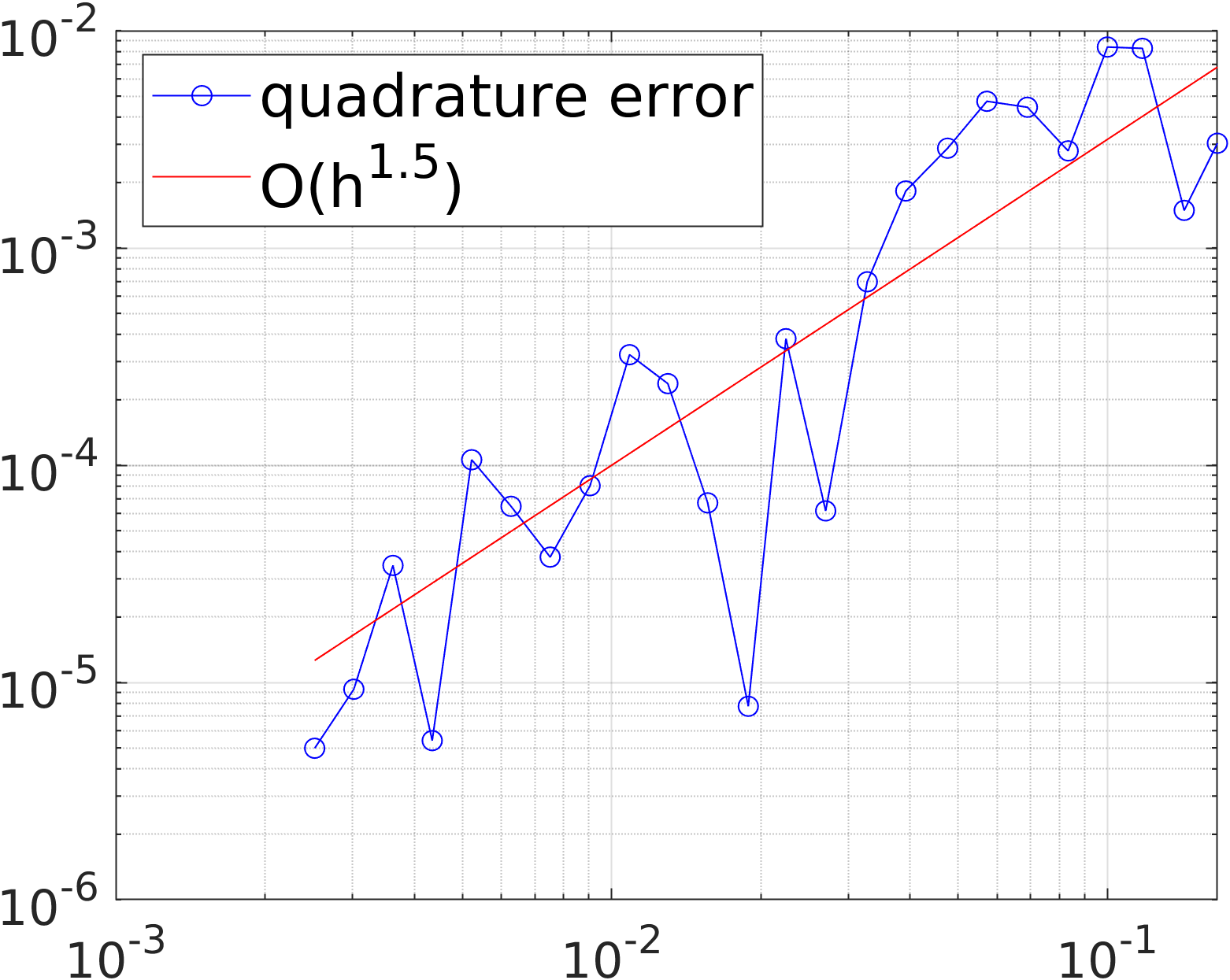}\qquad
        \includegraphics[scale=0.34]{ 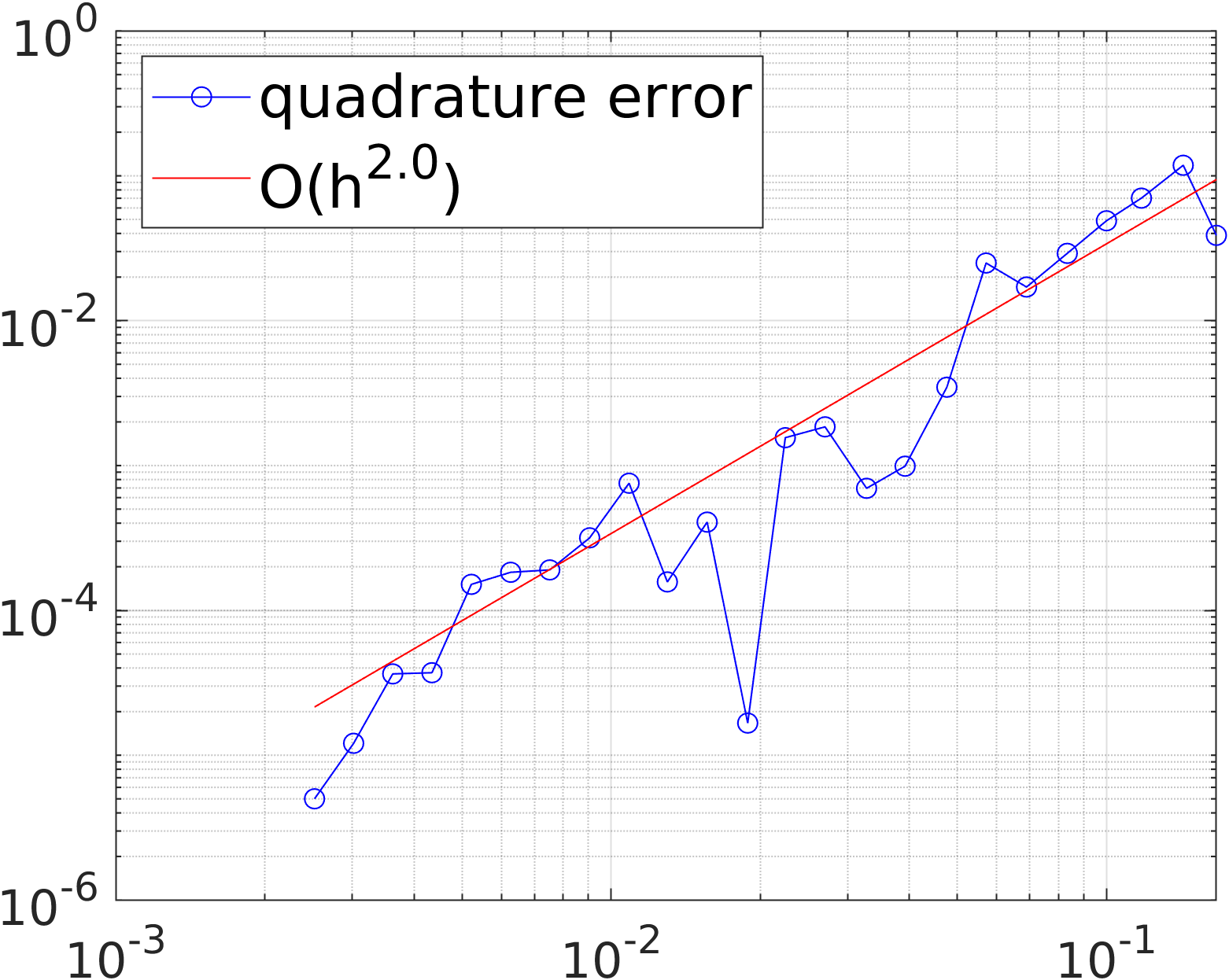}
        \caption{The error decay for different tube widths and $\gamma = \frac{\sqrt{5}+1}{2}$. Left: $\eps = 2h$. Middle: $\eps = 2h^{\frac{1}{2}}$. Right: $\eps = 0.1$.}
        \label{fig: line decay 2}
    \end{figure}

\subsubsection{Segment under random rotations}\label{SEC: VAR SEG}

To verify numerically the variance for the quadrature error on a unit segment with random rigid transformations, we choose the integrand function $f(\bx) = |\bx|^2$. The weight function is $\theta_{\eps}^{\cos}\in \cW_2$. Each experiment is performed with 32 random rigid transformations independently. A decay rate of $\cO(h^{3-\alpha})$ has been observed in Figure~\ref{fig: line var decay 2} for different tube widths $\eps = \Theta(h^{\alpha})$. In particular, for $\alpha = 0$ and $\alpha = \frac{1}{2}$, the decay rates match the prediction of Theorem~\ref{THM: STAT 2}. However, Theorem~\ref{THM: STAT 2} seems to provide an overestimate for the case of $\alpha = 1$.
\begin{figure}[!htb]
    \centering
    \includegraphics[scale=0.34]{ 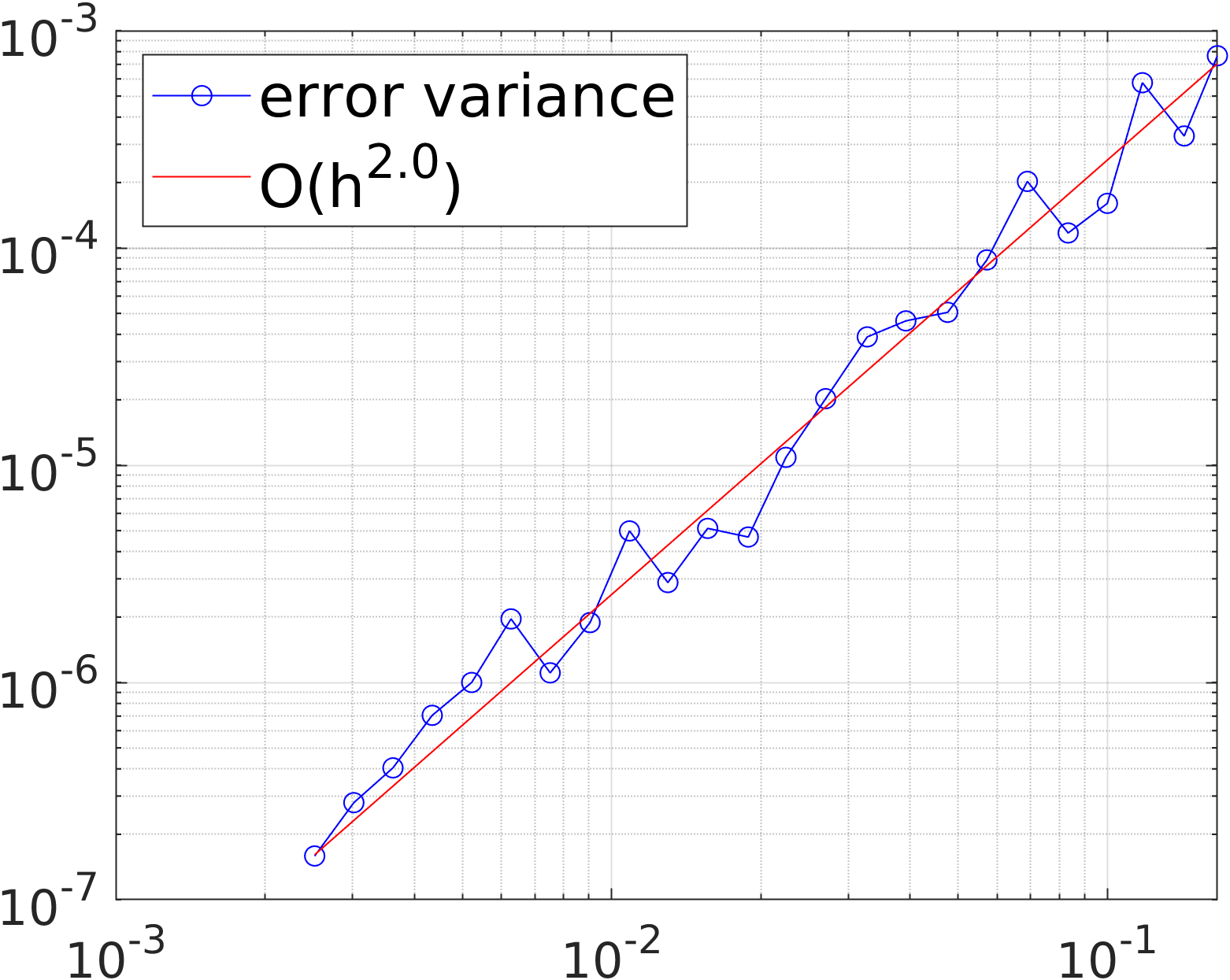}\qquad
    \includegraphics[scale=0.34]{ 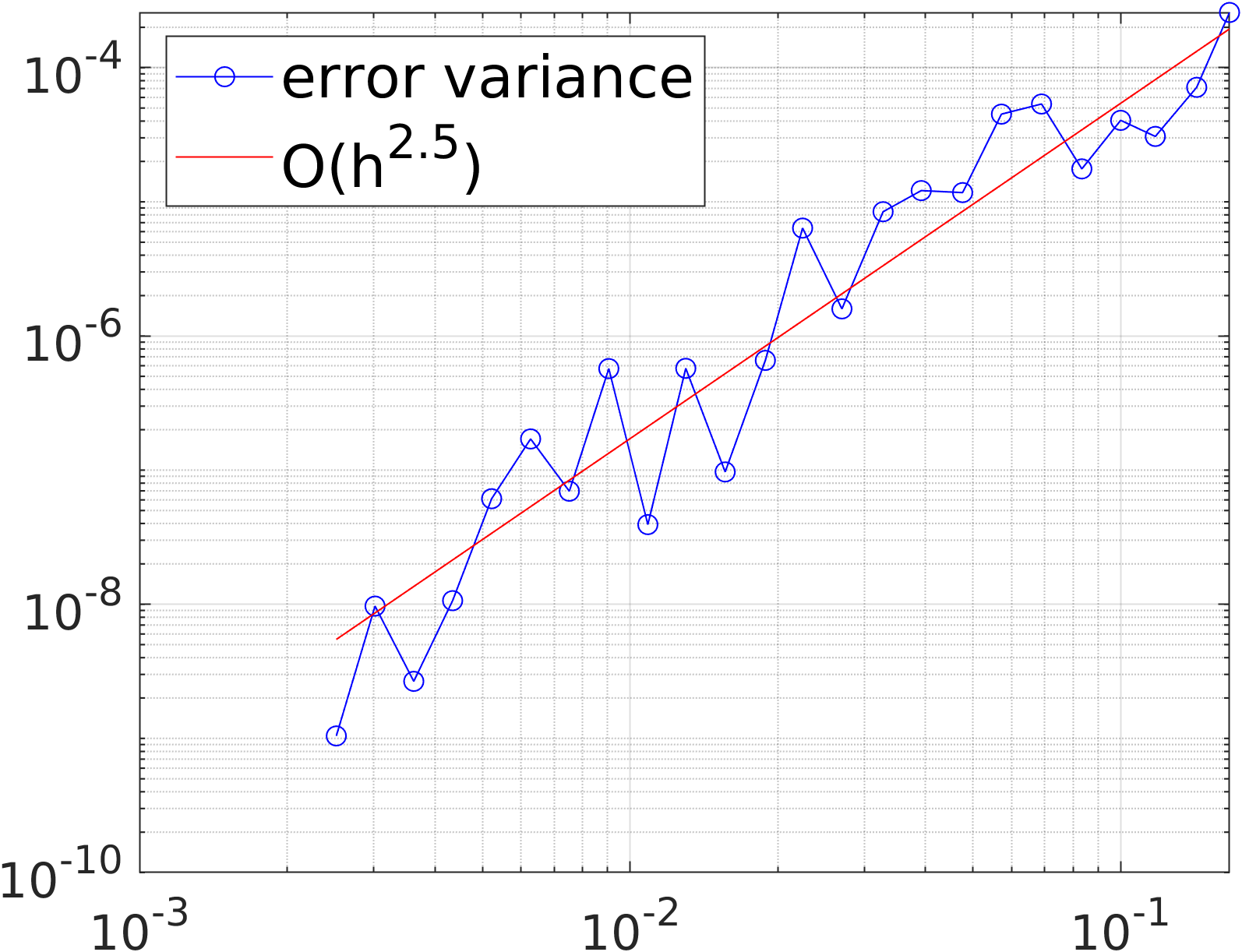}\qquad
    \includegraphics[scale=0.34]{ 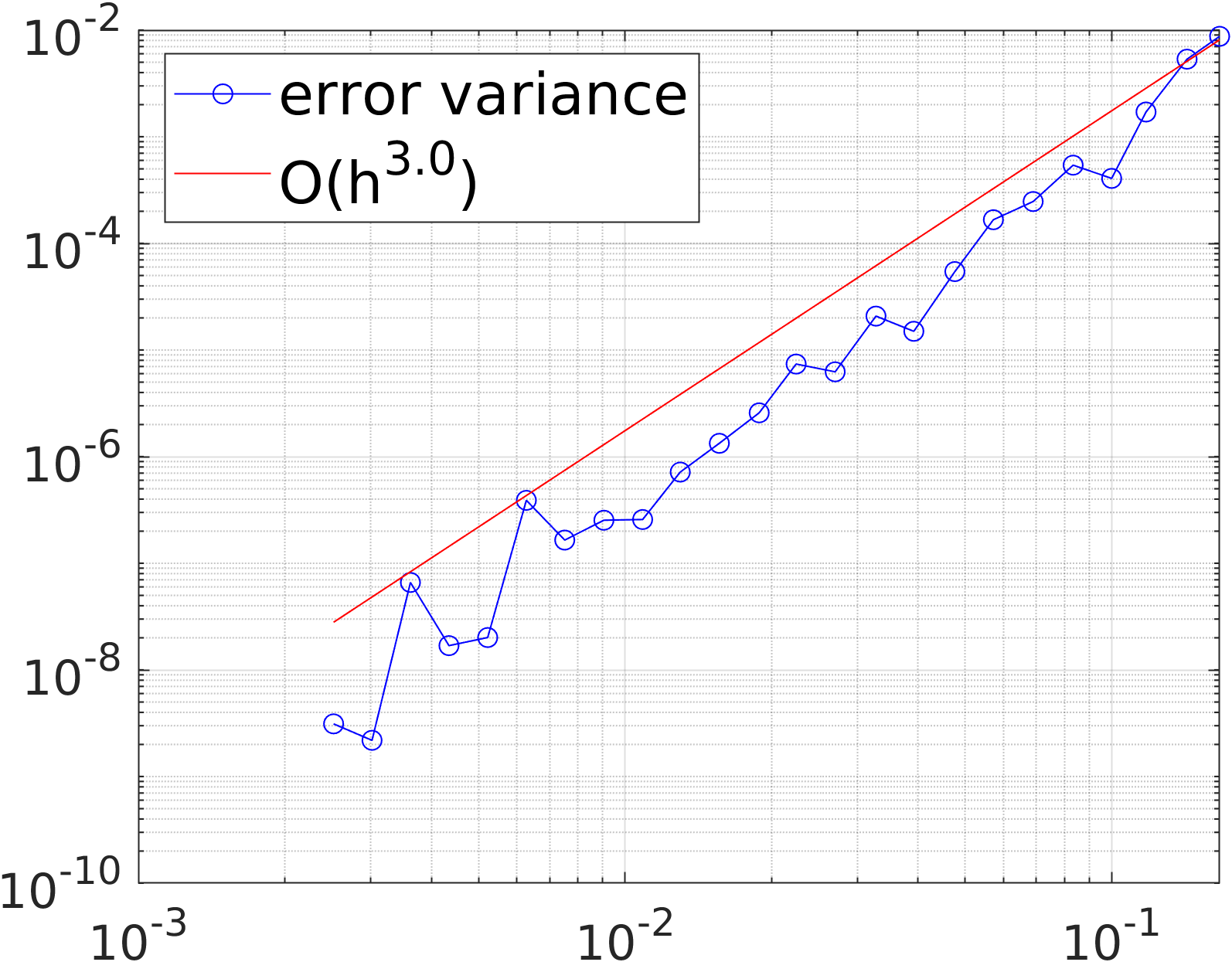}
    \caption{Variance of quadrature error with weight function $\theta_{\eps}^{\cos}$. Left: $\eps = 2h$. Middle: $\eps = 2h^{\frac{1}{2}}$. Right: $\eps = 0.1$.}
    \label{fig: line var decay 2}
\end{figure}

\subsubsection{Capsule shape}
\label{SEC: VAR CAP}

The estimates in Theorem~\ref{LEM: GENERAL CURVES} and Theorem~\ref{THM: STAT 2} imply that, under random rigid transformations, the implicit boundary integral produces a variance of quadrature error of $\cO(h^{\min(3-\alpha, 1+2(q+1)(1-\alpha), 3 - 2\alpha})$ for a boundary made of strongly convex curves and segments in two dimensions. This can be verified by considering the capsule shape boundary with random rigid transformations; see Figure~\ref{fig: capsule shape}.
\begin{figure}[!htb]
    \centering
    \includegraphics[scale=0.4]{ 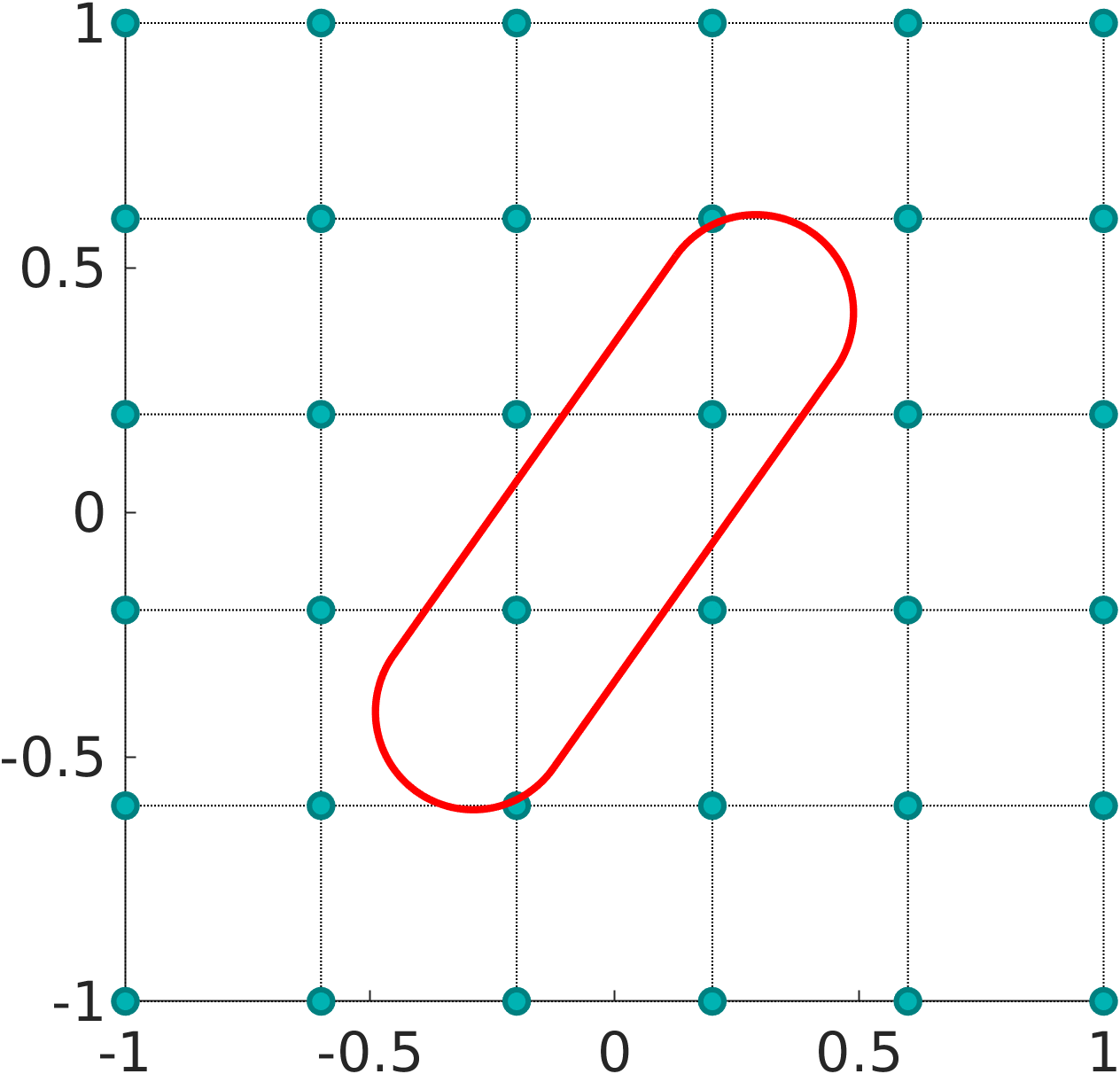}
    \caption{Capsule shape boundary (solid red line) consists of two segments of unit length and two semicircle boundaries with radius $r=0.2$.}
    \label{fig: capsule shape}
\end{figure} 
The integrand function is chosen as
\begin{equation*}
    f(x, y) = \cos(x^2 - y) \sin(y^2 - x^3).
\end{equation*}
The weight function $\theta_{\eps}^{\Delta}\in \cW_1$. Each experiment is performed with 32 random rigid transformations independently. The variance of quadrature error is shown in Figure~\ref{fig: var decay}. The error matches the theoretical estimate $\cO(h^{\min(3-\alpha, 5-4\alpha, 3-2\alpha)}) = \cO(h^{3-2\alpha})$.
\begin{figure}[!htb]
    \centering
    \includegraphics[scale=0.34]{ 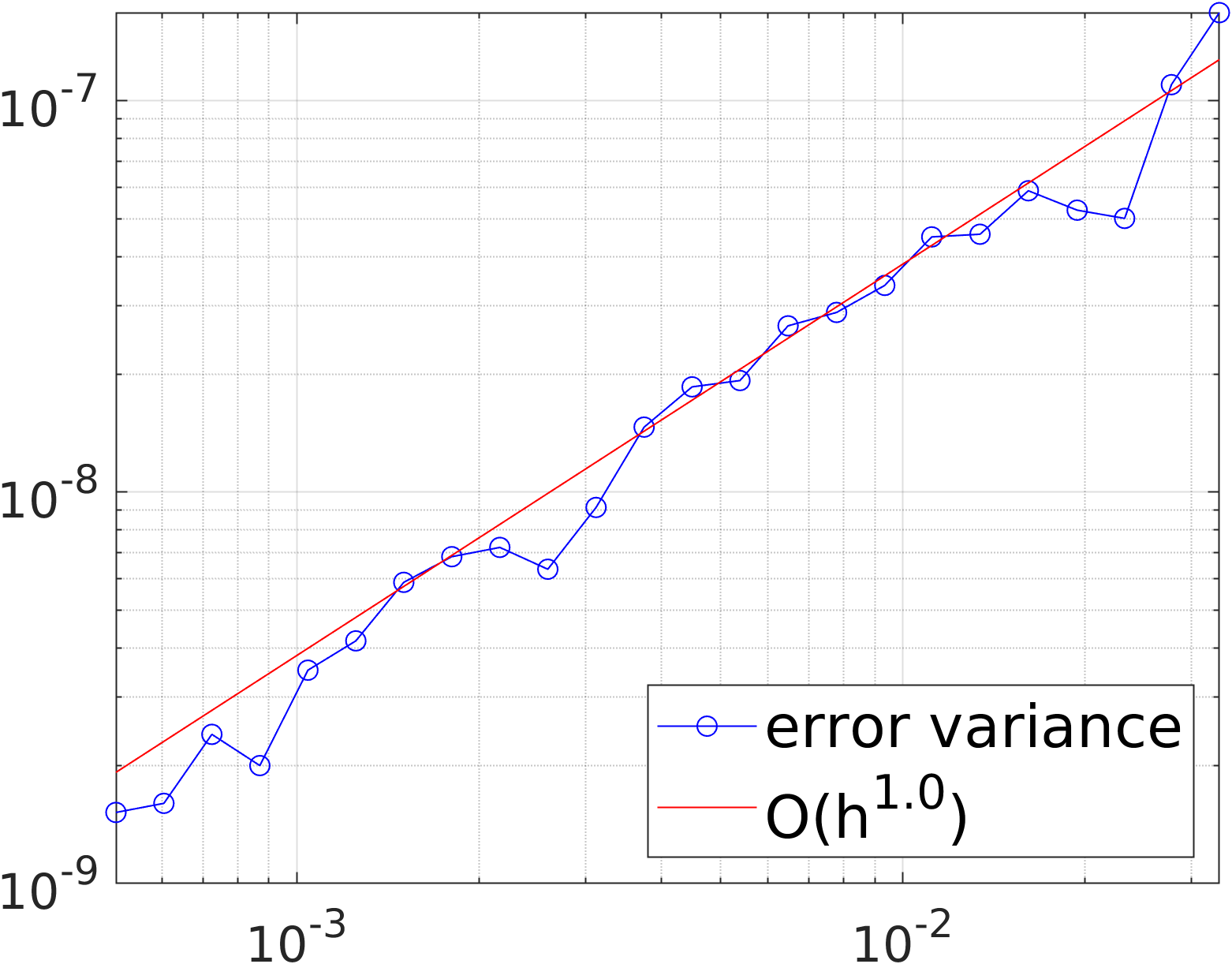}\qquad
    \includegraphics[scale=0.34]{ 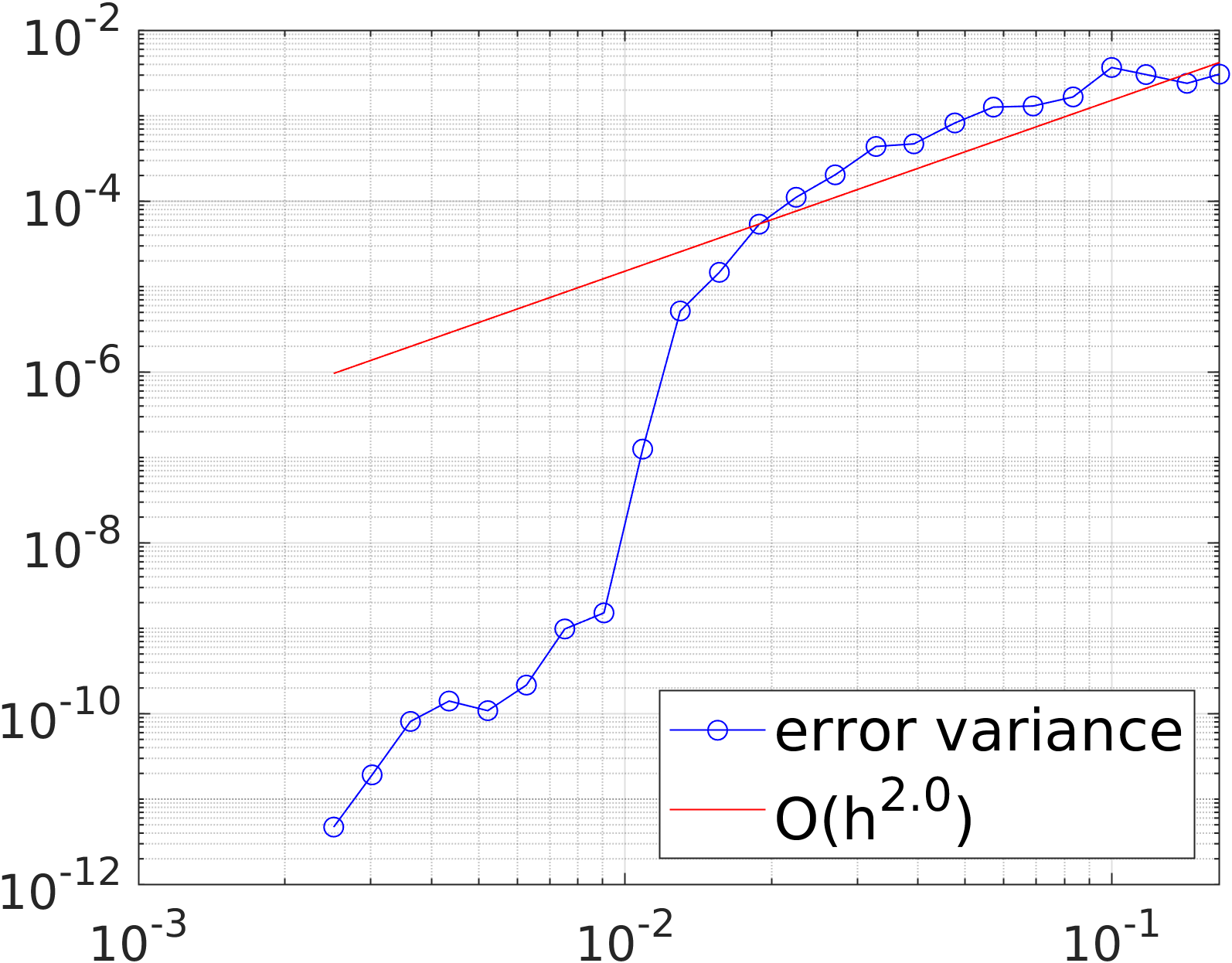}\qquad
    \includegraphics[scale=0.34]{ 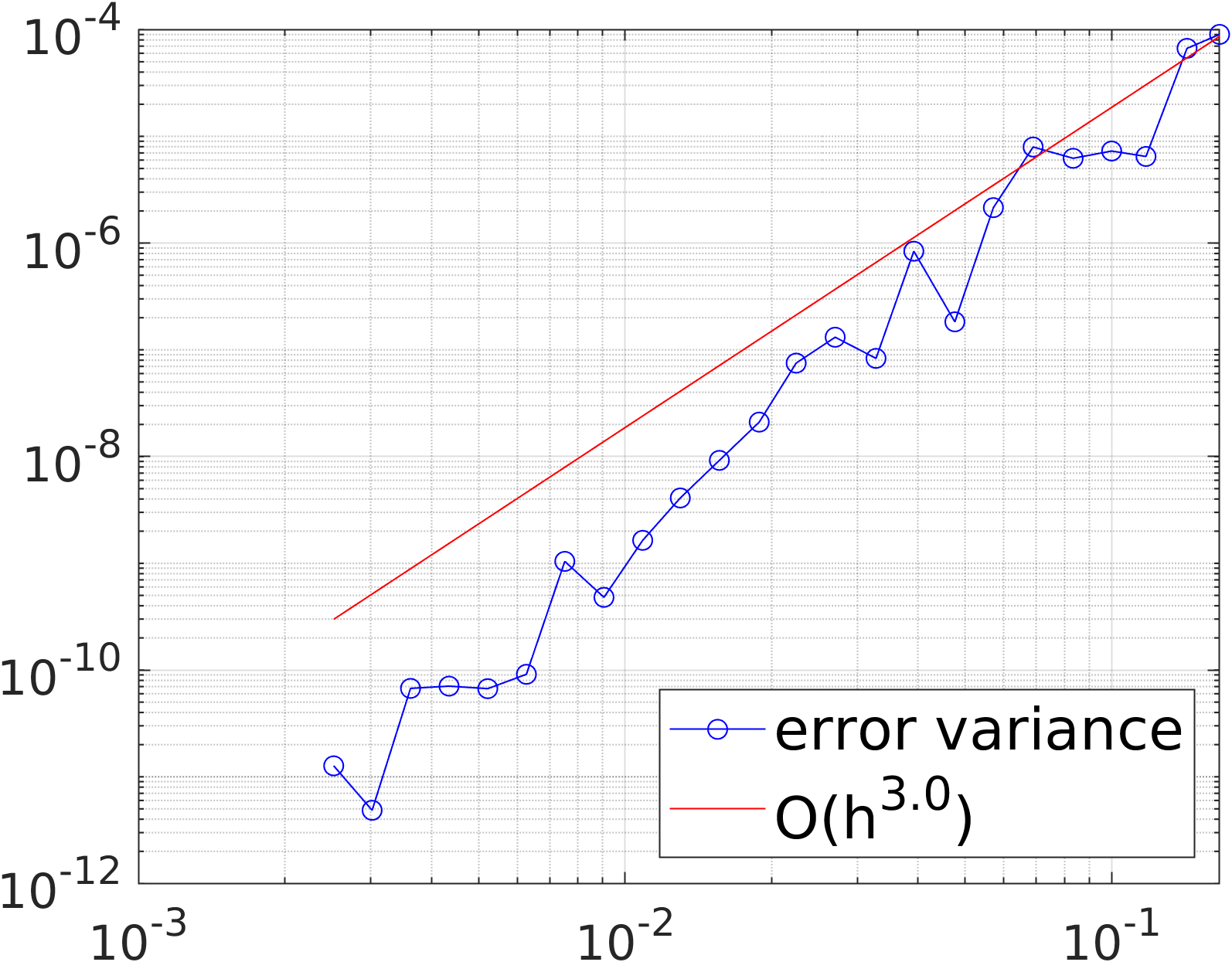}
    \caption{Variance of quadrature error with weight function $\theta_{\eps}^{\Delta}$. Left: $\eps = 2h$. Middle: $\eps = 2h^{\frac{1}{2}}$. Right: $\eps = 0.1$.}
    \label{fig: var decay}
\end{figure}

\section{Conclusion}
\label{SEC: 5}

This work analyzes the approximation of the numerical quadrature of the implicit boundary integral method under different assumptions on the shape and regularity of the boundary. We show that the quadrature error gains an additional order of $\frac{d-1}{2}$ for a smooth, strongly convex boundary. For a general smooth convex boundary in 2D with finitely many points of vanishing curvatures, the average quadrature error has an additional order of $\frac{1}{2}$ under random rigid transformations. The variance of quadrature error estimate is also extended to open curves instead of closed boundaries in 2D. For generic piecewise smooth curves with vanishing curvatures that $\kappa\in [2, \frac{3+\sqrt{5}}{2})$ (see Section~\ref{SEC: 3.1}), the average quadrature error is $\cO(h^{\frac{1}{\kappa}+ \frac{1}{\kappa-1}(1-\alpha)})$. If the curve is a segment, the average quadrature error will become
$\cO(h^{\min(\frac{3-\alpha}{2}, \frac{1}{2} + (1-\alpha)(q+1))})$ instead. Error estimates for general curves with larger $\kappa$, as well as high-dimensional quadrature error estimates for general smooth surfaces and polytopes, will be studied in future works.

\section*{Acknowledgment}

The work of YZ is partially supported by the National Science Foundation through grant DMS-2309530. KR is partially supported by NSF DMS-2309802. RT is partially supported by the Army Research Office, under
Cooperative Agreement Number W911NF-19-2-0333 and National Science Foundation Grant DMS-2110895.

\appendix

\section{Proof of Lemma~\ref{LEM: STATIONARY PHASE}}\label{PRF: STATION}
\begin{proof}
    For any $\bx\in T_{\eps}$, one can write it as $\bx = \bx' + d_{\Gamma} (\bx) \bn(\bx')$, where $\bx' = P_{\Gamma}\bx$. The operator $P_{\Gamma}$ is the projection onto $\Gamma$ and $d_{\Gamma}$ is the signed distance function to $\Gamma$. Take a partition of unity $\{\phi_j\}_{j=1}^N$ on $\Gamma$ that $\phi_{j} \in C_0^{\infty}(\Gamma)$. On $\supp \phi_j$, we set a point $\by_j$ as the origin. Then locally $\Gamma$ can be represented by
    \begin{equation*}
        \Gamma = \{(\by', y_d) \in B\mid y_d = \rho_j(\by')\}
    \end{equation*}
    where $B$ is a ball centered at the origin. One can arrange $\rho_j$ such that $\rho_j(\bzero) = 0$, $\nabla_{\by'} \rho_j(\by')|_{\by' = \bzero} = \bzero$ and $|\nabla_{\by'} \rho_j (\bzero)| = 1$ on $\Gamma$. Next, we extend the definition of $\phi_j$ to $T_{\eps}$ without losing regularity by setting
    \begin{equation*}
        \widetilde \phi_j(\bx) := \phi_{j}(P_{\Gamma}\bx)\,.
    \end{equation*}
    Under the above setup, the Fourier transform of $\cQ\widetilde\phi_j$ can be written as
    \begin{equation}\label{EQ: Fourier Q}
    \begin{aligned}
        \widehat{\cQ \widetilde \phi_j}(\bzeta) &= \int_{\bbR^d} e^{-2\pi i (\bx\cdot \bzeta)} \cQ(\bx) \widetilde \phi_j(\bx) d\bx \\
        &= \int_{\Gamma} \int_{-\eps}^{\eps}   f(\bx')  \widetilde\phi_j(\bx')  e^{-2\pi i (\bx' + s\bn(\bx'))\cdot \bzeta}  \theta_{\eps}(s) ds  d\sigma \\
        &= \int_{\Gamma} f(\bx') \widetilde\phi_j(\bx')  e^{-2\pi i \bx'\cdot \bzeta} \int_{-\eps}^{\eps}    e^{-2\pi i s\bn(\bx') \cdot \bzeta}  \theta_{\eps}(s) ds  d\sigma 
    \end{aligned}
    \end{equation}
    where $\sigma$ is the Lebesgue measure on $\Gamma$. Using the local frame $(\by', y_d)$ to replace $\bx'\in\Gamma$ and writing $\bzeta = (\bzeta', \zeta_d)$, we define
    \begin{equation*}
        \psi(\by', \bzeta', \zeta_d, s) := \by'\cdot \bzeta' + \rho_j(\by')  \zeta_d + s \bn(\by', \rho_j(\by'))\cdot \bzeta\,.
    \end{equation*}
    Let $c = \sup_{\bx'\in\supp\phi_j} |\nabla_{\by'}\rho_j|$ with $\nabla_{\by'} \bn$ being the shape operator (hence $ \kappa = \sup_{\bx'\in \supp\phi_j} |\nabla_{\by'} \bn|$ is bounded by the principal curvature).
    Suppose the tube width $\eps$ is small enough that $\eps\kappa < 1$ and the smallest eigenvalue of $\nabla^2_{\by'} \rho_j$ is larger than $\eps \sqrt{1+\tau^2}\|\nabla^2_{\by'} \bn\|_{\cL(\bbR^{d}\mapsto \bbR^{d}\times \bbR^d)} $ for $\tau = \frac{c}{1-\eps \kappa}$.

\noindent{\bf{Case I}. } If $|\bzeta'| > \frac{c}{1 - \eps\kappa} |\zeta_d|$  , we find that 
    \begin{equation*}
    \begin{aligned}
         |\nabla_{\by'}(\by' \cdot \bzeta' + \rho_j(\by')\zeta_d + s \bn(\by')\cdot \bzeta)| &= |\bzeta' + {\nabla_{\by'} \rho_j} \zeta_d + s \nabla_{\by'}\bn(\by')\cdot \bzeta |\\&\ge |\bzeta'|(1-s\kappa) - |\zeta_d| |\nabla_{\by'} \rho_j|  > 0\,.       
    \end{aligned}
    \end{equation*}
    Using Fubini's theorem, we can represent the integral~\eqref{EQ: Fourier Q} as 
    \begin{equation*}
    \begin{aligned}
            \widehat{\cQ \widetilde \phi_j}(\bzeta) &= \int_{-\eps}^{\eps}  \theta_{\eps}(s) 
 \int_{\Gamma} f(\bx') \widetilde\phi_j(\bx')  e^{-2\pi i (\bx' + s\bn(\bx') )\cdot \bzeta}  d\sigma ds  \\
 &= \int_{-\eps}^{\eps}  \theta_{\eps}(s) 
 \int_{\bbR^{d-1}} e^{-2\pi i \psi(\by', \bzeta', \zeta_d, s)} f(\by', y_d) \widetilde \phi_j (\by', y_d)  \left[1 + |\nabla_{\by'}\rho_j|^2\right]^{\frac{1}{2}} d\by' ds\,.
    \end{aligned}
    \end{equation*}
    Since there is no stationary point in the phase $\psi$, this integral decays as $\cO(|\bzeta|^{-N})$ for any fixed $N$.
    
\noindent{\bf{Case II}.}  If $|\bzeta'| \le \frac{c}{1 - \eps\kappa} 
 |\zeta_d|$, the choice of $\eps$ makes $\nabla_{\by'}^2\phi = \nabla^2_{\by'} \rho_j \zeta_d + s \nabla^2 \bn \cdot \bzeta$ positive-definite. Define 
    \begin{equation*}
        \Phi(\bx', \bzeta) := \int_{-\eps}^{\eps}    e^{-2\pi i s\bn(\bx') \cdot \bzeta}  \theta_{\eps}(s) ds\,.
    \end{equation*} 
Because $\theta_{\eps}\in\cW_q$, we can apply integration by parts $q$ times for $\Phi$ to obtain 
    \begin{equation*}
    \begin{aligned}
         \Phi(\bx', \bzeta) &:= \int_{-\eps}^{\eps} e^{-2\pi i s \bn(\bx')\cdot \bzeta} \theta_{\eps}(s) ds\\
         &=\left(\frac{1}{2\pi i \bn(\bx')\cdot \bzeta}\right)^q \sum_{l=1}^L\int_{s_l}^{s_{l+1}} e^{-2\pi i s\bn(\bx')\cdot \bzeta} \frac{d^q}{ds^q} \theta_{\eps}(s) ds\,,
    \end{aligned}
    \end{equation*}
    where $[s_l, s_{l+1}]$ denotes the support of the $l$-th piece $C^1$ component of $\frac{d^q}{ds^q}\theta_{\eps}(s)$. By applying integration by parts again on each support $[s_l, s_{l+1}]$, we have that
    \begin{equation*}
    \begin{aligned}
            &\int_{s_l}^{s_{l+1}} e^{-2\pi i s \bn(\bx')\cdot \bzeta} \frac{d^q}{ds^q}\theta_{\eps}(s) ds \\
            &= -\frac{1}{2\pi i \bn(\bx')\cdot \bzeta}  e^{-2\pi i s \bn(\bx')\cdot \bzeta}\frac{d^q}{ds^q} \theta_{\eps}(s) \Big|_{s_l}^{s_{l+1}}  +\frac{1}{2\pi i \bn(\bx')\cdot \bzeta}\int_{s_l}^{s_{l+1}} e^{-2\pi i s \bn(\bx')\cdot \bzeta} \frac{d^{q+1}}{ds^{q+1}} \theta_{\eps}(s) ds\,.
    \end{aligned}
    \end{equation*}
    Therefore~\eqref{EQ: Fourier Q} becomes
    \begin{equation}\nonumber
    \begin{aligned}
        \widehat{\cQ \widetilde \phi_j}(\bzeta) &= \int_{\bbR^d} e^{-2\pi i (\bx\cdot \bzeta)} \cQ(\bx) \widetilde \phi_j(\bx) d\bx \\
        &= - \frac{1}{|\bzeta|^{q+1}}\sum_{l=1}^L \left(\frac{d^q}{ds^q} \theta_{\eps}\Big|_{s_l^{+}}^{s_l^{-}} \right) \int_{\Gamma} \left(\frac{1}{2\pi i \bn(\bx')\cdot \frac{\bzeta}{|\bzeta|}}  \right)^{q+1} f(\bx') \widetilde\phi_j(\bx') e^{-2\pi i (\bx' + s_l \bn(\bx') ) \cdot \bzeta}  d\sigma \\
       & +  \frac{1}{|\bzeta|^{q+1}}\sum_{l=1}^L \int_{s_l}^{s_{l+1}} \frac{d^{q+1}}{ds^{q+1}} \theta_{\eps}(s)\int_{\Gamma} f(\bx') \left(\frac{1}{2\pi i \bn(\bx')\cdot \frac{\bzeta}{|\bzeta|}}  \right)^{q+1}\widetilde\phi_j(\bx') e^{-2\pi i (\bx' + s \bn(\bx') ) \cdot \bzeta}   d\sigma  ds\,.
    \end{aligned}
    \end{equation}
    The main task is now to estimate the integral
    \begin{equation}\label{EQ: STA PHASE}
    \begin{aligned}
        &\int_{\Gamma}\left(\frac{1}{2\pi i \bn(\bx')\cdot \frac{\bzeta}{|\bzeta|}}  \right)^{q+1} f(\bx') \widetilde\phi_j(\bx') e^{-2\pi i (\bx' + s \bn(\bx') ) \cdot \bzeta}   d\sigma \\&= \int_{\bbR^{d-1}} e^{-2\pi i \phi(\by', \bzeta', \zeta_d) } \left(\frac{1}{2\pi i \bn(\bx')\cdot \frac{\bzeta}{|\bzeta|}}  \right)^{q+1}f(\by', y_d) \widetilde \phi_j (\by', y_d)  \left[1 + |\nabla_{\by'}\rho_j|^2\right]^{\frac{1}{2}} d\by'\,.
    \end{aligned}
    \end{equation}
Since for all $\bx'\in \supp\phi_j $ we have $|\bn(\bx') \cdot \frac{\bzeta}{|\bzeta|} | > c'$ for certain $c' > 0$, we can easily deduce from~\eqref{EQ: STA PHASE} that 
$|\widehat{\cQ}(\bzeta)| = \cO(\cR_{\eps}\cP_{\eps}|\bzeta|^{-(q+1)} )$, where 
    \begin{equation*}
    \begin{aligned}
        \cR_{\eps} &:=  \|\theta_{\eps}^{(q)}\|_{L^{\infty}[-\eps,\eps]} +  \eps\|\theta_{\eps}^{(q+1)}\|_{L^{\infty}[-\eps, \eps]} =  \cO(\eps^{-q-1})\,,\\
        \cP_{\eps} &:= \sup_{s\in [-\eps, \eps]} \int_{\Gamma} \left(\frac{1}{2\pi i \bn(\bx')\cdot \frac{\bzeta}{|\bzeta|}}  \right)^{q+1}f(\bx') \widetilde\phi_j(\bx') e^{-2\pi i (\bx' + s \bn(\bx') ) \cdot \bzeta}   d\sigma\,.
    \end{aligned}
    \end{equation*}
    In the final step, we apply the standard stationary phase approximation to $\cP_{\eps}$. This produces a factor of $|\bzeta|^{-(d-1)/2}$. Therefore, we have $|\widehat{\cQ}(\bzeta)| = \cO(\eps^{-q-1}|\bzeta|^{-(d-1)/2 - (q+1)}) $. This finishes the proof.
\end{proof}
\begin{remark}
It is clear from the proof that we can relax the regularity assumption on the boundary from $C^{\infty}$ to $C^{k}$ with $k = \frac{d+5}{2} + \max(\frac{d-1}{2}, q)$.
\end{remark}

\begin{remark}\label{REMARK:Convex}
We also observe from the proof that the result, in a modified form, can hold for general convex surfaces with at least one positive principal curvature. Indeed, since the stationary phase approximation directly depends on the number of positive principal curvatures of $\Gamma$, we can replace the dimensionality constant $d$ in the above proof (and therefore in Lemma~\ref{LEM: STATIONARY PHASE}) with $\Lambda+1$ ($\Lambda$ being the number of positive principal curvatures). The result remains true.
\end{remark}

\section{Proof of Lemma~\ref{LEM: ESTIMATE}}
\label{PRF: ESTIMATE}

\begin{proof}
   Using Lemma~\ref{LEM: STATIONARY PHASE}, we have 
   \begin{equation*}
       | \widehat{\cQ}(h^{-1}\bn) \widehat{\psi}(\delta \bn)| = \cO(\eps^{-(q+1)} h^{\frac{d+1}{2} + q} |\bn|^{-\frac{d+1}{2} - q}) |\widehat{\psi}(\delta \bn)|\,.
   \end{equation*}
   We separate the summation $\sum_{|\bn|\neq 0} \widehat{\cQ}(h^{-1}\bn) \widehat{\psi}(\delta \bn)$ into two groups: $\{\bn\in\bbZ^d\mid 1\le |\bn|\le \delta^{-\beta}\}$ and $\{\bn\in\bbZ^d\mid |\bn| > \delta^{-\beta}\}$. The parameter $\beta > 0$ is a constant to balance the summation between the two groups.  Using the uniform boundedness of $|\widehat{\psi}(\delta \bn)|$, the first group can be estimated as
   \begin{equation}\label{EQ:Q-P}
   \begin{aligned}
        \sum_{1\le |\bn|\le \delta^{-\beta}} | \widehat{\cQ}(h^{-1}\bn) \widehat{\psi}(\delta \bn)|  &=  \sum_{1\le |\bn|\le \delta^{-\beta}} \cO(\eps^{-q-1} h^{\frac{d+1}{2} + q}|\bn|^{-\frac{d+1}{2} - q}) \\
        &= \cO(\eps^{-q-1} h^{\frac{d+1}{2} + q}) \int_{1}^{\delta^{-\beta}} r^{-\frac{d+1}{2} - q} r^{d-1} dr \\
        &= \cO(\eps^{-q-1} h^{\frac{d+1}{2} + q}) \int_{1}^{\delta^{-\beta}} r^{\frac{d-3}{2} - q} dr \\
        &=\begin{cases}
            \cO(\eps^{-q-1} h^{\frac{d+1}{2} + q} \delta^{-\beta(\frac{d-1}{2}-q)}) \quad & \frac{d-1}{2} > q\,, \\ 
            \cO(\eps^{-q-1} h^{\frac{d+1}{2} + q}  |\log\delta|)\quad &\frac{d-1}{2} = q\,, \\
             \cO(\eps^{-q-1} h^{\frac{d+1}{2} + q})\quad &\frac{d-1}{2} < q\,.
        \end{cases}
   \end{aligned}
   \end{equation}
   For the second group, we use that fact that $|\widehat{\psi}(\delta \bn)| =\cO(|\bn\delta|^{-\nu})$ for any fixed $\nu > 0$. In particular, we take $\nu =  \frac{d+1}{2}$ to have 
   \begin{equation}\label{EQ:Q-N}
   \begin{aligned}
    \sum_{|\bn| > \delta^{-\beta}} | \widehat{\cQ}(h^{-1}\bn) \widehat{\psi}(\delta \bn)| &= \cO(\eps^{-q-1} h^{\frac{d+1}{2} + q})\int_{\delta^{-\beta}}^{\infty} r^{-\frac{d+1}{2} - q} |r\delta|^{-\nu} r^{d-1} dr \\
    &= \cO(\eps^{-q-1} h^{\frac{d+1}{2} + q} \delta^{\beta (q+1) - \nu})\,.
   \end{aligned}
   \end{equation}
   Next, we choose $\beta$ to balance the contributions from~\eqref{EQ:Q-P} and ~\eqref{EQ:Q-N}. This is split into three different cases as follows.
   \begin{enumerate}
       \item When $\frac{d-1}{2} > q$, we take $\beta = 1$. This leads to
       \begin{equation*}
           \sum_{|\bn|\neq 0} |\widehat{\cQ}(h^{-1}\bn) \widehat{\psi}(\delta \bn)| =  \cO(\eps^{-q-1} h^{\frac{d+1}{2} + q} \delta^{-(\frac{d-1}{2}-q)})\,. 
       \end{equation*}
       \item When $\frac{d-1}{2} = q$, we can take an arbitrary $\beta > 1$ to get 
        \begin{equation*}
           \sum_{|\bn|\neq 0} |\widehat{\cQ}(h^{-1}\bn) \widehat{\psi}(\delta \bn)| =  \cO(\eps^{-q-1} h^{\frac{d+1}{2} + q} |\log\delta|)\,. 
       \end{equation*}
       \item When $\frac{d-1}{2} < q$, the contribution from~\eqref{EQ:Q-N} is negligible. We may take an arbitrary $\beta \ge 1$. This leads to
       \begin{equation*}
           \sum_{|\bn|\neq 0} |\widehat{\cQ}(h^{-1}\bn) \widehat{\psi}(\delta \bn)| =  \cO(\eps^{-q-1} h^{\frac{d+1}{2} + q})\,. 
       \end{equation*}
   \end{enumerate}
   The proof is complete.
\end{proof}

\section{The van der Corput Lemma}

We now prove a special version of the van der Corput Lemma that we used in Section~\ref{SEC: 3}. Let us first recall the standard van der Corput Lemma. The proof of the lemma can be found in~\cite{stein1993harmonic}.
\begin{lemma}[van der Corput Lemma~\cite{stein1993harmonic}]\label{LEM: van der Corput}
    Let $\phi:\bbR\mapsto\bbR$ be a $C^2$ function on interval $J$. (i) If $|\phi'(x)| > c>0$ for all $x\in J$ and $\phi''(x)$ does not change its sign, then for any $f\in C^1(J)$, 
    \begin{equation}\label{EQ: van der Corput 1}
        \int_{J} e^{i\lambda \phi(x)} f(x) dx = \cO\left(\frac{\|f\|_{C^1}}{\lambda c} \right).
    \end{equation}
    (ii) If $|\phi''(x)| > c>0$ for all $x\in J$, then for any $f\in C^2(J)$, 
    \begin{equation}\label{EQ: van der Corput 2}
        \int_{J} e^{i\lambda \phi(x)} f(x) dx = \cO\left(\|f\|_{C^1}\sqrt{\frac{1}{\lambda c}}\right).
    \end{equation}
\end{lemma}
Note that we do not need $f$ to be compactly supported on $J$ in the van der Corput Lemma. Using the above lemma, we can prove the following result.


\begin{lemma}\label{LEM: van der Corput Revised}
Let $\kappa > 2$ be given and $g(x) = |x|^{\kappa} h(x)$ with $h\in C^2[-1, 1]$ and $h(x)\neq 0$ on $[-1, 1]$. 
Let $r < 1$ be sufficiently small that 
    $$r\le \frac{\kappa(\kappa - 1)| h(0)|}{2(\|h''\|_{\infty} + \kappa(\kappa + 1)\|h'\|_{\infty} ) }\, . $$
If $f\in C^{2}[-r,r]$, then for any $\bzeta = (\zeta_1, \zeta_2)\in\bbR^2$ such that $|\zeta_2|> \beta |\zeta_1|$ for some $\beta>0$, we have
    \begin{equation}\label{EQ: DEGEN VAN DER}
        \int_{-r}^{r} e^{-2\pi i (\zeta_1 x + \zeta_2 g(x))} f(x) dx = \cO(|\bzeta|^{1/\kappa})\,.
    \end{equation}
\end{lemma}
\begin{proof}
Let us first remark that the assumption that $\bzeta = (\zeta_1, \zeta_2)$ satisfy $|\zeta_2|> \beta |\zeta_1|$ for some $\beta>0$ is necessary as otherwise there are no stationary points near the origin (in the small interval $[-r, r]$). 

We use the classical dyadic decomposition of the interval $(-r, r)$ to rewrite the integral in~\eqref{EQ: DEGEN VAN DER} into
\begin{equation}\nonumber
        \int_{-r}^{r} e^{-2\pi i (\zeta_1 x + \zeta_2 g(x))} f(x) dx = \sum_{\ell=0}^{\infty} \cJ_\ell\,, 
\end{equation}
where
\begin{equation*}
    \begin{aligned}
           \cJ_\ell &:= \int_{ -\frac{r}{2^{\ell}} }^{ -\frac{r}{2^{\ell+1}} } e^{-2\pi i (\zeta_1 x + \zeta_2 g(x))} f(x) dx  + \int_{\frac{r}{2^{\ell+1}}}^{\frac{r}{2^{\ell}}} e^{-2\pi i (\zeta_1 x + \zeta_2 g(x))} f(x) dx  \\
            &= \frac{1}{2^\ell}\int_{-r}^{-r/2} e^{-2\pi i (\zeta_1 2^{-\ell}z + \zeta_2 h(2^{-\ell}z) 2^{-\ell\kappa}|z|^{\kappa})} f(2^{-\ell}z) dz \\
            &\quad + \frac{1}{2^\ell}\int_{r/2}^{r} e^{-2\pi i (\zeta_1 2^{-l}z + \zeta_2 h(2^{-\ell}z) 2^{-\ell\kappa}|z|^{\kappa})} f(2^{-\ell}z) dz\,.
    \end{aligned}
\end{equation*}
When $\ell$ is sufficiently large, $|2^{-\ell\kappa}\zeta_2| <1$. We can conclude straightforwardly that 
\begin{equation}\label{EQ:Large ell}
    |\cJ_\ell| \le 2^{-\ell}\|f\|_{\infty} = \cO(1)\,,
\end{equation} 
since the integrating interval length is $r<1$. Otherwise, let $\phi(z):= \zeta_1 2^{-\ell}z + \zeta_2 h(2^{-\ell}z) 2^{-\ell\kappa}|z|^{\kappa}$. We then observe that 
\begin{equation*}
    \begin{aligned}
    |\phi''(z)| &\ge  |\zeta_2| 2^{-\ell\kappa}|z|^{\kappa-2} \left( \kappa(\kappa - 1)  |h(2^{-\ell}z)| - 2^{-2\ell} |h''(2^{-\ell}z)||z|^{2} - 2\kappa 2^{-\ell}  |h'(2^{-\ell}z)| |z| \right)    \\
    &\ge |\zeta_2| 2^{-\ell\kappa}|z|^{\kappa-2} \left( \kappa(\kappa - 1)  \inf_{(-r,r)}|h| - \|h''\|_{\infty}r^2 - 2\kappa   \|h'\|_{\infty} r \right) \quad (\text{since }r^2 < r)\\
    &\ge |\zeta_2| 2^{-\ell\kappa}|z|^{\kappa-2} \left( \kappa(\kappa - 1) \left( |h(0)| - r\|h'\|_{\infty}\right)  - \|h''\|_{\infty}r - 2\kappa   \|h'\|_{\infty} r \right) \\
    &\ge |\zeta_2| 2^{-\ell\kappa}|z|^{\kappa-2} \left(\kappa(\kappa - 1) |h(0)| - r \left(\kappa(\kappa + 1)   \|h'\|_{\infty} +  \|h''\|_{\infty}  \right) \right)\\
    &\ge \frac{1}{2}|\zeta_2| 2^{-\ell\kappa}|z|^{\kappa-2}\kappa(\kappa - 1) |h(0)| \quad (\text{since } |z|\ge \frac{r}{2})\\
    &\ge \frac{1}{2}|\zeta_2| 2^{-\ell\kappa}\left|\frac{r}{2}\right|^{\kappa-2}\kappa(\kappa - 1) |h(0)|\,.
    \end{aligned}
\end{equation*}
The second part of the van der Corput Lemma, that is, the estimate in~\eqref{EQ: van der Corput 2}, then gives us that
\begin{equation}\label{EQ:Small ell}
        |\cJ_\ell| =\cO\left(2^{-\ell}\sqrt{\frac{1}{|\zeta_2| 2^{-\ell\kappa}}}\right)\,.
\end{equation}
We now combine~\eqref{EQ:Large ell} and~\eqref{EQ:Small ell} to conclude that 
\begin{equation*}
        |\cJ_\ell| = \cO\left(\frac{2^{-\ell}}{(1+|\zeta_2| 2^{-\ell\kappa})^{1/2}} \right).
\end{equation*}
Taking the summation over $\cJ_\ell$, we find that 
\begin{equation*}
    \begin{aligned}
     |\sum_{\ell=0}^{\infty} \cJ_\ell| \le \sum_{\ell=0}^{\infty} \frac{2^{-\ell}}{(1+|\zeta_2| 2^{-\ell\kappa})^{1/2}}     
    &\le \sum_{2^{\ell\kappa } < |\zeta_2|} 2^{-\ell}\sqrt{\frac{1}{|\zeta_2| 2^{-\ell\kappa}} } + \sum_{2^{\ell\kappa } \ge |\zeta_2|} 2^{-\ell} \\
    &\le |\zeta_2|^{-1/2} \sum_{2^{\ell\kappa } < |\zeta_2|} 2^{\ell(\kappa/2 - 1)} + 2|\zeta_2|^{-1/\kappa} \\&= \cO(|\zeta_2|^{-1/\kappa})\,.
    \end{aligned}
\end{equation*}
This, together with the fact that $|\zeta_2| > \beta |\zeta_1|$, gives the desired result in ~\eqref{EQ: DEGEN VAN DER}.
\end{proof}
\begin{remark}\label{RMK:Nonconvex}
    It is straightforward to see that the above proof also works for the case where $g(x) = \sgn(x) |x|^{\kappa} h(x)$.
\end{remark}

\bibliographystyle{siam}
\bibliography{main}

\begin{thebibliography}{10}

\bibitem{beck1994probabilistic}
{\sc J.~Beck}, {\em Probabilistic diophantine approximation, {I}. {Kronecker}
  sequences}, Ann. Math., 140 (1994), pp.~449--502.

\bibitem{berndt2018circle}
{\sc B.~C. Berndt, S.~Kim, and A.~Zaharescu}, {\em The circle problem of
  {Gauss} and the divisor problem of {Dirichlet}--still unsolved}, Am. Math.
  Mon., 125 (2018), pp.~99--114.

\bibitem{bertrandias1960calcul}
{\sc J.~P. Bertrandias}, {\em Calcul d'une int{\'e}grale au moyen de la suite
  xn= an. {\'e}valuation de l'erreur}, Annales de l'ISUP, 9 (1960),
  pp.~335--357.

\bibitem{bourgain2017mean}
{\sc J.~Bourgain and N.~Watt}, {\em Mean square of zeta function, circle
  problem and divisor problem revisited}, arXiv:1709.04340,  (2017).

\bibitem{brandolini2013koksma}
{\sc L.~Brandolini, L.~Colzani, G.~Gigante, and G.~Travaglini}, {\em On the
  {Koksma}--{Hlawka} inequality}, Journal of Complexity, 29 (2013),
  pp.~158--172.

\bibitem{brandolini2015lp}
\leavevmode\vrule height 2pt depth -1.6pt width 23pt, {\em $l^p$ and weak-$l^p$
  estimates for the number of integer points in translated domains}, Math.
  Proc. Camb. Phil. Soc., 159 (2015), pp.~471--480.

\bibitem{brandolini1997average}
{\sc L.~Brandolini, L.~Colzani, and G.~Travaglini}, {\em Average decay of
  {Fourier} transforms and integer points in polyhedra}, Arkiv f{\"o}r
  Matematik, 35 (1997), pp.~253--275.

\bibitem{chen2017implicit}
{\sc C.~Chen and R.~Tsai}, {\em Implicit boundary integral methods for the
  helmholtz equation in exterior domains}, Res. Math. Sci., 4 (2017), p.~19.

\bibitem{chen2014panorama}
{\sc W.~Chen, A.~Srivastav, G.~Travaglini, et~al.}, {\em A panorama of
  discrepancy theory}, vol.~2107, Springer, 2014.

\bibitem{duke1990representation}
{\sc W.~Duke and R.~Schulze-Pillot}, {\em Representation of integers by
  positive ternary quadratic forms and equidistribution of lattice points on
  ellipsoids}, Invent. Math., 99 (1990), pp.~49--57.

\bibitem{engquist2005discretization}
{\sc B.~Engquist, A.-K. Tornberg, and R.~Tsai}, {\em Discretization of {Dirac}
  delta functions in level set methods}, J. Comput. Phys., 207 (2005),
  pp.~28--51.

\bibitem{hardy1922some}
{\sc G.~H. Hardy and J.~E. Littlewood}, {\em Some problems of diophantine
  approximation: The lattice-points of a right-angled triangle}, Proc. London
  Math. Soc., 2 (1922), pp.~15--36.

\bibitem{hardy1922some2}
\leavevmode\vrule height 2pt depth -1.6pt width 23pt, {\em Some problems of
  {Diophantine} approximation: The lattice-points of a right-angled
  triangle.(second memoir.)}, in Abhandlungen aus dem Mathematischen Seminar
  der Universit{\"a}t Hamburg, vol.~1, Springer, 1922, pp.~211--248.

\bibitem{heath1999lattice}
{\sc D.~Heath-Brown}, {\em Lattice points in the sphere}, in Number Theory in
  Progress, K.~Gy\"ory, H.~Iwaniec, and J.~Urbanowicz, eds., De Gruyter,
  Berlin, 1999, pp.~883--892.

\bibitem{hlawka1962angenaherten}
{\sc E.~Hlawka}, {\em Zur angen{\"a}herten berechnung mehrfacher integrale},
  Monatshefte f{\"u}r Mathematik, 66 (1962), pp.~140--151.

\bibitem{huxley1996area}
{\sc M.~N. Huxley}, {\em Area, Lattice Points, and Exponential Sums}, Clarendon
  Press, 1996.

\bibitem{huxley2003exponential}
{\sc M.~N. Huxley}, {\em Exponential sums and lattice points {III}}, Proc.
  London Math. Soc., 87 (2003), pp.~591--609.

\bibitem{huxley2005exponential}
{\sc M.~N. Huxley}, {\em Exponential sums and the {Riemann} zeta function {V}},
  Proc. London Math. Soc., 90 (2005), pp.~1--41.

\bibitem{huxley2014fourth}
\leavevmode\vrule height 2pt depth -1.6pt width 23pt, {\em A fourth power
  discrepancy mean}, Monatshefte f{\"u}r Mathematik, 173 (2014), pp.~231--238.

\bibitem{izzo2023high}
{\sc F.~Izzo, O.~Runborg, and R.~Tsai}, {\em High-order corrected trapezoidal
  rules for a class of singular integrals}, Adv. Comput. Math., 49 (2023),
  p.~60.

\bibitem{izzo2022corrected}
{\sc F.~Izzo, Y.~Zhong, O.~Runborg, and R.~Tsai}, {\em Corrected {Trapezoidal}
  rule-ibim for linearized {Poisson-Boltzmann} equation}, arXiv:2210.03699,
  (2022).

\bibitem{kendall1948number}
{\sc D.~G. Kendall}, {\em On the number of lattice points inside a random
  oval}, Q. J. Math.,  (1948), pp.~1--26.

\bibitem{kublik2013implicit}
{\sc C.~Kublik, N.~M. Tanushev, and R.~Tsai}, {\em An implicit interface
  boundary integral method for {Poisson’s} equation on arbitrary domains}, J.
  Comput. Phys., 247 (2013), pp.~279--311.

\bibitem{kublik2016integration}
{\sc C.~Kublik and R.~Tsai}, {\em Integration over curves and surfaces defined
  by the closest point mapping}, Research in the mathematical sciences, 3
  (2016).

\bibitem{kublik2018extrapolative}
\leavevmode\vrule height 2pt depth -1.6pt width 23pt, {\em An extrapolative
  approach to integration over hypersurfaces in the level set framework},
  Mathematics of Computation,  (2018).

\bibitem{kuipers2012uniform}
{\sc L.~Kuipers and H.~Niederreiter}, {\em Uniform distribution of sequences},
  Courier Corporation, 2012.

\bibitem{macdonald1971polynomials}
{\sc I.~G. Macdonald}, {\em Polynomials associated with finite cell-complexes},
  J. London Math. Soc., 2 (1971), pp.~181--192.

\bibitem{matousek1999geometric}
{\sc J.~Matousek}, {\em Geometric Discrepancy: {An} Illustrated Guide},
  vol.~18, Springer Science \& Business Media, 1999.

\bibitem{sloan1994lattice}
{\sc I.~H. Sloan and S.~Joe}, {\em Lattice Methods for Multiple Integration},
  Oxford University Press, 1994.

\bibitem{stein1993harmonic}
{\sc E.~M. Stein and T.~S. Murphy}, {\em Harmonic Analysis: Real-variable
  Methods, Orthogonality, and Oscillatory Integrals}, Princeton University
  Press, 1993.

\bibitem{tarnopolska1979number}
{\sc M.~Tarnopolska-Weiss}, {\em On the number of lattice points in a compact
  n-dimensional polyhedron}, Proc. Am. Math. Soc., 74 (1979), pp.~124--127.

\bibitem{tornberg2004numerical}
{\sc A.-K. Tornberg and B.~Engquist}, {\em Numerical approximations of singular
  source terms in differential equations}, J. Comput. Phys., 200 (2004),
  pp.~462--488.

\bibitem{zaremba1966good}
{\sc S.~K. Zaremba}, {\em Good lattice points, discrepancy, and numerical
  integration}, Annali di Matematica Pura ed Applicata, 73 (1966),
  pp.~293--317.

\bibitem{zhong2018implicit}
{\sc Y.~Zhong, K.~Ren, and R.~Tsai}, {\em An implicit boundary integral method
  for computing electric potential of macromolecules in solvent}, J. Comput.
  Phys., 359 (2018), pp.~199--215.

\end{thebibliography}

\end{document}